\documentclass[final,onefignum,onetabnum]{siamart250211}

\usepackage{lipsum,amsmath,amssymb}
\usepackage{amsfonts}
\usepackage{graphicx,subfigure}
\usepackage{epstopdf}
\usepackage{algorithmic}
\ifpdf
\DeclareGraphicsExtensions{.eps,.pdf,.png,.jpg}
\else
\DeclareGraphicsExtensions{.eps}
\fi

\allowdisplaybreaks[4]

\newcommand{\inner}[1]{\left\langle #1 \right\rangle}
\newcommand{\norm}[1]{\left\Vert #1\right\Vert}
\newcommand{\bb}[1]{\mathbb{#1}}

\newcommand{\conv}[0]{\mathrm{conv}\,}%convex hull
\newcommand{\X}{{ \ca{X} }}

\newcommand{\ca}[1]{\mathcal{#1}}

\newcommand{\M}[0]{\mathcal{M}}

\newcommand{\tp}{^\top}

\newcommand{\A}{\ca{A}}

\newcommand{\mk}{{m_{k} }}

\newcommand{\mkp}{{m_{k+1} }}

\newcommand{\xk}{{x_{k} }}
\newcommand{\yk}{{y_{k} }}

\newcommand{\wk}{{w_{k} }}
\newcommand{\zk}{{z_{k} }}
\newcommand{\xkp}{{x_{k+1} }}

\newcommand{\wkp}{{w_{k+1} }}
\newcommand{\zkp}{{z_{k+1} }}

\newcommand{\ykp}{{y_{k+1} }}

\newcommand{\D}{\ca{D}}

\newcommand{\Rn}{\mathbb{R}^n}
\newcommand{\Rp}{\mathbb{R}^p}

\newcommand{\Omegajoint}{{\Omega}}

\newcommand{\revise}[1]{\textcolor{blue}{#1}}

\usepackage{tikz}
\usetikzlibrary{shapes.geometric, arrows.meta, positioning, calc, shadows}
\newtheorem{theo}{Theorem}[section]
\newtheorem{lem}[theo]{Lemma}
\newtheorem{prop}[theo]{Proposition}

\newtheorem{defin}[theo]{Definition}
\newtheorem{rmk}[theo]{Remark}
\newtheorem{assumpt}[theo]{Assumption}
\headers{A Hybrid Subgradient Method for Nonconvex Bilevel Optimization}{Nachuan Xiao, Xiaoyin Hu, Xin Liu, and Kim-Chuan Toh}

\title{A Hybrid Subgradient Method for Nonsmooth Nonconvex Bilevel Optimization}

\author{
	Nachuan Xiao\thanks{School of Data Science, The Chinese University of Hong Kong, Shenzhen, China
		(\email{xncxy@cuhk.edu.cn}).}
	\and 
	Xiaoyin Hu\thanks{School of Mathematical Sciences, Shenzhen University, Shenzhen 518060, Guangdong, China.  (\email{hxy@amss.ac.cn}).}
	\and 
	Xin Liu\thanks{State Key Laboratory of Scientific and Engineering Computing, Academy of Mathematics and Systems Science, Chinese Academy of Sciences, and University of Chinese Academy of Sciences, China (\email{liuxin@lsec.cc.ac.cn}).}
	\and 
	Kim-Chuan Toh\thanks{Department of Mathematics, National University of Singapore, Singapore 119076. (\email{mattohkc@nus.edu.sg}).}
}

\usepackage{amsopn}

\ifpdf
\hypersetup{
	pdftitle={A Hybrid Subgradient Method for Nonconvex Bilevel Optimization},
	pdfauthor={Nachuan Xiao, Xiaoyin Hu, Xin Liu, and Kim-Chuan Toh}
}
\fi

\begin{document}
	
	\maketitle
	
	% REQUIRED
	\begin{abstract}
		In this paper, we focus on the nonconvex-nonconvex bilevel optimization problem (BLO), where both upper-level and lower-level objectives are nonconvex, with the upper-level problem potentially being nonsmooth. We develop a two-timescale momentum-accelerated subgradient method (TMG) that employs two-timescale stepsizes, and establish its local convergence when initialized within a sufficiently small neighborhood of the feasible region. To develop a globally convergent algorithm for (BLO), we introduce a feasibility restoration scheme (FRG) that drives iterates toward the feasible region. Both (TMG) and (FRG) only require the first-order information of the upper-level and lower-level objective functions, ensuring efficient computations in practice. We then develop a novel hybrid method that alternates between (TMG) and (FRG) and adaptively estimates its hyperparameters. Under mild conditions, we establish the global convergence properties of our proposed algorithm. Preliminary numerical experiments demonstrate the high efficiency and promising potential of our proposed algorithm. 
	\end{abstract}
	
	% REQUIRED
	\begin{keywords}
		bilevel optimization, nonsmooth optimization, subgradient method, hybrid method, penalty function
	\end{keywords}
	
	% REQUIRED
	\begin{MSCcodes}
		65K05, 90C06
	\end{MSCcodes}
	
	\section{Introduction}
	
	In this paper, we consider the following nonconvex-nonconvex bilevel optimization problem:
	\begin{equation}
		\label{Prob_BLO}
		\tag{BLO}
		\begin{aligned}
			\min_{x \in \Rn, ~y \in \Rp} \quad & f(x, y), \\
			\text{s. t.} \quad & y \in \ca{S}(x).
		\end{aligned}
	\end{equation}
	Here $f: \Rn \times \Rp \to \bb{R}$ is the objective function of the upper-level (UL) subproblem, which is assumed to be locally Lipschitz continuous but possibly nonsmooth and nonconvex over $\Rn \times \Rp$ throughout this paper. Moreover, $\ca{S}(x)$ represents the set of stationary points of the following lower-level (LL) subproblem:
	\begin{equation}
		\label{Eq_BLO_LL}
		\min_{y \in \Rp}\quad g(x, y),
	\end{equation}
	where $g: \Rn \times \Rp \to \bb{R}$ is \revise{four} times continuously differentiable and generally nonconvex over $\Rn \times \Rp$.

	The optimization problem of the form \eqref{Prob_BLO} has found widespread applications in deep learning, particularly in reinforcement learning \cite{konda1999actor}, hyperparameter optimization \cite{tappen2008logistic,hutter2011sequential,franceschi2017forward,lorraine2020optimizing}, and meta-learning \cite{finn2017model,samuel2009learning}. The nested structure of bilevel optimization problems presents significant challenges to the development of efficient optimization algorithms, especially when $\ca{S}(x)$ represents the global minimizers of the LL subproblem \eqref{Eq_BLO_LL}. Even in a simplified case where both UL and LL subproblems involve linear functions with linear constraints, finding the global minimizers of \eqref{Prob_BLO} is NP-hard \cite{buchheim2023bilevel}. Furthermore, as shown in \cite{bolte2024geometric}, minimizing any lower semicontinuous function is equivalent to solving a specific bilevel optimization problem of the form \eqref{Prob_BLO} with differentiable UL and LL subproblems. 
	
	Consequently, a great number of existing works on bilevel optimization focus on cases where $\ca{S}(x)$ coincides with the first-order stationary points of the LL subproblem \eqref{Eq_BLO_LL}, which yields the following equality-constrained optimization problem (ECP), 
	\begin{equation}
		\label{Prob_Con}
		\tag{ECP}
		\begin{aligned}
			\min_{x \in \Rn, ~y \in \Rp} \quad & f(x, y), \\
			\text{s. t.} \quad & \nabla_y g(x, y) = 0.  
		\end{aligned}
	\end{equation}
	Here we denote $\M := \{(x, y) \in \Rn \times \Rp: \nabla_y g(x, y) = 0\}$ as the feasible region of \eqref{Prob_Con}. 
	
	For solving the bilevel optimization problem, most existing works establish global convergence for their developed optimization approaches under certain regularity conditions. 
	A widely adopted regularity condition for bilevel optimization  \eqref{Prob_BLO} requires the strong convexity of $g(x, y)$ with respect to its $y$-variable for any fixed $x \in \Rn$. Under such regularity conditions, the bilevel optimization problem \eqref{Prob_BLO} is equivalent to the constrained optimization problem \eqref{Prob_Con}. Moreover, the implicit function theorem guarantees that $\ca{S}(x)$ is singleton-valued and continuously differentiable over $\mathbb{R}^n$. Based on that regularity condition for LL subproblems, some existing optimization approaches for \eqref{Prob_Con} \cite{domke2012generic,pedregosa2016hyperparameter,ghadimi2018approximation,grazzi2020iteration,ji2021bilevel} employ a {\it double-loop} scheme by introducing an inner loop in each iteration to obtain an approximated estimation for $\ca{S}(x)$ (i.e., an approximated projection onto $\M$). However, these approaches may suffer from poor performance as one has to take multiple steps in the inner loop to solve the LL subproblems to a desired accuracy \cite{khanduri2021near}. 
	
	To address these limitations, several {\it single-loop} approaches \cite{chen2021single,hong2020two,khanduri2021near} are proposed by updating the $x$- and $y$-variables simultaneously, hence avoiding inner loops for finding approximate solutions of the LL subproblems. Additionally, \cite{hu2022improved} develops an exact penalty function for \eqref{Prob_Con}, enabling the direct application of various unconstrained optimization techniques. However, these approaches are not  Hessian-free, in the sense that they rely on the exact or approximated computation of $\nabla_{yy}^2 g(x,y)$ and its inverse, which makes the computations intractable under large-scale settings. 
	
	Beyond the strong convexity of LL subproblems, several existing efficient approaches are developed \cite{liu2022bome,huang2023momentum,shen2023penalty,xiao2024alternating,lu2024slm} with slightly weaker regularity conditions, such as the Polyak-Łojasiewicz (PL) condition or local convexity of the LL subproblems with respect to the $y$-variable. For example, \cite{liu2021towards} proposes a double-loop algorithm that differentiates the multiple-step gradient descent steps of the LL subproblem. \cite{arbel2022non} proposes an implicit differentiation approach that solves the bilevel optimization problem \eqref{Prob_BLO} by differentiating the solution map of its LL subproblem.

	While general nonconvex-strongly-convex bilevel optimization has been extensively studied in the literature, the optimization approaches for nonconvex-nonconvex bilevel optimization remain limited. Some existing works employ double-loop schemes to solve general nonconvex-nonconvex bilevel optimization problems, with certain regularity conditions. For example, \cite[Assumption 3.1(5)]{liu2021towards} assumes that all the stationary points of $f$ lie in $\M$. Moreover, \cite{liu2022bome,shen2023penalty,lu2024slm,xiao2024alternating} impose the PL condition for the LL subproblems.  As demonstrated in \cite{henrion2011calmness,bolte2024geometric}, these regularity conditions are difficult to verify a priori and frequently fail to be satisfied in common bilevel optimization settings.
	
	Furthermore, for solving general nonconvex-nonconvex bilevel optimization problems, the value-function approach has been developed to reformulate the bilevel optimization problem into another corresponding single-level constrained optimization problem \cite{liu2021value,ye2023difference,liu2024moreau,lu2024first}. Specifically, when applied to the bilevel optimization problem \eqref{Prob_BLO}, the formulation can be transformed into the following constrained optimization problem  \cite{liu2024moreau},
	\begin{equation}
		\label{Prob_Con_value}
		\begin{aligned}
			\min_{x \in \Rn, ~y \in \Rp} \quad & f(x, y), \\
			\text{s. t.} \quad & g(x, y) \leq v_{\gamma}(x,y).  
		\end{aligned}
	\end{equation}
	Here $v_{\gamma}(x,y) := \min_{z \in \Rp} g(x, z) + \frac{1}{2\gamma} \norm{z-y}^2$. Then \cite{liu2024moreau} considers the following penalty function for \eqref{Prob_Con_value}, 
	\begin{equation}
		\label{Prob_Con_value_pen}
		\chi_{\beta}(x, y) := f(x, y) + \beta (g(x, y) - v_{\gamma}(x,y)).
	\end{equation}
	Reformulating \eqref{Prob_BLO} into \eqref{Prob_Con_value_pen} avoids expensive Hessian-vector and Jacobian-vector products, making these approaches computationally efficient. However, it is important to note that the penalizing the constraint $g(x, y) - v_{\gamma}(x, y) = 0$ in \eqref{Prob_Con_value_pen} does not necessarily imply $\nabla_y g(x, y) = 0$. Consequently, $\chi_{\beta}(x, y)$ may introduce infeasible stationary points into the constrained optimization problem \eqref{Prob_Con}, and hence to the bilevel optimization problem \eqref{Prob_BLO}. As a result, the gradient-based method developed by these value function approaches (e.g., \cite[Algorithm 2]{liu2024moreau}) cannot guarantee the feasibility of the obtained solutions without any additional regularity condition. 
	
	Moreover, the theoretical analysis in value function approaches \cite{liu2024moreau} relies on estimating the decrease of $\chi_{\beta}(\xk, \yk)$ over the iterates, which is only available for the case where the UL objective function $f$ is weakly convex. This limitation restricts the applicability of these approaches to bilevel optimization problems with Clarke-regular objective functions \cite{clarke1990optimization}. 
	For general nonsmooth nonconvex optimization problems, recent works \cite{davis2020stochastic,ruszczynski2020convergence,bianchi2021closed,bolte2021conservative,castera2021inertial,bolte2022long,le2023nonsmooth,xiao2023adam} employ the ordinary differential equation (ODE) approach \cite{benaim2005stochastic,borkar2009stochastic,duchi2018stochastic,davis2020stochastic} to establish convergence properties of subgradient methods in the unconstrained minimization of \emph{path-differentiable} \cite{bolte2021conservative} functions. As demonstrated in \cite{davis2020stochastic,bolte2021conservative,castera2021inertial}, the class of path-differentiable functions encompasses a wide range of real-world applications, including neural networks with nonsmooth activation functions. However, developing subgradient methods for nonsmooth nonconvex constrained optimization is less explored. \cite{xiao2024developing} develops a general framework embedding proximal subgradient methods into Lagrangian-based methods for nonsmooth nonconvex constrained optimization. Their framework, however, relies on the error-bound condition to guarantee the feasibility of the solutions, which is usually absent in \eqref{Prob_Con}. This leads us to the following question:
	\begin{quote}
		\it 
		Can we develop efficient subgradient methods for nonsmooth nonconvex bilevel optimization problems under mild regularity conditions, with guaranteed convergence to feasible stationary points? 
	\end{quote}

	\subsection{Motivations}
	\subsubsection{Regularity condition}
	In this paper, we focus on the case where $\ca{S}(x)$ is the set of first-order stationary points of the LL subproblem \eqref{Eq_BLO_LL}. Hence, the bilevel optimization problem \eqref{Prob_BLO} is equivalent to the equality-constrained optimization problem \eqref{Prob_Con}.  To ensure well-posedness of \eqref{Prob_Con}, we impose the following regularity condition on \eqref{Prob_Con}: 
	\begin{assumpt}
		\label{Assumption_joint_nondegeneracy}
		% \begin{enumerate}
			The matrix $
			J_g(x, y) :=\left[\begin{smallmatrix}
				\nabla_{xy}^2 g(x, y)\\
				\nabla_{yy}^2 g(x, y)
			\end{smallmatrix}
			\right] \in \bb{R}^{(n+p)\times p}$ has full column rank for any $(x, y) \in \M$. 
			% \end{enumerate}
	\end{assumpt}
	Assumption \ref{Assumption_joint_nondegeneracy} requires that the constraint of \eqref{Prob_Con} satisfies the linear-independence constraint qualification (LICQ) \cite{jorge2006numerical}. In practice, this condition is often satisfied when the attained LL stationary point is nondegenerate. In particular, if $\nabla^2_{yy}g(x,y)$ is nonsingular at $(x,y)\in \M$, then $J_g(x,y)$ automatically has full column rank. This covers the cases where the LL subproblems are strongly convex with respect to the $y$-variable. Moreover, the full column rank property may still hold even when $\nabla^2_{yy}g(x,y)$ is singular, provided that the coupling block $\nabla^2_{xy}g(x,y)$ and $\nabla^2_{yy}g(x,y)$ do not share a non-trivial null space. For instance, if $g(x,y)=\phi(y)-\langle By,x\rangle$, then $J_g(x,y)=
	\left[\begin{smallmatrix}
		-B^\top, \;
		\nabla^2\phi(y)
	\end{smallmatrix}\right]^\top$ can have full column rank whenever the intersection of the nullspace of $B$ and the nullspace of $\nabla^2\phi(y)$ is $\{0\}$.

	Furthermore, from the Sard-type properties of semi-algebraic mappings \cite{drusvyatskiy2013semi,drusvyatskiy2016generic}, we can conclude that when $g$ is semi-algebraic (or more generally, defined in some $o$-minimal structures \cite{van1996geometric}), Assumption \ref{Assumption_joint_nondegeneracy} holds for generic tilt perturbations to the LL subproblems of \eqref{Prob_BLO}. More precisely, as illustrated in \cite[Theorem 5.2]{drusvyatskiy2016generic}, the matrix $J_g(x, y)$
	has full column rank over $\{(x, y) \in \Rn \times \Rp: \nabla_y g(x, y) + v = 0\}$ for almost every $v \in \Rp$. Therefore,  Assumption \ref{Assumption_joint_nondegeneracy} is generally satisfied and thus mild in practical scenarios.

	As LICQ holds for \eqref{Prob_Con} under Assumption \ref{Assumption_joint_nondegeneracy}, we can identify $\M$ as an embedded manifold in $\Rn \times \Rp$, as shown in \cite{absil2008optimization,boumal2023intromanifolds}. Although various Riemannian optimization approaches are developed in recent years \cite{absil2008optimization,boumal2023intromanifolds}, computing the geometric structures of the manifold $\M$ is usually computationally expensive. For example, computing the retraction of $\M$ can be regarded as computing a projection from the tangent space to $\M$, which requires the computation of exact first-order stationary points of the LL subproblems of \eqref{Prob_BLO}.  Moreover, computing the Riemannian subgradient of \eqref{Prob_Con} is equivalent to computing the projection to the tangent spaces of $\M$ at $(x, y)$, which requires the computation of the inverse of $\nabla_{yy}^2 g(x, y)$ in each iteration.

	To develop efficient optimization approaches while utilizing the fact that $\M$ is an embedded manifold, \cite{hu2022improved} proposes a globally exact penalty function for nonconvex-strongly-convex bilevel optimization problems based on the constrained dissolving approach \cite{xiao2023dissolving}. The authors in \cite{hu2022improved} illustrate that the nonconvex strongly-convex bilevel optimization problem is equivalent to the unconstrained minimization of the proposed penalty function. As a result, various existing unconstrained optimization approaches can be directly employed to solve \eqref{Prob_BLO} through the proposed penalty function. 
	However, their analysis relies on a restrictive regularity condition that assumes the strong convexity of $g$ with respect to $y$-variable, hence is not applicable for solving \eqref{Prob_Con} under Assumption \ref{Assumption_joint_nondegeneracy}. Consequently, the design and analysis of optimization methods for \eqref{Prob_Con} that effectively exploit the underlying manifold structures of the feasible region $\M$ under Assumption \ref{Assumption_joint_nondegeneracy} remains an open problem.

	\subsubsection{Two-timescale scheme}
	
	To develop an efficient subgradient method for solving \eqref{Prob_Con},
	we introduce $p(x, y) :=\frac{1}{2} \norm{\nabla_y g(x, y)}^2$ as a quadratic penalty term for the constraint in \eqref{Prob_Con} and consider the following two-timescale momentum-accelerated subgradient method (TMG), 
	\begin{equation}
		\label{Eq_Subroutine_manifold}
		\tag{TMG}
		\left\{
		\begin{aligned}
			&(d_{x, k}, d_{y, k}) \in \D_f(\xk, \yk), \quad(u_{x, k}, u_{y, k}) = \nabla p(\xk, \yk),
			\\
			&(m_{x, k+1}, m_{y, k+1}) = \alpha (m_{x, k}, m_{y, k}) + (1-\alpha )(d_{x, k}, d_{y, k}),\\
			&\xkp = \xk - \eta_k m_{x, k+1} -  \theta_k (u_{x,k}+ e_{x, k}),\\
			&\ykp = \yk - \eta_k m_{y, k+1} -  \theta_k (u_{y,k} + e_{y, k}).
		\end{aligned}
		\right.
	\end{equation}
	Here the update scheme \eqref{Eq_Subroutine_manifold} takes two-timescale stepsizes, in the sense that $\lim_{k\to +\infty} \theta_k = 0$ and $\lim_{k\to +\infty} \frac{\eta_k}{\theta_k} = 0$.
	Moreover, $\D_f$ refers to the conservative field of the UL objective function $f$ \cite{bolte2021conservative}, which generalizes the Clarke subdifferential  \cite{clarke1990optimization}. Additionally, the sequence $\{(m_{x, k}, m_{y, k})\}$ corresponds to the heavy-ball momentum terms and we initialize $(m_{x, 0}, m_{y, 0}) = (d_{x, 0}, d_{y, 0})$, while the momentum parameter $\alpha \in [0, 1)$ is a fixed constant. 
	The sequences $\{e_{x, k}\}$ and $\{e_{y, k}\}$ denote the errors in the evaluation of $\nabla p(\xk, \yk)$. As shown later in Remark \ref{Rmk_interpolation_nabla_p}, 
	these error terms allow for approximate gradient evaluations using two-point interpolation, avoiding the need for explicit computation of Hessians of $g$. Consequently, \eqref{Eq_Subroutine_manifold} only requires first-order information of $f$ and $g$, making it computationally efficient for large-scale problems.
	
	Analyzing the convergence of \eqref{Eq_Subroutine_manifold} presents significant challenges due to the employed two-timescale stepsizes. Existing ODE approaches \cite{benaim2005stochastic,borkar2009stochastic,duchi2018stochastic,davis2020stochastic} are developed for subgradient methods that take single-timescale stepsizes, and there are limited results on subgradient methods with two-timescale stepsizes. 
	To the best of our knowledge, there is no existing framework discussing the convergence properties of  \eqref{Eq_Subroutine_manifold} under nonsmooth settings.

	\subsection{Contributions}
	In this paper, we focus on developing an efficient subgradient method for \eqref{Prob_BLO} with global convergence guarantees. 
	To overcome the challenges of establishing convergence properties of two-timescale stepsizes under Assumption \ref{Assumption_joint_nondegeneracy}, we introduce the following auxiliary mapping $\A$ \revise{that is well-defined on a neighborhood of $\M$},
	\begin{equation}
		\label{Eq_defin_cdmapping}
		\A(x, y) = (x, y) - (J_g(x, y)^{\dagger})\tp \nabla_y g(x, y),
	\end{equation}
	where $J_g(x, y)^{\dagger} = (J_g(x, y)\tp J_g(x, y))^{-1} J_g(x, y)\tp$ denotes the pseudo-inverse of $J_g(x, y)$ \cite{golub2013matrix}. Based on  $\A$, we construct the auxiliary sequence ${(\wk, \zk)}:= {\A(\xk, \yk)}$ and show that it shares the same cluster points as the original sequence ${(\xk, \yk)}$.
	
	More importantly, we demonstrate that the auxiliary sequence ${(\wk, \zk)}$ follows a single-timescale update scheme, which can be interpreted as an inexact subgradient descent method for the auxiliary function
	\begin{equation}
		\label{Eq_defin_cdf}
		h_{\beta}(x, y) := f(\A(x, y)) + \frac{\beta}{2} \norm{\nabla_y g(x, y)}^2,
	\end{equation}
	where $\beta > 0$ is a penalty parameter. We establish the equivalence between \eqref{Prob_Con} and \eqref{Eq_defin_cdf} in the aspect of their first-order stationary points, and show the local convergence properties of \eqref{Eq_Subroutine_manifold}. Specifically, we prove that for  any initial point $(x_0, y_0)$ sufficiently close to $\M$, and under appropriately selected stepsize sequences $\{\eta_k\}$ and $\{\theta_k\}$, every cluster point of the iterates $\{(\xk, \yk)\}$ generated by the update scheme \eqref{Eq_Subroutine_manifold} is a first-order stationary point of \eqref{Prob_Con}.

	Furthermore, to design a globally convergent subgradient method, we consider the following feasibility restoration scheme (FRG):
	\begin{equation}
		\label{Eq_Subroutine_gy}
		\tag{FRG}
		\left\{
		\begin{aligned}
			&(d_{x, k}, d_{y, k}) \in \D_f(\xk, \yk),  %%\quad (u_{x, k}, u_{y, k}) \in \nabla p(\xk, \yk),
			\\
			&(m_{x, k+1}, m_{y, k+1}) = \alpha (m_{x, k}, m_{y, k}) + (1-\alpha )(d_{x, k}, d_{y, k}),\\
			&\xkp = \xk - \eta_k m_{x, k+1},\\
			&\ykp = \yk -\eta_k m_{y, k+1} - \theta_k \nabla_y g (\xk, \yk). 
		\end{aligned}
		\right.
	\end{equation}
	Here $\{\eta_k\}$ and $\{\theta_k\}$ are two-timescale stepsizes, and we initialize $(m_{x, 0}, m_{y, 0}) = (d_{x, 0}, d_{y, 0})$. Moreover,  \eqref{Eq_Subroutine_gy} only requires the computation of the first-order information of both $f$ and $g$, hence avoiding the expensive computations of the second-order derivatives of $g$. Under mild conditions, we prove that the sequence $\{(\xk, \yk)\}$ satisfies $\liminf_{k\to +\infty} \norm{\nabla_y g(\xk, \yk)} = 0$. This illustrates that the update scheme \eqref{Eq_Subroutine_gy} drives the sequence $\{(\xk, \yk)\}$ towards the feasible region $\M$.

	Building on the update schemes \eqref{Eq_Subroutine_manifold} and \eqref{Eq_Subroutine_gy}, we propose the Two-phase Hybrid Subgradient Descent (TPHSD) method in Algorithm \ref{Alg_TPHSD}, which adaptively alternates between optimization and feasibility restoration based on the distance between the iterates and the feasible region $\mathcal{M}$. When the distance is no greater than the tolerance parameter, Algorithm \ref{Alg_TPHSD} employs the momentum-accelerated scheme \eqref{Eq_Subroutine_manifold} to drive near-feasible iterates toward a first-order stationary point. Conversely, when the distance from the current iterate to the feasible region $\M$ exceeds the tolerance parameter, Algorithm \ref{Alg_TPHSD} employs the feasibility restoration scheme \eqref{Eq_Subroutine_gy} to drive the iterate close to $\M$. The transition from the momentum-accelerated scheme \eqref{Eq_Subroutine_manifold} to the feasibility restoration scheme \eqref{Eq_Subroutine_gy} triggers an adaptive reduction of the tolerance parameter, which prevents oscillation and ensures global convergence of our proposed algorithm.
	
	% Based on the update schemes \eqref{Eq_Subroutine_manifold} and \eqref{Eq_Subroutine_gy}, we propose a hybrid method in Algorithm \ref{Alg_TPHSD}, which alternates between  \eqref{Eq_Subroutine_gy} and \eqref{Eq_Subroutine_manifold} based on certain criteria. Specifically, the feasibility restoration scheme \eqref{Eq_Subroutine_gy} is employed when the iterates are far away from the feasible region $\M$. On the other hand, once feasibility is sufficiently restored, the momentum-accelerated subgradient method \eqref{Eq_Subroutine_manifold} is employed to drive the iterates to a first-order stationary point
	% of \eqref{Prob_Con}.

	Under mild conditions, we establish the global convergence of Algorithm \ref{Alg_TPHSD}. Specifically, we prove that for any initial point, every cluster point of the iterates $\{(\xk, \yk)\}$ generated by Algorithm \ref{Alg_TPHSD} is a first-order stationary point of \eqref{Prob_Con}. We also conduct preliminary numerical experiments to verify the effectiveness of our proposed algorithm.

	\subsection{Organization}
	The outline of the rest of this paper is as follows. In Section 2, we present the notations and preliminary concepts that are necessary for the proofs in this paper. We present the details of our proposed Algorithm \ref{Alg_TPHSD} and establish its global convergence properties in Section 3. In Section 4, we perform illustrative numerical examples to show the efficiency of Algorithm \ref{Alg_TPHSD}. We conclude the paper in the last section.

	\section{Preliminary}
	
	\subsection{Basic notations}
	\label{Subsection_basic_notation}
	Let $\inner{\cdot, \cdot}$ be the standard inner product and $\norm{\cdot}$ be the $\ell_2$-norm of a vector or an operator
	in a finite-dimensional Euclidean space. 
	We use $\bb{B}_{\delta}(x,y):= \{ (\tilde{x}, \tilde{y}) \in \Rn \times \Rp: \norm{\tilde{x} - x}^2 + \norm{\tilde{y} - y}^2 \leq \delta^2 \}$ 
	to denote the ball
	centered at $(x,y)$ with radius $\delta$. Moreover, for a given set $\X$ and point $x$, $\mathrm{dist}(x, \ca{X})$ denotes the distance between $x$ and  $\ca{X}$, i.e. $\mathrm{dist}(x, \ca{X}) := \mathop{\inf}_{y \in \ca{X}} ~\norm{x-y}$.

	For any differentiable function $g: \mathbb{R}^n \times \mathbb{R}^p \to \mathbb{R}$,  $\nabla_x g(x, y)$ and $\nabla_y g(x, y)$ denote the partial derivatives of $g$ with respect to $x$ and $y$, respectively. Furthermore, we define $\nabla_{xy}^2 g(x, y)$ and $\nabla_{yy}^2 g(x, y)$ as the partial Jacobians of $\nabla_y g(x, y)$ with respect to the $x$- and $y$-variables, respectively. More precisely, 
	\begin{footnotesize}
		\begin{equation*}
			\nabla_{xy}^2 g(x, y) := \left[  \begin{matrix}
				\frac{\partial^2 g(x, y) }{\partial x_1\partial y_1} & \cdots & \frac{\partial^2 g(x, y) }{\partial x_1\partial y_p} \\
				\vdots & \ddots & \vdots \\
				\frac{\partial^2 g(x, y) }{\partial x_n\partial y_1} & \cdots & \frac{\partial^2 g(x, y) }{\partial x_n\partial y_p} \\
			\end{matrix}\right], 
			\quad 
			\nabla_{yy}^2 g(x, y) := \left[  \begin{matrix}
				\frac{\partial^2 g(x, y) }{\partial y_1\partial y_1} & \cdots & \frac{\partial^2 g(x, y) }{\partial y_1\partial y_p} \\
				\vdots & \ddots & \vdots \\
				\frac{\partial^2 g(x, y) }{\partial y_p\partial y_1} & \cdots & \frac{\partial^2 g(x, y) }{\partial y_p\partial y_p} \\
			\end{matrix}\right]. 
		\end{equation*} 
	\end{footnotesize}
	Additionally, $\nabla_{yx}^2 g(x, y)$ is defined as the transpose of $\nabla_{xy}^2 g(x, y)$, i.e., $\nabla_{yx}^2 g(x, y) = \left(\nabla_{xy}^2 g(x, y)\right)\tp$. 
	
	Furthermore, a set-valued mapping $\ca{D}: \bb{R}^m \rightrightarrows \bb{R}^s$ is a mapping from $\bb{R}^m$ to a collection of subsets of $\bb{R}^s$. $\D$ is said to be graph-closed if its graph, defined as
	\begin{equation*}
		\mathrm{graph}(\D) := \left\{ (w,z) \in \bb{R}^m \times \bb{R}^s: w \in \bb{R}^m, z \in \D(w) \right\},
	\end{equation*}
	is a closed subset of $\bb{R}^m\times \bb{R}^s$.
	For any $\delta \geq 0$, we denote $\ca{D}^{\delta}:\bb{R}^m \rightrightarrows \bb{R}^s$ as
	\begin{equation*}
		\D^{\delta}(z) = \bigcup_{w \in \bb{B}_{\delta}(z)} (\D(w) + \bb{B}_{\delta}(0)). 
	\end{equation*}
	
	An absolutely continuous curve is a continuous mapping $\gamma: \bb{R} \to \Rn\times \Rp $, such that its derivative $\gamma'$ exists almost everywhere in $\bb{R}$. Moreover, the difference $\gamma(t) - \gamma(0)$ equals the Lebesgue integral of $\gamma'$ between $0$ and $t$ for all $t \in \bb{R}_+$, i.e., $\gamma(t) = \gamma(0) + \int_{0}^t \gamma'(\tau) \mathrm{d} \tau$ for any  $t \in \bb{R}_+$. 
	
	For any positive sequence $\{\eta_k\}$, we define 
	$\lambda(0) := 0$, $\lambda(i) := \sum_{k = 0}^{i-1} \eta_k$ for $i\geq 1$, and $\Lambda(t) := \sup  \{k \geq 0: t\geq \lambda(k)\} $. More explicitly, $\Lambda(t) = p$ if $\lambda(p) \leq t < \lambda(p+1)$ for any $p \geq 0$. In particular, $\Lambda(\lambda(p)) = p$. 
	Similarly, for any positive sequence $\{\theta_k\}$, we denote $\lambda_{\theta}(0) := 0$, $\lambda_{\theta}(i) := \sum_{k = 0}^{i-1} \theta_k$ for $i\geq 1$, and $\Lambda_{\theta}(t) := \sup  \{k \geq 0: t\geq \lambda_{\theta}(k)\} $.

	%%%%%%%%%%%%%%%%%%%%%%%%%%%%%
	\subsection{Conservative field and path-differentiable function}
	In this subsection, we introduce the concept and fundamental properties of conservative fields, which extend Clarke subdifferential to encompass a broader class of nonsmooth, nonconvex functions while maintaining desirable variational properties. For simplicity, we provide a self-contained description and highlight some essential ingredients for our theoretical analysis. Interested readers may refer to several recent papers \cite{bolte2021conservative,castera2021inertial} for more details.

	We first present the definitions related to the concept of Clarke subdifferential based on the contents in \cite[Section 2]{clarke1990optimization}. 
	\begin{defin}
		\label{Defin_Subdifferential}
		For any given locally Lipschitz continuous function $f: \bb{R}^q \to \bb{R}$ and any $x \in \bb{R}^q$, the Clarke subdifferential of $f$ at $x$, denoted by $\partial f(x)$, is defined as 
		\begin{equation*}
			\begin{aligned}
				\partial f(x) := &\conv\left(\left\{ w \in \bb{R}^q : \exists \{x_k\} \to x, \text{ with }  \{x_k\}\subseteq R_{f},  \{\nabla f(x_k)\} \to w \right\} \right).
			\end{aligned}
		\end{equation*}
		Here $R_{f} = \{x \in \bb{R}^q: \text{$f$ is differentiable at $x$}\}$. 
	\end{defin}
	
	It is worth mentioning that for any locally Lipschitz continuous function $f: \Rn\times \Rp \to \bb{R}$, its Clarke subdifferential is compact and convex for any $(x, y) \in \Rn \times \Rp$. Moreover, the mapping $(x,y) \mapsto \partial f(x, y)$ is graph-closed over $\Rn \times \Rp$ \cite{clarke1990optimization}.

	Now we can present the definition of a set-valued conservative field. 
	\begin{defin}
		\label{Defin_conservative_field}
		Let $\ca{D}: \bb{R}^q \rightrightarrows \bb{R}^q$ be a set-valued graph-closed, nonempty-valued, locally bounded mapping, and $f: \bb{R}^q \to \bb{R}$ a locally Lipschitz continuous function.
		Then $f$ is a potential function for $\D$ if for any $x \in \bb{R}^q$, any absolutely continuous trajectory $\gamma: [0,1]\to \bb{R}^q$ with $\gamma(0) = 0$ and $\gamma(1) = x$, and any measurable functions $v: [0,1]\to \bb{R}^q$ such that $v(t) \in \D(\gamma(t))$ holds for almost every $t \in [0, 1]$, we have
		\begin{equation}
			\label{Eq_Defin_Conservative_mappping}
			f(x) = f(0) + \int_{0}^1 \inner{\gamma'(t), v(t)} \mathrm{d}t. 
		\end{equation}
		Additionally, we say $\D$ is a conservative field for $f$, and such an $f$ is called path-differentiable. 
	\end{defin}
	
	It is worth mentioning that any conservative field defines a unique potential function up to a constant, since the value of the integral does not depend on the selection of the path in \eqref{Eq_Defin_Conservative_mappping}.  Moreover, for any $f$ that is a potential function for some conservative field $\ca{D}$, $\partial f$ is a conservative field that admits $f$ as its potential function, and $\partial f(x, y) \subseteq \mathrm{conv}(\ca{D}(x, y))$ holds for any $(x, y) \in \Rn\times \Rp$ \cite[Corollary 1]{bolte2021conservative}. 
	
	Next, we present the definition of the optimality condition for \eqref{Prob_BLO}, which also appears in \cite{hu2022improved,xiao2024developing}.
	\begin{defin}
		\label{Defin_stationary_point}
		For any given path-differentiable UL objective function $f$ and its corresponding conservative field $\D_f$, we say $(x, y) \in \M$ is a first-order stationary point of \eqref{Prob_BLO} if there exists $\lambda \in \Rp$ such that $0 \in \D_f(x, y) + J_g(x,y)\lambda$. 
	\end{defin}
	
	Finally, we present the following auxiliary lemma whose proof
	directly follows from the graph-closedness and local boundedness of $\D$, hence we omit it for simplicity. 
	\begin{lem}
		\label{Le_approximate_evaluation}
		Given any convex-valued graph-closed set-valued mapping $\D$, any diminishing positive sequence $\{\delta_k\}$,  any uniformly bounded points $\{\xk\}$, and any sequence $\{d_k\}$ that satisfies $d_k \in \conv\left( \D^{\delta_k}(\xk) \right)$, we have that 
		for any constant $\varepsilon>0$, there exists $K_\varepsilon>0$ such that $d_k \in \D^\varepsilon(\xk)$ holds for any $k\geq K_\varepsilon$. Moreover, there exists a diminishing positive sequence $\{\tilde{\delta}_k\}$ such that $d_k \in \D^{\tilde{\delta}_k}(\xk)$ holds for any $k\geq 0$.
	\end{lem}

	\subsection{Differential inclusion}
	In this subsection, we introduce some fundamental concepts related to the differential inclusion, which is essential for establishing the convergence properties of stochastic subgradient methods, as discussed in \cite{benaim2005stochastic,davis2020stochastic,bolte2021conservative,josz2023lyapunov}.   
	\begin{defin}
		\label{Defin_DI}
		For any locally bounded graph closed set-valued mapping $\ca{D}: \Rn \times \Rp \rightrightarrows \Rn \times \Rp$ that is nonempty compact convex-valued,  we say that the absolutely continuous mapping $\gamma:\bb{R}_+ \to \Rn \times \Rp$ is a solution for the differential inclusion 
		\begin{equation}
			\label{Eq_def_DI}
			\left(\frac{\mathrm{d} x}{\mathrm{d}t}, \frac{\mathrm{d} y}{\mathrm{d}t}\right) \in \ca{D}(x, y),
		\end{equation}
		with initial point $(x_0, y_0)$, if $\gamma(0) = (x_0, y_0)$ and $\dot{\gamma}(t) \in \ca{D}(x(t), y(t))$ holds for almost every $t\geq 0$. 
	\end{defin}

	\begin{defin}[Definition II in \cite{benaim2005stochastic}]
		\label{Defin_perturbed_solution}
		We say that an absolutely continuous mapping $\gamma$
		is a perturbed solution to \eqref{Eq_def_DI}  if there exists a locally integrable function $u: \bb{R}_+ \to \Rn \times \Rp$, such that 
		\begin{enumerate}
			\item For any $T>0$, it holds that $\lim\limits_{t \to +\infty} \sup\limits_{0\leq l\leq T} \norm{\int_{t}^{t+l} u(s) ~\mathrm{d}s} = 0$. 
			\item There exists $\delta: \bb{R}_+ \to \bb{R}$ such that $\lim\limits_{t \to +\infty} \delta(t) = 0$ and $\dot{\gamma}(t) - u(t) \in \D^{\delta(t)}(\gamma(t))$. 
		\end{enumerate}
	\end{defin}
	
	Now consider the sequence $\{(\xk, \yk)\}$ generated by the  following updating scheme,  
	\begin{equation}
		\label{Eq_def_Iter}
		(\xkp, \ykp) = (\xk, \yk) - \eta_k(d_{x, k} + \xi_{x, k+1}, d_{y, k} + \xi_{y, k+1}),
	\end{equation}
	where $\{\eta_k\}$ is a diminishing sequence of positive real numbers. 
	We define the interpolated process of $\{(\xk,\yk)\}$ generated by \eqref{Eq_def_Iter} as follows. 
	\begin{defin}
		The  (continuous-time) interpolated process of $\{(\xk,\yk)\}$ generated by \eqref{Eq_def_Iter} is the mapping $w: \bb{R}_+ \to \Rn \times \Rp$ such that 
		\begin{equation} \label{eq-interpolated}
			w(\lambda(i) + s) := (x_{i},\, y_{i}) + \frac{s}{\eta_i} 
			(x_{i+1} - x_{i},\, y_{i+1} - y_{i}), \quad s\in[0, \eta_i), \;\;
			i\geq 0.
		\end{equation}
	\end{defin}

	Then we introduce the following lemma from \cite[Lemma 4]{xiao2023adam}, which shows that the interpolated process of $\{(\xk,\yk)\}$ from \eqref{Eq_def_Iter} is a perturbed solution of the differential inclusion \eqref{Eq_def_DI}.   
	\begin{lem}
		\label{Le_interpolated_process}
		Let $\ca{D}: \Rn \times \Rp \rightrightarrows \Rn \times \Rp$ be a locally bounded graph-closed set-valued mapping that is nonempty, compact, and convex-valued.
		Suppose the following conditions hold in \eqref{Eq_def_Iter}:
		\begin{enumerate}
			\item $\sup_{k \geq 0} \norm{\xk} + \norm{\yk}<+\infty$, $\sup_{k \geq 0} \norm{d_{x, k}} + \norm{d_{y, k}} < +\infty$. 
			\item There exists a nonnegative sequence $\{\delta_k\}$  such that $\lim_{k\to +\infty} \delta_k = 0$ and $(d_{x, k}, d_{y, k}) \in \D^{\delta_k}(\xk, \yk)$.
			\item For any $T> 0$, we have $\lim\limits_{s \to +\infty} \sup\limits_{s\leq i \leq \Lambda(\lambda(s) + T)}\norm{ \sum_{k = s}^{i} \eta_k (\xi_{x, k+1}, \xi_{y, k+1})} =0$. 
		\end{enumerate}
		Then the interpolated process \eqref{eq-interpolated} of $\{(\xk, \yk)\}$ is a perturbed solution for \eqref{Eq_def_DI}. 
	\end{lem}

	\section{Global Convergence of Hybrid Method}
	
	In this section, we develop an efficient subgradient method for \eqref{Prob_BLO} under Assumption \ref{Assumption_joint_nondegeneracy}. Section \ref{Subsection_basic_assumption} introduces basic assumptions, describes the details of our proposed hybrid method, and presents the main result on its convergence properties. Moreover, the proofs for the main theorem are presented in Section \ref{Subsection_proof_main_results}.
	
	\subsection{Basic assumptions, the hybrid method, and the main result}
	\label{Subsection_basic_assumption}
	
	We begin by presenting the basic assumptions for the bilevel optimization problem \eqref{Prob_BLO}.
	\begin{assumpt}
		\label{Assumption_f}
		\begin{enumerate}
			\item The objective function $f$ is path-differentiable and admits a  convex-valued conservative field $\D_f$.
			\item The critical values of \eqref{Prob_Con} have empty interior in $\bb{R}$, in the sense that $\{f(x, y): \nabla_y g(x, y) = 0, ~\exists \lambda \in \Rp \text{ such that }0 \in \D_f(x, y) + J_g(x, y) \lambda \}$ has empty interior in $\bb{R}$. 
		\end{enumerate}
	\end{assumpt}
	As demonstrated in \cite{bolte2021conservative}, the class of path-differentiable functions is general enough to enclose a wide range of nonsmooth nonconvex functions. As a result, Assumption \ref{Assumption_f}(1) is a mild condition. Moreover, by applying the nonsmooth Morse-Sard property to \eqref{Prob_Con}, we can conclude that when both $f$ and $g$ are semi-algebraic and $\D_f = \partial f$, the set $\{f(x, y): \nabla_y g(x, y) = 0, ~\exists \lambda \in \Rp \text{ such that }0 \in \D_f(x, y) + J_g(x, y) \lambda \}$ is finite \cite{davis2020stochastic}. As a result, Assumption \ref{Assumption_f}(2) is also mild in practice. 
	
	\begin{rmk}
		Under Assumption \ref{Assumption_joint_nondegeneracy}, it is important to note that $\ca{S}(x)$ may exhibit non-Lipschitz behavior and can even be set-valued over $\Rn$. For instance, consider the case where $n = p = 1$ and $g(x, y) = \frac{1}{q} y^q - xy$, with $q > 2$ being a fixed even integer. In this case, $g(x, y)$ satisfies Assumption \ref{Assumption_joint_nondegeneracy}, while $\ca{S}(x) = x^{1/(q-1)}$. Furthermore, if we define $\psi(t) = \max\{0, (t^2-1)^5\}$ and set $g(x, y) = \frac{1}{q} \psi(y)^q - xy$, again with $q > 2$ as a fixed even integer, it follows that $g(x, y)$ satisfies Assumption \ref{Assumption_joint_nondegeneracy}, while $\ca{S}(0) = [-1, 1]$. 
		
		These examples illustrate that the mapping $x \mapsto f(x, \ca{S}(x))$ can be non-Lipschitz and even set-valued over $\Rn$. Consequently, standard double-loop methods \cite{domke2012generic,pedregosa2016hyperparameter,ghadimi2018approximation,grazzi2020iteration,ji2021bilevel} cannot be directly applied to solve \eqref{Prob_BLO} under Assumption \ref{Assumption_joint_nondegeneracy}, without imposing additional regularity conditions on the lower-level subproblem \eqref{Eq_BLO_LL}.
	\end{rmk}

	\begin{algorithm}[tb]
		\caption{Two-phase Hybrid Subgradient Descent Method (TPHSD).}
		\label{Alg_TPHSD}
		\begin{algorithmic}[1]
			\REQUIRE Initial values $(x_0, y_0)$, tolerance $\varepsilon_0>0$, $\text{flag}_{-1} = 1$, $\ca{S}_{\mathrm{idx}} = \{+\infty\}$.  
			\FOR{$k = 0,1,2,\ldots$}
			\STATE Compute $u_k + e_k$ as an approximated evaluation of $\nabla p(\xk, \yk)$;
			\IF{$\text{flag}_{k-1} = 1$}
			\IF{$\norm{u_k + e_k} \geq \varepsilon_k \norm{\nabla_y g(\xk, \yk)}$}
			\STATE $\text{flag}_{k} = 1$, $\varepsilon_{k+1} = \varepsilon_k$, run the subroutine \eqref{Eq_Subroutine_manifold}
			\ELSE
			\STATE $\text{flag}_{k} = 2$, $\varepsilon_{k+1} = \frac{\varepsilon_k}{2}$, $\ca{S}_{\mathrm{idx}} = \ca{S}_{\mathrm{idx}} \cup \{k\}$,
			run the subroutine \eqref{Eq_Subroutine_gy}
			\ENDIF
			\ELSE
			\IF{$\norm{\nabla_y g(\xk, \yk)} \leq \varepsilon_{k}$}
			\STATE $\text{flag}_{k} = 1$, $\varepsilon_{k+1} = \varepsilon_k$, $\ca{S}_{\mathrm{idx}} = \ca{S}_{\mathrm{idx}} \cup \{k\}$,
			run the subroutine \eqref{Eq_Subroutine_manifold}
			
			\ELSE
			\STATE $\text{flag}_{k} = 2$, $\varepsilon_{k+1} = \varepsilon_k$, run the subroutine \eqref{Eq_Subroutine_gy}
			\ENDIF
			\ENDIF
			\ENDFOR
		\end{algorithmic}
	\end{algorithm}

	Based on Assumption \ref{Assumption_joint_nondegeneracy} and Assumption \ref{Assumption_f}, we propose a hybrid method that alternates between the two-timescale subgradient method \eqref{Eq_Subroutine_manifold} and the feasibility restoration scheme \eqref{Eq_Subroutine_gy}. As established later in Theorem \ref{Theo_local_convergence_NS_joint_nondeg}, when initialized within a $\rho$-neighborhood of the manifold $\M$ (i.e., $\{(x, y) \in \Rn \times \Rp: \mathrm{dist}((x, y), \M) \leq \rho\}$), the sequence of iterates $\{({x}_k, {y}_k)\}$ is restricted to within such neighborhood and every cluster point is a first-order stationary point of \eqref{Prob_Con}. However, the exact value of $\rho$ is generally unknown in practice. To address this, Algorithm \ref{Alg_TPHSD} employs a hybrid scheme that adaptively adjusts $\varepsilon_k$ to approximate $\rho$.

	In each iteration of Algorithm \ref{Alg_TPHSD}, when $\nabla_y g(\xk, \yk)\neq 0$, we compute the ratio  $\mu_k := \frac{\norm{u_k + e_k}}{\norm{\nabla_y g(\xk, \yk)}}$ to evaluate whether $\{(\xk, \yk)\}$ lies in the particular neighborhood of the manifold $\M$.  Specifically, we consider two cases as follows.
	\begin{itemize}
		\item When $\mu_k \geq \varepsilon_k$ and $\norm{\nabla_y g(\xk, \yk)} \leq \varepsilon_k$, the current iterate $(\xk, \yk)$ is sufficiently close to $\M$. In this case, the algorithm employs the subgradient method \eqref{Eq_Subroutine_manifold} to compute the next iterate. As shown later in Theorem \ref{Theo_local_convergence_NS_joint_nondeg}, with a sufficiently small $\varepsilon_k$ and a sufficiently large $k$, the iterates $\{(\xk, \yk)\}$ are restricted to within $\Omegajoint$ and converge to the first-order stationary points of \eqref{Prob_BLO}. 
		\item Conversely, when $\mu_k < \varepsilon_k$ or $\norm{\nabla_y g(\xk, \yk)} > \varepsilon_k$, the current iterate $(\xk, \yk)$ lies outside the desired neighborhood of $\M$. This suggests that the current tolerance $\varepsilon_k$ is too large to certify that the iterate is in the attraction region, hence we reduce it. Then, Algorithm \ref{Alg_TPHSD} reduces $\varepsilon_{k+1}$ and switches to the feasibility restoration step \eqref{Eq_Subroutine_gy} to drive the iterates towards $\M$, until the iterates re-enter the $\varepsilon_{k+1}$-neighborhood of $\M$. 
	\end{itemize}
	The detailed algorithm is presented in Algorithm \ref{Alg_TPHSD}. Moreover, let $e_k := (e_{x,k}, e_{y,k})$ and $u_k := (u_{x,k}, u_{y,k})$, and we make the following assumptions on Algorithm \ref{Alg_TPHSD}.
	\begin{assumpt}
		\label{Assumption_alg_smooth}
		\begin{enumerate}
			\item The sequence $\{(\xk, \yk)\}$ is uniformly bounded. That is, there exists $M_{iter} > 0$ such that $\sup_{k\geq 0} \norm{\xk} + \norm{\yk} \leq M_{iter}$. 
			\item The stepsizes are positive and two-timescale, in the sense that 
			\begin{equation} 
				\sum_{k\geq 0} \eta_k = +\infty,  \quad \lim_{k\to +\infty} \theta_k = 0,  \quad \lim_{k\to +\infty} \frac{\eta_k}{\theta_k} = 0, \quad \lim_{k\to +\infty} \frac{\theta_k^2}{\eta_k} = 0. 
			\end{equation}
			\item There exists $M_e> 0$ and a diminishing sequence of nonnegative real numbers $\{\omega_k\}$ such that 
			\begin{equation}
				\norm{e_k} \leq M_e \min\{1, \omega_k\}\norm{\nabla_y g(\xk, \yk)}^2, \quad \forall k\geq 0. 
			\end{equation}
		\end{enumerate}
	\end{assumpt}
	
	Here are some remarks on Assumption \ref{Assumption_alg_smooth}. 
	\begin{rmk}
		\label{Rmk_interpolation_nabla_p}
		For any $(\xk, \yk)$, consider the following finite difference scheme for computing $(u_{x, k}, u_{y, k})$,
		\begin{equation*}
			(u_{x, k} + e_{x,k}, u_{y, k} + e_{y,k}) = \frac{\nabla g(\xk, \yk + t_k \nabla_{y}g(\xk, \yk)) - \nabla g(\xk, \yk)}{t_k},
		\end{equation*}
		then there exists a constant $M^{\star}>0$ such that we have the following estimation for the approximation error,
		\begin{equation*}
			\norm{\frac{\nabla g(\xk, \yk + t_k \nabla_{y}g(\xk, \yk)) - \nabla g(\xk, \yk)}{t_k} - \nabla p(\xk, \yk)  } \leq  M^{\star}t_k \norm{\nabla_y g(\xk, \yk)}^2 . 
		\end{equation*}
		Therefore, by choosing $t_k = o(1)$, we can conclude that the above finite difference scheme satisfies Assumption \ref{Assumption_alg_smooth}(3). 
	\end{rmk}
	Furthermore, Assumption \ref{Assumption_alg_smooth}(1) is standard in the convergence analysis of stochastic subgradient methods for minimizing nonsmooth and non-Clarke-regular functions \cite{benaim2005stochastic,borkar2009stochastic,davis2020stochastic,bolte2020mathematical,bolte2021conservative,castera2021inertial,xiao2023adam}.
	Additionally, to choose the stepsizes $\{\eta_k\}$ and $\{\theta_k\}$ that satisfy Assumption \ref{Assumption_alg_smooth}(2), we can use the following polynomial stepsize scheme:
	$$ 
	\eta_k = \eta_0(k+1)^{-a}, \qquad \theta_k = \theta_0(k+1)^{-b}, \qquad \text{where} \quad 0<\frac{a}{2} < b < a < 1. 
	$$

	Now we present the following theorem, which establishes the convergence properties of Algorithm \ref{Alg_TPHSD}. The proof of Theorem \ref{Theo_convergence_NS_joint_nondeg} is presented in Section \ref{Subsection_proof_main_results} for a clearer presentation of our main results. 
	\begin{theo}
		\label{Theo_convergence_NS_joint_nondeg}
		Suppose Assumption \ref{Assumption_joint_nondegeneracy}, Assumption \ref{Assumption_f}, and Assumption \ref{Assumption_alg_smooth}  hold. Then for any sequence $\{(\xk, \yk)\}$ generated by Algorithm \ref{Alg_TPHSD},  any of its cluster points is a first-order stationary point of \eqref{Prob_BLO}. 
	\end{theo}

	\subsection{Proof of Theorem \ref{Theo_convergence_NS_joint_nondeg}}
	\label{Subsection_proof_main_results}
	
	To make the main proof technique clearer, we outline our convergence analysis below:
	\begin{itemize}
		\item \textbf{Section \ref{Subsection_cdf}: Constructing the Lyapunov Function}. 
		We first analyze the constraint-dissolving penalty function $h_\beta(x, y)$ defined in \eqref{Eq_defin_cdf}. We prove that it acts as an exact penalty function for (ECP). More importantly, we show that $h_\beta(x, y)$ can serve as a Lyapunov function for the continuous-time subdifferential flow, which links the convergence properties between \eqref{Eq_Subroutine_manifold} and certain continuous-time differential inclusions.
		
		\item \textbf{Section \ref{Subsection_TMG}: Local Convergence of \eqref{Eq_Subroutine_manifold}}. 
		Directly analyzing the two-timescale scheme \eqref{Eq_Subroutine_manifold} is challenging. Therefore, we construct an auxiliary sequence $(w_k, z_k) := \A(x_k, y_k)$. We prove that this auxiliary sequence can be regarded as a single-timescale inexact subgradient descent on $h_\beta(x, y)$. Then we demonstrate that if \eqref{Eq_Subroutine_manifold} is initialized within a sufficiently close neighborhood $\Omega$ of the feasible region $\M$, the sequence will stay within $\Omega$ and converge to a first-order stationary point of \eqref{Prob_BLO}.
		
		% Moreover, under the mild condition that the upper-level objective $f$ is coercive, we prove that the iterates generated by (TMG) remain uniformly bounded when provided with sufficiently small stepsizes. 
		
		\item \textbf{Section \ref{Subsection_globalization}: Globalization through the Hybrid Approach}.
		Finally, we analyze the convergence of the proposed Two-phase Hybrid Subgradient Descent (TPHSD) algorithm. We prove that the feasibility restoration scheme \eqref{Eq_Subroutine_gy} drives iterates from anywhere in the space toward the attraction region $\Omega$. We then show that the adaptive switching between (TMG) and (FRG) in Algorithm \ref{Alg_TPHSD} is triggered only a finite number of times. Consequently, Algorithm \ref{Alg_TPHSD} eventually follows the \eqref{Eq_Subroutine_manifold} phase within $\Omega$, yielding global convergence to a first-order stationary point of \eqref{Prob_BLO}.
	\end{itemize}

	% We first present Section 3.2.1 to demonstrate the basic properties of the constraint dissolving function $h_{\beta}(x, y)$ defined in \eqref{Eq_defin_cdf}, which serves as the Lyapunov function in our theoretical analysis. Then in Section 3.2.2, we establish the convergence properties of the subgradient method \eqref{Eq_Subroutine_manifold}, as illustrated in Theorem \ref{Theo_local_convergence_NS_joint_nondeg}. Specifically, we show that the subgradient method \eqref{Eq_Subroutine_manifold} converges when initialized sufficiently close to the manifold $\mathcal{M}$, with appropriately chosen stepsizes. 
	
	% Then in Section 3.2.3, we demonstrate that the hybrid scheme introduced in Algorithm \ref{Alg_TPHSD} guarantees that switching between the subgradient method \eqref{Eq_Subroutine_manifold} and the feasibility restoration scheme \eqref{Eq_Subroutine_gy} occurs only finitely many times. As a result, after a finite number of iterations, the sequence ${({x}_k, {y}_k)}$ follows the update scheme \eqref{Eq_Subroutine_manifold}, with its initial points sufficiently close to the manifold $\M$. Combining this result with the convergence guarantees of \eqref{Eq_Subroutine_manifold} in Theorem \ref{Theo_local_convergence_NS_joint_nondeg}, we complete the proof of Theorem \ref{Theo_convergence_NS_joint_nondeg}.

	To establish the convergence properties of Algorithm \ref{Alg_TPHSD}, we first define several constants under Assumption \ref{Assumption_joint_nondegeneracy}, Assumption \ref{Assumption_f}, and Assumption \ref{Assumption_alg_smooth} as follows:
	\begin{itemize}
		\item 
		$L_f := \sup \{\norm{d} : d\in \mathcal{D}_f(x,y), 
		\norm{x} + \norm{y} \leq M_{iter} \} 
		$,
		\item $L_g := \sup_{\norm{x} + \norm{y} \leq M_{iter}} \norm{\nabla g(x, y)}$, \quad
		$M_g := \sup_{\norm{x} + \norm{y} \leq M_{iter}} \norm{\revise{\nabla^2 g}(x,y)}$,
		
		\item  $L_p := \sup_{\norm{x} + \norm{y} \leq M_{iter}} \norm{\nabla p(x, y)}$, \quad $M_p := \sup_{\norm{x} + \norm{y} \leq M_{iter}} \norm{\nabla^2 p(x, y)}$, 
		
		% \item $L_A := \revise{\sup_{(x,y)\in\Omega}} \norm{\nabla \A(x, y)}$, \;
		% $M_A := \revise{\sup_{(x,y)\in\Omega}} \norm{\nabla^2 \A(x, y)}$,
		
		\item $M_{g,3} := \sup_{\norm{x} + \norm{y} \leq M_{iter}} \norm{\nabla^3 g(x, y) }$, 
		
		\item $\sigma := \min\Big\{ 1, \; \inf \big\{\sigma_{\min}\left(J_g(x,y) \right) : \norm{x} + \norm{y} \leq M_{iter},(x,y)\in \mathcal{M}
		\big\} \Big\}$, 
		\item $\rho:= \min\left\{ \frac{\sigma}{4M_g L_g(4M_e +1)}, \frac{M_g}{8 M_e}, \frac{\sigma M_g}{2M_{g,3} +1}, \frac{\sigma}{64M_e}\right\}$. 
	\end{itemize}

	Then, we define the attractive region $\Omegajoint$ as 
	\begin{equation}
		\Omegajoint := \Big\{(x, y) \in \Rn \times \Rp: \mathrm{dist}\left((x, y), \M \right) \leq \frac{\rho}{M_g},  ~\norm{x} + \norm{y} \leq M_{iter} \Big\},
	\end{equation}
	which depends on both Assumption \ref{Assumption_joint_nondegeneracy} and \revise{Assumption \ref{Assumption_alg_smooth}(2)-(3), together with the choice of $M_{iter}$. Then we define 
		\begin{equation*}
			L_A := \sup_{(x,y)\in\Omega} \norm{\nabla \A(x, y)}, \quad M_A := \sup_{(x,y)\in\Omega} \norm{\nabla^2 \A(x, y)},
		\end{equation*}
		where the restriction to $\Omega$ guarantees the well-definedness of $\A$. }

	\subsubsection{Basic properties of exact penalty function}
	\label{Subsection_cdf}
	
	In this part, we present some basic properties and the exactness of the constraint dissolving function $h_{\beta}(x, y)$. Recall the mapping $J_g: \Rn \times \Rp \to \bb{R}^{(n+p)\times p}$ defined by
	\begin{equation}
		J_g(x, y) := \left[
		\begin{smallmatrix}
			\nabla_{xy}^2 g(x, y)\\
			\nabla_{yy}^2 g(x, y)
		\end{smallmatrix}
		\right], \quad \text{and} \quad 
		\nabla p(x,y) = J_g(x,y)\nabla_y g(x,y).
	\end{equation} 
	We have the following lemma based on the non-degeneracy of $J_g(x, y)$ over $\M$. 
	\begin{lem}
		\label{Le_welldef_penalty_joint_nondeg}
		Suppose Assumption \ref{Assumption_joint_nondegeneracy} and Assumption \ref{Assumption_f} hold. Then for any $(x, y) \in \Omegajoint$, it holds that $\norm{\nabla p(x, y)} \geq \frac{\sigma}{2} \norm{\nabla_y g(x, y)}$. \revise{Moreover, we have $\mathrm{dist}((x,y),\M)\leq \frac{2}{\sigma}\norm{\nabla_y g(x,y)}$.}
	\end{lem}
	\begin{proof}
		\revise{For any $(x, y) \in \Omegajoint$, let $(\tilde{x}, \tilde{y})$ be a nearest point in $\M$. Then $\norm{(x, y) - (\tilde{x}, \tilde{y})}=\mathrm{dist}((x,y),\M)\leq \frac{\rho}{M_g}$.}  Then from Assumption \ref{Assumption_joint_nondegeneracy} and the choices of $\rho$, it holds that 
		\begin{equation*}
			\sigma_{\min}\left(J_g(x,y) \right) \geq \sigma_{\min}\left(
			J_g(\tilde{x},\tilde{y})\right) - \frac{\rho}{M_g} M_{g,3} \geq  \sigma - \frac{\rho}{M_g} M_{g,3} \geq   \frac{\sigma}{2}. 
		\end{equation*}
		Therefore, we have $\nabla p(x,y) = J_g(x,y)\nabla_y g(x,y)$
		and
		\begin{equation*}
			\norm{\nabla p(x, y)}  \geq \sigma_{\min}\left(
			J_g(x,y) \right) \norm{\nabla_y g(x, y)} \geq \frac{\sigma}{2} \norm{\nabla_y g(x, y)}.
		\end{equation*}
		
		\revise{On the other hand, let $v:=(x-\tilde{x},y-\tilde{y})$, then we have  $v\in \mathrm{Range}\big(J_g(\tilde{x},\tilde{y})\big)$ as $v$ is normal to $\M$ at $(\tilde{x}, \tilde{y})$.  Then based on the definition of $M_{g,3}$, we have $\norm{\nabla_y g(x,y) - J_g(\tilde{x},\tilde{y})\tp v} \leq \frac{M_{g,3}}{2}\norm{v}^{2}$. Then notice that $\sigma_{\min}(J_g(\tilde{x},\tilde{y}))\geq\sigma$, it follows from the choice of $\rho$ that  
			\begin{equation*}
				\norm{\nabla_y g(x,y)} \geq \left(\sigma-\frac{M_{g,3}}{2}\norm{v}\right)\norm{v} \geq \frac{\sigma}{2}\norm{v} =\frac{\sigma}{2}\mathrm{dist}((x,y),\M),
			\end{equation*}
			where the last inequality uses $\norm{v}\leq \rho/M_g\leq \sigma/(2M_{g,3}+1)$.} 
		This completes the proof. 
	\end{proof}
	
	Recall the constraint dissolving mapping $\A$ defined in \eqref{Eq_defin_cdmapping}, Lemma \ref{Le_welldef_penalty_joint_nondeg} illustrates that $\A$ is well-defined within $\Omegajoint$. We denote the Jacobian of $\A$ as $\nabla \A$, and define the following two auxiliary mappings: 
	\begin{equation}
		\label{Eq_defin_PA_RA}
		P_{\A}(x, y) = I_{n+p} - J_g(x, y) J_g(x, y)^{\dagger}, \quad R_{\A}(x, y) = \nabla \A(x, y) - P_{\A}(x, y). 
	\end{equation}
	It is worth mentioning that $P_{\A}(x,y) \nabla p(x,y) = 0$ holds for any $(x, y) \in \Omega$. 
	
	Then the following lemma establishes the relationship between $\norm{R_{\A}(x, y) \nabla p(x, y)}$ and $\norm{\nabla_y g(x, y)}$, which is important for establishing the exactness of $h_{\beta}(x, y)$. 
	\begin{lem}
		\label{Le_error_bound_cdf}
		Suppose Assumption \ref{Assumption_joint_nondegeneracy} and Assumption \ref{Assumption_f} hold. Then $\A$ is continuously differentiable over $\Omegajoint$. Moreover, it holds for any $(x, y) \in \Omega$ that $\norm{R_{\A}(x, y) \nabla p(x, y)} \leq \frac{8M_gM_{g, 3}}{\sigma^2} \norm{\nabla_y g(x, y)}^2$. 
	\end{lem}
	\begin{proof}
		From the Lipschitz continuity of $J_g$ over $\Omega$ and Lemma \ref{Le_welldef_penalty_joint_nondeg}, the Lipschitz constant of $J_g(x, y)^{\dagger}$ over $\Omega$ is $\frac{8M_{g, 3}}{\sigma^2}$, as demonstrated in \cite[Equation (4.12)]{golub1973differentiation}.  
		
		Note that $R_{\A}(x,y) = - \nabla (J_g^\dagger)\tp(x,y) \nabla_y g(x,y)$. Therefore, it holds for any $(x, y) \in \Omega$ that $\norm{R_{\A}(x, y)} \leq \frac{8M_{g, 3}}{\sigma^2} \norm{\nabla_y g(x, y)}$. Then it holds that $\norm{R_{\A}(x, y) \nabla p(x, y)} \leq \frac{8M_{g, 3}}{\sigma^2} \norm{\nabla_y g(x, y)} \norm{\nabla p(x, y)} \leq \frac{8M_gM_{g, 3}}{\sigma^2} \norm{\nabla_y g(x, y)}^2$. 
		This completes the proof. 
	\end{proof}

	Furthermore, we consider the following set-valued mapping
	\begin{equation}
		\D_{h_{\beta}}(x, y) := \nabla \A(x, y) \D_f(\A(x, y)) + \beta \nabla p(x, y). 
	\end{equation}
	As shown in the following lemma, the function $h_{\beta}(x, y)$ admits the set-valued mapping $\D_{h_{\beta}}(x, y)$ as its conservative field. 
	\begin{lem}
		\label{Le_conservative_field_joint_nondeg}
		Suppose Assumption \ref{Assumption_joint_nondegeneracy} and Assumption \ref{Assumption_f} hold, then for any $\beta > 0$, $h_{\beta}$ is a path-differentiable function that admits  $\D_{h_{\beta}}$ as its conservative field. 
	\end{lem}
	\begin{proof}
		Notice that $f$ admits $\D_f$ as its conservative field, and both $\A$ and $p$ are differentiable over $\Omega$,
		together with the chain rule for conservative field in \cite[Lemma 5]{bolte2021conservative}, we can conclude that $\D_{h_{\beta}}$ is a conservative field for $h_{\beta}$. This completes the proof. 
	\end{proof}

	The following proposition illustrates that the function $h_{\beta}(x, y)$ is an exact penalty function for \eqref{Prob_BLO}, in the sense that any first-order stationary point of $h_{\beta}(x, y)$ within $\Omega$ is feasible for \eqref{Prob_Con}, and hence is a first-order stationary point of \eqref{Prob_Con}. 
	\begin{prop}
		\label{Prop_exactness_joint_nondeg}
		Suppose Assumption \ref{Assumption_joint_nondegeneracy} and Assumption \ref{Assumption_f} hold. Then for any $(x, y) \in \Omegajoint$ and any $\beta> \frac{32 M_g M_{g, 3}L_f}{\sigma^4}$, whenever $(x, y)$ is a first-order stationary point of $h_{\beta}$, we have $\nabla_y g(x, y) = 0$. 
	\end{prop}
	\begin{proof}
		Since $(x, y)$ is a first-order stationary point of $h_{\beta}$, we have
		\begin{equation*}
			\begin{aligned}
				&0 \in \nabla \A(x, y) \D_f(\A(x, y)) + \beta \nabla p(x, y)\\
				={}& (P_{\A}(x, y) + R_{\A}(x, y)) \D_f(\A(x, y)) + \beta \nabla p(x, y).
			\end{aligned}
		\end{equation*}
		Then there exists $d\in \D_f(\A(x, y))$ such that
		\begin{equation*}
			\begin{aligned}
				&0  =  \inner{\nabla p(x, y), (P_\A(x,y) + R_{\A}(x, y)) d } + \beta \norm{ \nabla p(x, y)}^2\\
				\geq{}& \frac{\beta \sigma^2}{4}\norm{ \nabla_y g(x, y)}^2 - \inner{\nabla p(x, y), R_{\A}(x, y) d }\\
				\geq{}& \frac{\beta \sigma^2}{4}\norm{ \nabla_y g(x, y)}^2 - 
				\frac{8M_gM_{g, 3}L_f}{\sigma^2} \norm{ \nabla_y g(x, y)}^2= \frac{\sigma^2}{4} \left( \beta - \frac{32 M_gM_{g, 3}L_f}{\sigma^4}\right)\norm{ \nabla_y g(x, y)}^2.
			\end{aligned}
		\end{equation*}
		Here, the first inequality follows from Lemma \ref{Le_welldef_penalty_joint_nondeg} and the fact that $P_{\A}(x,y) \nabla p(x,y) = 0$,
		the last inequality follows from Lemma \ref{Le_error_bound_cdf}.  Together with the fact that $\beta> \frac{32 M_gM_{g, 3}L_f}{\sigma^4}$, we can conclude that $\nabla_y g(x, y) = 0$. This completes the proof. 
	\end{proof}

	The following lemma illustrates that any path-differentiable function $\tilde{h}$ is a Lyapunov function for its corresponding subdifferential flow. The proof of Lemma \ref{Le_DI_cdf} directly follows from \cite{davis2020stochastic}, hence is omitted for simplicity. 
	\begin{lem}
		\label{Le_DI_cdf}
		For any path-differentiable function $\tilde{h}: \Rn \times \Rp \to \bb{R}$ that admits $\D_{\tilde{h}}: \Rn \times \Rp \rightrightarrows \Rn \times \Rp$ as its conservative field, the differential inclusion $\left(\frac{\mathrm{d}x}{\mathrm{d}t}, \frac{\mathrm{d}y}{\mathrm{d}t} \right) \in -\D_{\tilde{h}}(x, y)$
		admits $\tilde{h}$ as its Lyapunov function with stable set $\ca{S}_{DI} := \{(x, y) \in \Rn \times \Rp:  0 \in \D_{\tilde{h}}(x, y)\}$. 
	\end{lem}

	\subsubsection{Convergence properties of subgradient method}
	\label{Subsection_TMG}
	In this part, we analyze the convergence properties of the subgradient method \eqref{Eq_Subroutine_manifold}. 
	
	We first have the following proposition showing that any sequence $\{(\xk, \yk)\}$ with the initial point sufficiently close to $\M$ is restricted to within $\Omegajoint$, when the stepsizes $\{\eta_k\}$ and $\{\theta_k\}$ satisfy the following condition,
	\begin{equation}
		\label{Eq_Cond_Stepsizes}
		\begin{aligned}
			&\sup_{k\geq 0}\theta_k \leq \min\left\{\frac{\sigma}{2(M_g+1)(16M_p + M_g +1)},\; \frac{\rho\sigma}{16 M_g^2 L_g (M_g+ M_e  L_g)} \right\},\\
			&\sup_{k\geq 0} \eta_k \leq \frac{\rho \sigma}{16L_f(M_g+1)^2 }, \quad \sup_{k\geq 0}\frac{\eta_k}{\theta_k} 
			\leq \frac{\rho\sigma^2 }{128L_f(M_g+1)}, 
			\\
			&\sup_{k\geq 0}\frac{\theta_k^2}{\eta_k} \leq \min\left\{\frac{\sigma}{2M_g}, \frac{L_f}{M_g(L_p + M_e\rho^2)}\right\}. 
		\end{aligned}
	\end{equation}

	\begin{prop}
		\label{Prop_restrict_manifold_joint_nondeg_New}
		Suppose Assumption \ref{Assumption_joint_nondegeneracy}, Assumption \ref{Assumption_f}, and Assumption \ref{Assumption_alg_smooth} hold. For any sequence $\{(\xk, \yk)\}$ generated by \eqref{Eq_Subroutine_manifold} with $(x_0, y_0) \in \Omega$, and the sequences of stepsizes $\{\eta_k\}$ and $\{\theta_k\}$  satisfy \eqref{Eq_Cond_Stepsizes}, it holds that $(\xk, \yk) \in \{(x, y): \norm{x} + \norm{y} \leq M_{\mathrm{iter}},~ \mathrm{dist}\left((x, y), \M \right) \leq \frac{36 L_f}{\sigma^2}\sup_{k\geq 0}\frac{\eta_k}{\theta_k}\} \subseteq \Omega$ for any $k\geq 0$. 
	\end{prop}
	\begin{proof}
		For notational convenience, we let $r= \sup_{k\geq 0}\frac{\eta_k}{\theta_k}$.
		Firstly, from the fact that 
		$%\sup_{k\geq 0}\frac{\eta_k}{\theta_k} 
		r \leq \frac{\sigma^2 \rho}{128L_fM_g}$ in \eqref{Eq_Cond_Stepsizes}, we can conclude that 
		$\frac{36 L_f}{\sigma^2}%\sup_{k\geq 0}\frac{\eta_k}{\theta_k} 
		r \leq \frac{\rho}{2M_g}$. Therefore, it follows from the definition of $\Omega$ that 
		\begin{equation*}
			\left\{(x, y): \norm{x} + \norm{y} \leq M_{\mathrm{iter}},~ \mathrm{dist}\left((x, y), \M \right) \leq \frac{36 L_f}{\sigma^2} %\sup_{k\geq 0}\frac{\eta_k}{\theta_k} 
			r \right\} \subseteq \Omega.
		\end{equation*}
		
		Let $d_k = (d_{x, k}, d_{y, k})$, $\mk = (m_{x, k}, m_{y, k})$, $u_k = (u_{x, k}, u_{y, k})$ and $e_k = (e_{x, k}, e_{y, k})$. From the update scheme of $\{\mk\}$, it holds for any $k \geq 0$ that 
		\begin{equation} \label{eq-m}
			\norm{\mkp} = \norm{\alpha^{k+1} m_0 +  (1-\alpha)\sum_{i = 0}^k \alpha^{k-i} d_i} \leq \max\left\{\sup_{0\leq i\leq k} \norm{d_i}, \norm{m_0} \right\}\leq L_f. 
		\end{equation}
		Then from the update scheme of \eqref{Eq_Subroutine_manifold}, it holds that 
		\begin{small}
			\begin{equation}
				\label{Eq_Prop_restrict_manifold_joint_nondeg_0}
				\begin{aligned}
					&p(\xkp, \ykp) - p(\xk, \yk) \\
					\leq{}& \inner{\nabla p(\xk, \yk), -\eta_k \mkp - \theta_k  (u_k + e_k)} + \frac{M_p}{2} \norm{\eta_k \mkp + \theta_k  (u_k + e_k)}^2\\
					% \leq{}& \inner{\nabla p(\xk, \yk), -\eta_k \mkp - \theta_k  (u_k + e_k)} + M_p \eta_k^2 \norm{\mkp}^2 + M_p \theta_k^2 \norm{u_k + e_k}^2  \\
					\leq{}& \inner{\nabla p(\xk, \yk), -\eta_k \mkp - \theta_k  (u_k + e_k)} + M_p \eta_k^2 \norm{\mkp}^2 + 2M_p \theta_k^2 \norm{u_k}^2 + 2M_p \theta_k^2 \norm{e_k}^2
					\\
					\leq{}& \inner{\nabla p(\xk, \yk),  - \theta_k  u_k} + \left(M_p \eta_k^2 + \frac{16\eta_k^2}{\theta_k}\right) \norm{\mkp}^2 \\
					&+ \left(2M_p \theta_k^2 +16 \theta_k\right)\norm{e_k}^2 + \frac{\theta_k}{8} \norm{\nabla p(\xk, \yk)}^2
					\\
					\leq{}&- \frac{7\theta_k}{8} \norm{\nabla p(\xk, \yk)}^2 
					+\frac{18\eta_k^2}{ \theta_k}  L_f^2+ 18\theta_k M_e^2 \norm{\nabla_y g(\xk, \yk)}^4 
					\\
					\leq{}& - \frac{7\sigma^2 \theta_k}{16} p(\xk, \yk) + \frac{18\eta_k^2}{ \theta_k}  L_f^2 + 72 \theta_k M_e^2 p(\xk, \yk)^2 \leq -\frac{\sigma^2\theta_k}{4} p(\xk, \yk) + \frac{18\eta_k^2}{\theta_k}  L_f^2. 
				\end{aligned}
			\end{equation}
		\end{small}
		Here the third inequality uses the following inequalities, 
		\begin{equation*}
			\begin{aligned}
				&\inner{\nabla p(\xk, \yk),  - \eta_k \mkp} \leq \frac{16\eta_k^2}{\theta_k} \norm{\mkp}^2 + \frac{\theta_k}{16} \norm{\nabla p(\xk, \yk)}^2,\\
				&\inner{\nabla p(\xk, \yk),  -  \theta_k e_k} \leq \frac{\theta_k}{16} \norm{\nabla p(\xk, \yk)}^2 + 16 \theta_k \norm{e_k}^2. 
			\end{aligned}
		\end{equation*}
		In addition, the fourth inequality follows from Assumption \ref{Assumption_alg_smooth} and the fact that $\theta_k \leq \frac{1}{M_p}$, and the fifth inequality follows from the fact that 
		\begin{equation*}
			\norm{\nabla p(x, y)}^2 = \norm{J_g(x, y) \nabla_y g(x, y)}^2 \geq \frac{\sigma^2}{4} \norm{\nabla_y g(x, y)}^2 = \frac{\sigma^2}{2} p(x, y),
		\end{equation*}
		for any $(x, y) \in \Omegajoint$. The last inequality follows from the choice of $\Omega$, which leads to 
		\begin{equation}
			\label{Eq_Prop_restrict_manifold_joint_nondeg_3}
			\norm{\nabla_y g(\xk, \yk)} \leq M_g \mathrm{dist}((\xk, \yk), \M) \leq \rho \leq \frac{\sigma}{16 M_e}, 
		\end{equation}
		and  it implies that $72  M_e^2 p(\xk, \yk) \leq \frac{3\sigma^2}{16} $.
		
		For any $(\xk, \yk)$ such that $\norm{\nabla_y g(\xk, \yk)} \geq  \frac{12 L_f}{\sigma} %\sup_{k\geq 0}\frac{\eta_k}{\theta_k}
		r$, it follows from \eqref{Eq_Prop_restrict_manifold_joint_nondeg_0} that 
		\begin{equation}
			\label{Eq_Prop_restrict_manifold_joint_nondeg_1}
			\norm{\nabla_y g(\xkp, \ykp)} \leq \norm{\nabla_y g(\xk, \yk)}.
		\end{equation}
		On the other hand, for any $(\xk, \yk)$ such that $\norm{\nabla_y g(\xk, \yk)} \leq \frac{12 L_f}{\sigma} %\sup_{k\geq 0}\frac{\eta_k}{\theta_k}
		r$, it holds  that $\mathrm{dist}((\xk, \yk), \M) \leq \frac{2}{\sigma} \norm{\nabla_y g(\xk, \yk)} \leq \frac{24 L_f}{\sigma^2} r$.
		
		Furthermore, from \eqref{Eq_Prop_restrict_manifold_joint_nondeg_3} we have that $\norm{\nabla_y g(\xk, \yk)}^2 \leq \rho^2$. As a result, it follows from the update scheme \eqref{Eq_Subroutine_manifold} and Assumption \ref{Assumption_alg_smooth}(3) that  
		\begin{equation*}
			\label{Eq_Prop_restrict_manifold_joint_nondeg_2}
			\begin{aligned}
				&\norm{(\xkp, \ykp) - (\xk, \yk)} \leq  \norm{\eta_k \mkp + \theta_k (u_k + e_k)} 
				\\
				\leq{}& \norm{\mkp}\eta_k + (\norm{u_k} + \norm{e_k})\theta_k 
				\;\leq{}\; L_f \eta_k  + \left(L_p + M_e\rho^2 \right) \theta_k. 
			\end{aligned}
		\end{equation*}
		Therefore, from the Lipschitz continuity of $\nabla g$, it holds that 
		\begin{equation*}
			\begin{aligned}
				&\norm{\nabla_y g(\xkp, \ykp)} \leq \norm{\nabla_y g(\xk, \yk)} + M_g \left( L_f \eta_k  + \left(L_p + M_e\rho^2 \right) \theta_k  \right) 
				\\
				\leq{}& \norm{\nabla_y g(\xk, \yk)} + \left(M_g L_f \sup_{k\geq 0} \theta_k + M_g (L_p + M_e\rho^2) \sup_{k\geq 0} \frac{\theta_k^2}{\eta_k} \right) r \\
				\leq{}& \norm{\nabla_y g(\xk, \yk)} + \frac{6 L_f}{\sigma}
				%\sup_{k\geq 0}\frac{\eta_k}{\theta_k}\ 
				r \leq  \frac{18 L_f}{\sigma}
				%\sup_{k\geq 0}\frac{\eta_k}{\theta_k}
				r. 
			\end{aligned}
		\end{equation*}
		
		Now we are ready to prove that $\norm{\nabla_y g(\xk, \yk)} \leq \frac{18 L_f}{\sigma}
		%\sup_{k\geq 0}\frac{\eta_k}{\theta_k}
		r $ holds for any $k\geq 0$. We prove this statement by contradiction. That is, we assume that there exists an index $i \geq 0$ such that $\norm{\nabla_y g(x_{i+1}, y_{i+1})} > \frac{18 L_f}{\sigma} %\sup_{k\geq 0}\frac{\eta_k}{\theta_k}
		r$ and $\norm{\nabla_y g(x_{i}, y_{i})} \leq \frac{18 L_f}{\sigma}%\sup_{k\geq 0}\frac{\eta_k}{\theta_k}
		r$. 
		Then  \eqref{Eq_Prop_restrict_manifold_joint_nondeg_1} implies that $\norm{\nabla_y g(x_{i+1}, y_{i+1})} \leq \norm{\nabla_y g(x_{i}, y_{i})} \leq \frac{18 L_f}{\sigma} r$, which leads to a contradiction. 
		Therefore, we can conclude that $\norm{\nabla_y g(\xk, \yk)} \leq \frac{18 L_f}{\sigma} r$ holds for any $k\geq 0$. 
		As a result, it follows from Lemma \ref{Le_welldef_penalty_joint_nondeg} that 
		\begin{equation*}
			\sup_{k\geq 0} \mathrm{dist}((\xk, \yk), \M) \leq \frac{2}{\sigma} \sup_{k\geq 0} \norm{\nabla_y g(\xk, \yk)} \leq \frac{36 L_f}{\sigma^2} 
			%\sup_{k\geq 0} \frac{\eta_k}{\theta_k}.
			r.
		\end{equation*}
		
		Moreover, the condition $\sup_{k\geq 0}\frac{\eta_k}{\theta_k} \leq \frac{\sigma^2 \rho}{128L_fM_g}$ in \eqref{Eq_Cond_Stepsizes} implies that $\frac{36 L_f}{\sigma^2}\sup_{k\geq 0}\frac{\eta_k}{\theta_k} \leq \frac{\rho}{2M_g}$. This illustrates that $\{(x, y): \norm{x} + \norm{y} \leq M_{\mathrm{iter}},~ \mathrm{dist}\left((x, y), \M \right) \leq \frac{36 L_f}{\sigma^2}\sup_{k\geq 0}\frac{\eta_k}{\theta_k}\} \subseteq \Omega$. 
		This completes the proof.  
	\end{proof}

	% 
	% Moreover, the condition $\sup_{k\geq 0}\frac{\eta_k}{\theta_k} \leq \frac{\sigma^2 \rho}{128L_fM_g}$ in \eqref{Eq_Cond_Stepsizes} implies that $\frac{36 L_f}{\sigma^2}\sup_{k\geq 0}\frac{\eta_k}{\theta_k} \leq \frac{\sigma}{2M_g}$. Consequently, we get the following result as a corollary of Proposition \ref{Prop_restrict_manifold_joint_nondeg_New}. 
	% \begin{prop}
		%     \label{Prop_restrict_manifold_joint_nondeg}
		%     Suppose Assumption \ref{Assumption_joint_nondegeneracy}, Assumption \ref{Assumption_f}, and Assumption \ref{Assumption_alg_smooth} hold. For any sequence $\{(\xk, \yk)\}$ generated by \eqref{Eq_Subroutine_manifold} with $(x_0, y_0) \in \Omegajoint$, and the sequences of stepsizes $\{\eta_k\}$ and $\{\theta_k\}$  satisfy \eqref{Eq_Cond_Stepsizes}, it holds that $(\xk, \yk) \in \Omegajoint$ for any $k\geq 0$. 
		% \end{prop}
	% 

	Then we have the following lemma illustrating that the sequence $\{(\xk, \yk)\}$ sequentially converges towards the feasible region $\M$. 
	\begin{lem}
		\label{Le_esti_gy_joint_nondeg}
		Suppose Assumption \ref{Assumption_joint_nondegeneracy}, Assumption \ref{Assumption_f}, and Assumption \ref{Assumption_alg_smooth} hold. For any sequence $\{(\xk, \yk)\}$ generated by \eqref{Eq_Subroutine_manifold} with $(x_0, y_0) \in \Omegajoint$, $\norm{m_0} \leq L_f$, and the sequences of stepsizes $\{\eta_k\}$ and $\{\theta_k\}$ satisfy \eqref{Eq_Cond_Stepsizes}, 
		it holds that $\lim_{k\to +\infty} \nabla_y g(\xk, \yk) = 0$. 
	\end{lem}
	\begin{proof}
		From the telescope sum of \eqref{Eq_Prop_restrict_manifold_joint_nondeg_0}, and the uniform boundedness of $\{(\xk, \yk)\}$ in Assumption \ref{Assumption_alg_smooth}(1), we can conclude that $\liminf_{k\to +\infty} \norm{\nabla_y g(\xk, \yk)} = 0$.
		
		We aim to prove this lemma by contradiction. That is, we assume that there exists a cluster point 
		$(\tilde{x}, \tilde{y})$ of $\{(x_k,y_k)\}$ such that $\nabla_y g(\tilde{x}, \tilde{y}) \neq 0$. Hence there exists $\varepsilon \in (0, \rho)$ such that $\norm{\nabla_y g(\tilde{x}, \tilde{y})} > \varepsilon$. Then we can choose two sequences of indices $\{k_{i,1}\}$ and $\{k_{i, 2}\}$ such that the following inequalities hold for any $i\geq 0$, 
		\begin{equation}
			\label{Eq_Le_esti_gy_joint_nondeg_1}
			\begin{aligned}
				&\norm{\nabla_y g(x_{k_{i,1}-1}, y_{k_{i,1}-1})} <\frac{\varepsilon}{2}, \quad \norm{\nabla_y g(x_{k_{i,2}+1}, y_{k_{i,2}+1})} > \varepsilon, \\     
				&\left\{(\xk, \yk): k_{i,1} \leq k\leq k_{i,2} \right\} \subset \left\{(x, y) \in \Omegajoint: \frac{\varepsilon}{2} \leq \norm{\nabla_y g(x, y)} \leq \varepsilon \right\}.
			\end{aligned}
		\end{equation}
		Note that from the choices of $\{k_{i,1}\}$ and $\{k_{i, 2}\}$, it holds that the sequence of iterates $\left\{(\xk, \yk): k_{i,1} \leq k\leq k_{i,2} \right\}$ travels from the boundary of $\left\{(x, y) \in \Omegajoint: \norm{\nabla_y g(x, y)} \leq \frac{\varepsilon}{2} \right\}$ to that of $\left\{(x, y) \in \Omegajoint: \norm{\nabla_y g(x, y)} \geq  \varepsilon \right\}$.

		Together with Assumption \ref{Assumption_alg_smooth}(2),  we can choose $K > 0$ such that $\sup_{k\geq K} \eta_k \leq \frac{\varepsilon}{18 M_g(L_f + M_gL_g)}$ and $\sup_{k\geq K} \frac{\eta_k}{\theta_k} \leq \frac{\sigma^2 \varepsilon}{36 L_f}$. 
		By Assumption \ref{Assumption_alg_smooth}(2) again, for any $i \geq 0$ such that $K\leq k_{i,1}$, it holds that $\sum_{j = k_{i, 1}}^{k_{i, 2}} \theta_j \geq \frac{\varepsilon}{4M_g(L_f + M_gL_g)}$. 
		As a result, for any $i \geq 0$ such that $K\leq k_{i,1}$, it follows from \eqref{Eq_Prop_restrict_manifold_joint_nondeg_0} that 
		\begin{equation}
			\label{Eq_Le_esti_gy_joint_nondeg_2}
			\begin{aligned}
				& p(x_{k_{i,2}+1}, y_{k_{i,2}+1}) - p(x_{k_{i,1}}, y_{k_{i,1}}) \leq   \sum_{k = k_{i,1} }^{k_{i,2}}
				\left(-\frac{\sigma^2\theta_k}{4} p(\xk, \yk) + \frac{18\eta_k^2}{\theta_k}  L_f^2\right) \\
				\leq{}& \sum_{k = k_{i,1} }^{k_{i,2}} -\frac{\sigma^2\theta_k}{8} p(\xk, \yk)  \leq  -\frac{\varepsilon^3 \sigma^2}{1024 M_g (L_f + M_gL_g)} < 0. 
			\end{aligned}
		\end{equation}
		Here the second inequality  follows from \eqref{Eq_Le_esti_gy_joint_nondeg_1}, which implies that $ \norm{\nabla_y g(\xk, \yk)} \geq \frac{\varepsilon}{2}$ holds for  any $k_{i,1}\leq k\leq k_{i,2}$ and any $i\geq 0$. As a result, for any $i\geq 0$ and any $k_{i,1}\leq k\leq k_{i,2}$, we can conclude that $p(\xk, \yk) \geq \frac{\varepsilon^2}{8} \geq \frac{144L_f^2}{\sigma^2} \cdot \frac{\eta_k^2}{\theta_k^2}$ and thus $\frac{18\eta_k^2}{\theta_k}  L_f^2 \leq \frac{\sigma^2\theta_k}{8} p(\xk, \yk)$.
		Here we have the contradiction between \eqref{Eq_Le_esti_gy_joint_nondeg_1} and \eqref{Eq_Le_esti_gy_joint_nondeg_2}. As a result, our assumption on the infeasibility of $(\tilde{x}, \tilde{y})$ is invalid, hence any  cluster point $(\tilde{x}, \tilde{y})$ of $\{(\xk, \yk)\}$ satisfies $\nabla_y g(\tilde{x}, \tilde{y}) = 0$. This completes the proof. 
	\end{proof}

	In the following proposition, we show that the sequence of momentum terms $\{\mk\}$ can be regarded as approximated evaluations of $\{\D_f(\xk, \yk)\}$. 
	\begin{prop}
		\label{Prop_dk_esti_UB}
		Suppose Assumption \ref{Assumption_joint_nondegeneracy}, Assumption \ref{Assumption_f}, and Assumption \ref{Assumption_alg_smooth}  hold. Then for any sequences of stepsizes $\{\eta_k\}$ and $\{\theta_k\}$ satisfying \eqref{Eq_Cond_Stepsizes}, there exists a positive sequence $\{\delta^{\star}_k\}$ such that $\lim_{k\to +\infty} \delta^{\star}_k = 0$, and $(1-\alpha) \sum_{i = 0}^k \alpha^{k-i}(d_{x, i}, d_{y, i}) \in  \conv\left(\D_f^{\delta^{\star}_k}(\xk, \yk) \right)$ holds for any $k\geq 0$.
	\end{prop}
	\begin{proof}
		
		We first show that there exists a nonnegative sequence $\{\tilde{\delta}_k\}$ such that $\lim_{k\to +\infty} \tilde{\delta}_k = 0$, and $\lim_{k\to +\infty}\mathrm{dist}\Big( (1-\alpha) \sum_{i = 0}^k \alpha^{k-i}(d_{x, i}, d_{y, i}),  \conv\left(\D_f^{\tilde{\delta}_k}(\xk, \yk) \right)\Big) =0$. 
		
		For any $C > 0$, let 
		$\hat{\delta}_k := \sum_{i= \max\{0, k-C\}}^k  \sqrt{\theta_i + \eta_i}$, 
		$\hat{T}_k := \Lambda(\lambda(k) - \hat{\delta}_k)$, $\tilde{\delta}_k = \hat{\delta}_k + \sup_{\hat{T}_k\leq i\leq k} \delta_i$, and $T_k := \Lambda(\lambda(k) - \tilde{\delta}_k)$. Then from the facts that $\lim_{k\to +\infty} \eta_k = 0$, $\lim_{k\to +\infty} \theta_k = 0$ and $\lim_{k\to +\infty}\frac{\eta_k}{\theta_k} = 0$, it holds that $\lim_{k\to +\infty} \hat{\delta}_k = 0$ and $\sup_{k\geq 0} \hat{\delta}_k \leq C\sup_{k\geq 0} \sqrt{\eta_k + \theta_k}$.
		As a result, we can conclude that $T_k \to +\infty$, $\tilde{\delta}_k \to 0$, and $\sup_{k\geq 0} \tilde{\delta}_k \leq 2C\sup_{k\geq 0} \sqrt{\eta_k + \theta_k}$. Moreover, it follows from the definition of $T_k$ that $k-T_k \geq C$ holds for any $k\geq 0$.

		Notice that $(d_{x, i}, d_{y, i}) \in \D_f^{\tilde{\delta}_k}(\xk, \yk)$ for any $i$  satisfying $T_k \leq i\leq k$. Then it holds that $(1-\alpha)\sum_{i = T_k}^k \alpha^{k-i}(d_{x, i}, d_{y, i}) \in (1 - \alpha^{k-T_k +1}) \conv\left( \D_f^{\tilde{\delta}_k}(\xk, \yk) \right)$.
		
		Moreover, from the uniform boundedness of $\{(d_{x, k}, d_{y, k})\}$, it holds that 
		\begin{equation*}
			\begin{aligned}
				&(1-\alpha)\norm{\left(  \sum_{i = 0}^k \alpha^{k-i}(d_{x, i}, d_{y, i})  \right) - \left( \sum_{i = T_k}^k \alpha^{k-i}(d_{x, i}, d_{y, i}) \right)}
				\;\leq\; \alpha^{k-T_k} \sup_{k\geq 0} \norm{(d_{x, k}, d_{y, k})}. 
			\end{aligned}
		\end{equation*}
		As a result, it holds for any $k \geq C$ that  
		\begin{equation}
			\label{Eq_Prop_dk_esti_UB_0}
			\begin{aligned}
				&\mathrm{dist}\Big((1-\alpha) \sum_{i = 0}^k \alpha^{k-i}(d_{x, i}, d_{y, i}),  ~\conv\left( \D_f^{\tilde{\delta}_k}(\xk, \yk) \right)  \Big) \\
				\leq{}& \mathrm{dist}\Big((1-\alpha) \sum_{i = 0}^k \alpha^{k-i}(d_{x, i}, d_{y, i}), ~ (1 - \alpha^{k-T_k +1})\conv\left( \D_f^{\tilde{\delta}_k}(\xk, \yk) \right)  \Big) \\
				&  + \alpha^{k-T_k +1} \Big(\sup_{k\geq 0, ~ w \in \D_f(\xk, \yk)} \norm{w}\Big) \\
				\leq{}& \alpha^{k-T_k} \sup_{k\geq 0} \norm{(d_{x, k}, d_{y, k})} + \alpha^{k-T_k +1} \Big(\sup_{k\geq 0, w \in \D_f(\xk, \yk)} \norm{w}\Big)\leq  2\alpha^{C-1} L_f. 
			\end{aligned}
		\end{equation}
		
		From the arbitrariness of $C> 0$, it holds that 
		\begin{equation*}
			\lim_{k\to +\infty} \mathrm{dist}\Big((1-\alpha) \sum_{i = 0}^k \alpha^{k-i}(d_{x, i}, d_{y, i}),  \conv\left( \D_f^{\tilde{\delta}_k}(\xk, \yk) \right)  \Big) = 0.
		\end{equation*}
		Then together with Lemma \ref{Le_approximate_evaluation}, we complete the proof. 
		
\end{proof}

Next, we aim to analyze the convergence of \eqref{Eq_Subroutine_manifold} to the stationary points of \eqref{Prob_Con}, by relating the sequence of its iterates $\{(\xk, \yk)\}$ to a differential inclusion that minimizes the constraint-dissolving function $h_{\beta}(x, y)$ defined in \eqref{Eq_defin_cdf}.

The following lemma illustrates that the accumulation of $\{\theta_k \norm{\nabla_y g(\xk, \yk)}^2\}$ can be controlled by the stepsizes $\{\eta_k\}$. 
\begin{lem}
	\label{Le_accumulation_gy_joint_nondeg_New}
	Suppose Assumption \ref{Assumption_joint_nondegeneracy}, Assumption \ref{Assumption_f}, and Assumption \ref{Assumption_alg_smooth}  hold. Then for any sequence $\{(\xk, \yk)\}$ generated by \eqref{Eq_Subroutine_manifold} with $(x_0, y_0) \in \{(x, y): \norm{x} + \norm{y} \leq M_{\mathrm{iter}},~ \mathrm{dist}\left((x, y), \M \right) \leq \frac{24 L_f}{\sigma^2}\sup_{k\geq 0}\frac{\eta_k}{\theta_k}\} \subseteq \Omega$, and  $\{\eta_k\}$ and $\{\theta_k\}$  satisfy \eqref{Eq_Cond_Stepsizes}, 
	it holds  for any $T>0$ that $\limsup\limits_{s\to +\infty}  \sum_{k = s}^{\Lambda(\lambda(s) + T)} \theta_k \norm{\nabla_y g(\xk, \yk) }^2 = 0$.
\end{lem}
\begin{proof}
	From the update scheme \eqref{Eq_Subroutine_manifold} and the inequality \eqref{Eq_Prop_restrict_manifold_joint_nondeg_0}, it holds for any $k\geq 0$ that $p(\xkp, \ykp) - p(\xk, \yk) \leq  -\frac{\sigma^2\theta_k}{8} \norm{\nabla_y g(\xk, \yk)}^2 + \frac{18\eta_k^2}{\theta_k}  L_f^2$. 
	Then we can estimate the upper bound for $\theta_k \norm{\nabla_y g(\xk, \yk)}^2$ by 
	\begin{equation}
		\label{Eq_Le_accumulation_gy_joint_nondeg_0}
		\theta_k \norm{\nabla_y g(\xk, \yk)}^2 \leq \frac{8}{\sigma^2} \left( p(\xk, \yk) - p(\xkp, \ykp) \right) + \frac{\revise{144} \eta_k^2L_f^2}{\theta_k\sigma^2}. 
	\end{equation}
	% Therefore, for any $s \geq 0$, it holds that 
	% \begin{equation*}
		%     \begin{aligned}
			%         &\sum_{k = s}^{\Lambda(\lambda(s) + T)} \theta_k \norm{\nabla_y g(\xk, \yk) }^2\leq \sum_{k = s}^{\Lambda(\lambda(s) + T)} \left( \frac{8}{\sigma^2} \left( p(\xk, \yk) - p(\xkp, \ykp) \right) + \frac{72 L_f^2}{\sigma^2} \cdot \frac{\eta_k^2}{\theta_k} \right)\\
			%         \leq{}& \frac{8}{\sigma^2}\sup_{s \geq 0} \left(p(x_s, y_s) - p(x_{\Lambda(\lambda(s) + T) +1}, y_{\Lambda(\lambda(s) + T) +1})  \right) + \frac{72 L_f^2}{\sigma^2} \sup_{s \geq 0} \sum_{k = s}^{\Lambda(\lambda(s) + T)}  \frac{\eta_k^2}{\theta_k}\\
			%         \leq{}& \frac{1152M_g^2L_f^2}{\sigma^4} \sup_{k\geq 0} \frac{\eta_k^2}{\theta_k^2} + \frac{72L_f^2 T}{\sigma^2}\sup_{k\geq 0} \eta_k. 
			%     \end{aligned}
		% \end{equation*}
	% Here the second inequality uses \eqref{Eq_Le_accumulation_gy_joint_nondeg_0}, while the final inequality uses  Proposition \ref{Prop_restrict_manifold_joint_nondeg_New}. 
	
	Moreover, by Lemma \ref{Le_esti_gy_joint_nondeg}, $\lim_{k\to +\infty}p(\xk, \yk) = 0$, which illustrates that 
	\begin{equation}
		\label{Eq_Le_accumulation_gy_joint_nondeg_1}
		\limsup_{s\to +\infty} \left(p(x_s, y_s) - p(x_{\Lambda(\lambda(s) + T) +1}, y_{\Lambda(\lambda(s) + T) +1})  \right) = 0 
	\end{equation}
	holds for any $T > 0$. 
	Therefore, we have 
	\begin{equation*}
		\begin{aligned}
			&0\leq \limsup_{s\to +\infty}  \sum_{k = s}^{\Lambda(\lambda(s) + T)} \theta_k \norm{\nabla_y g(\xk, \yk) }^2\\
			\leq{}& \limsup_{s\to +\infty} \sum_{k = s}^{\Lambda(\lambda(s) + T)} \left( \frac{8}{\sigma^2} \left( p(\xk, \yk) - p(\xkp, \ykp) \right) + \frac{\revise{144} L_f^2}{\sigma^2} \cdot \frac{\eta_k^2}{\theta_k} \right)\\
			\leq{}& \frac{8}{\sigma^2}\limsup_{s\to +\infty} \left(p(x_s, y_s) - p(x_{\Lambda(\lambda(s) + T) +1}, y_{\Lambda(\lambda(s) + T) +1})  \right) + \frac{\revise{144} L_f^2}{\sigma^2} \limsup_{s\to +\infty} \sum_{k = s}^{\Lambda(\lambda(s) + T)}  \frac{\eta_k^2}{\theta_k}\\
			={}& \frac{\revise{144} L_f^2}{\sigma^2}\limsup_{s\to +\infty} \sum_{k = s}^{\Lambda(\lambda(s) + T)}  \frac{\eta_k}{\theta_k} \cdot \eta_k = 0.
		\end{aligned}
	\end{equation*}
	Here the second inequality uses \eqref{Eq_Le_accumulation_gy_joint_nondeg_0}, and the first equality directly follows from \eqref{Eq_Le_accumulation_gy_joint_nondeg_1}. Additionally, the final equality uses the fact that $\lim_{k\to +\infty} \frac{\eta_k}{\theta_k} = 0$. 
	This completes the proof. 
\end{proof}

Next we consider the auxiliary sequence $\{(\wk, \zk)\} := \{\A(\xk, \yk)\}$, and let 
\begin{eqnarray*}
	r_k &:=& \frac{1}{\theta_k^2} \Big((\wkp, \zkp) -   (\wk, \zk) + \nabla \A(\xk, \yk) (\eta_k \mkp + \theta_k u_k + \theta_k e_k ) \Big),
	\\
	\xi_{k+1} &:=& \frac{1}{\eta_k}\left(\theta_k \nabla \A(\xk, \yk) u_k +\theta_k \nabla \A(\xk, \yk) e_k - \theta_k^2 r_k\right).
\end{eqnarray*}
Then from the update scheme of $\{(\xk, \yk)\}$ in \eqref{Eq_Subroutine_manifold}, we can conclude that for any $\beta > 0$, the auxiliary sequence $\{(\wk, \zk)\}$ follows the update scheme, 
\begin{equation}
	\label{Eq_update_auxiliary}
	\begin{aligned}
		&(\wkp, \zkp) = (\wk, \zk) -  \nabla \A(\xk, \yk) (\eta_k \mkp + \theta_k u_k + \theta_k e_k ) + \theta_k^2 r_k\\
		={}& (\wk, \zk) - \eta_k \nabla \A(\xk, \yk)\mkp - \theta_k \nabla \A(\xk, \yk) u_k - \theta_k \nabla \A(\xk, \yk) e_k + \theta_k^2 r_k \\
		={}& (\wk, \zk) - \eta_k \underbrace{\left(\nabla \A(\xk, \yk)\mkp + \beta \nabla p(\xk, \yk)\right)}_{\text{estimations for $\D_{h_{\beta}}(\xk, \yk)$}} -\eta_k\underbrace{ \left(\xi_{k+1} -  \beta \nabla p(\xk, \yk)\right)}_{\text{error terms}}. 
	\end{aligned}
\end{equation}

In the following proposition, we show that with appropriate initialization and stepsizes, the error terms in \eqref{Eq_update_auxiliary} can be controlled by the stepsizes $\{\eta_k\}$. 
\begin{prop}
	\label{Prop_local_convergence_noise_controll}
	Suppose Assumption \ref{Assumption_joint_nondegeneracy}, Assumption \ref{Assumption_f}, and Assumption \ref{Assumption_alg_smooth} hold. For any sequence $\{(\xk, \yk)\}$ generated by \eqref{Eq_Subroutine_manifold} with $(x_0, y_0) \in \{(x, y): \norm{x} + \norm{y} \leq M_{\mathrm{iter}},~ \mathrm{dist}\left((x, y), \M \right) \leq \frac{24 L_f}{\sigma^2}\sup_{k\geq 0}\frac{\eta_k}{\theta_k}\}$, and the sequences of stepsizes $\{\eta_k\}$ and $\{\theta_k\}$  satisfy \eqref{Eq_Cond_Stepsizes}, it holds for any $T > 0$ that 
	\begin{equation*}
		\limsup_{s \to +\infty} \sup_{s \leq j \leq \Lambda(\lambda(s) + T)}\norm{ \sum_{k = s}^{j} \eta_k \left(\xi_{k+1} - \beta \nabla p(\xk, \yk)\right) } = 0.
	\end{equation*}
\end{prop}
\begin{proof}
	Firstly, from Lemma \ref{Le_welldef_penalty_joint_nondeg} and the differentiability of $\nabla_y g$, we can conclude that $\A$ is differentiable and $\nabla \A$ is locally Lipschitz continuous over $\Omega$. Together with the uniform boundedness of $\{(\xk, \yk)\}$ and \eqref{Eq_Cond_Stepsizes}, we can conclude that 
	\begin{equation*}
		\norm{r_k} \leq 3M_{A} \left(\norm{u_k}^2 + \norm{e_k}^2 + \frac{\eta_k^2}{\theta_k^2} \norm{\revise{\mkp}}^2 \right)
		\leq \revise{3M_A\left(L_p^2+M_e^2\rho^4+
			\frac{\rho^2\sigma^4}{128^2(M_g+1)^2}\right)}
	\end{equation*}
	Then with the definition of $P_{\A}$ and $R_{\A}$ in \eqref{Eq_defin_PA_RA} and Lemma \ref{Le_error_bound_cdf}, it holds that 
	
	\begin{equation*}
		\begin{aligned}
			&\norm{\nabla \A(\xk, \yk) u_k } = \norm{P_{\A}(\xk, \yk) u_k + R_{\A}(\xk, \yk)u_k} \\
			={}& \norm{ R_{\A}(\xk, \yk)u_k} \leq \frac{8M_g M_{g, 3}}{\sigma^2}  \norm{\nabla_y g(\xk, \yk)}^2. 
		\end{aligned}
	\end{equation*}
	Therefore, for any \revise{$k\geq 0$}, it holds that 
	\begin{equation*}
		\eta_k \norm{\xi_{k+1}} \leq \theta_k \left(\frac{8M_g\revise{M_{g,3}}}{\sigma^2} + L_A M_e\right)\norm{\nabla_y g(\xk, \yk)}^2 + \theta_k^2  \norm{r_k}. 
	\end{equation*}
	As a result, for any $T> 0$, we have  
	\begin{equation}
		\label{Eq_Prop_local_convergence_noise_controll_0}
		\begin{aligned}
			&\limsup_{s \to +\infty} \sup_{s\leq j\leq \Lambda(\lambda(s) + T)} \norm{\sum_{k = s}^j \eta_k \xi_{k+1}} \leq \limsup_{s \to +\infty} \sum_{k = s}^{\Lambda(\lambda(s) + T)}\eta_k \norm{\xi_{k+1}} \\
			\leq{}& \limsup_{s \to +\infty} \sum_{k = s}^{\Lambda(\lambda(s) + T)}  \left(  \theta_k \left(\frac{8M_g M_{g, 3}}{\sigma^2} + L_A M_e\right)\norm{\nabla_y g(\xk, \yk)}^2 + \theta_k^2 \norm{r_k} \right)
			\;=\; 0. 
		\end{aligned}
	\end{equation}
	Here the last equality follows from Lemma \ref{Le_accumulation_gy_joint_nondeg_New} and the facts that $\lim_{k \to +\infty} \frac{\theta_k^2}{\eta_k} = 0$ and $\lim_{k \to +\infty} \theta_k = 0$ 
	in Assumption \ref{Assumption_alg_smooth}(2). 
	
	Furthermore, Lemma \ref{Le_esti_gy_joint_nondeg} illustrates that $\lim_{k\to +\infty} \norm{\nabla_y g(\xk, \yk)} = 0$, hence $\lim_{k\to +\infty} \norm{\nabla p(\xk, \yk)} = 0$. 
	Notice that the definition of $\Lambda(t)$ and $\lambda(k)$ in Section \ref{Subsection_basic_notation} illustrates that $ \sum_{k = s}^{\Lambda(\lambda(s) + T)}  \eta_k \leq T+ \sup_{k\ge0}\eta_k \leq T+1$. 
	As a result, it holds for any $\beta > 0$ that 
	\begin{equation}
		\label{Eq_Prop_local_convergence_noise_controll_1}
		\begin{aligned}
			&\limsup_{s \to +\infty} \sup_{s\leq j\leq \Lambda(\lambda(s) + T)} \norm{\sum_{k = s}^j \eta_k \beta \nabla p(\xk, \yk)} 
			\leq{} \limsup_{s \to +\infty} \sum_{k = s}^{\Lambda(\lambda(s) + T)}  \eta_k \beta \norm{\nabla p(\xk, \yk)} \\
			{}&\leq  \limsup_{s \to +\infty}  \sup_{s\leq k\leq \Lambda(\lambda(s) + T)} \beta (T+1)\norm{\nabla p(\xk, \yk)} \\
			={}& \limsup_{k\to +\infty} \beta (T+1)\norm{\nabla p(\xk, \yk)}= 0. 
		\end{aligned}
	\end{equation}

	As a result, by combining \eqref{Eq_Prop_local_convergence_noise_controll_0} and \eqref{Eq_Prop_local_convergence_noise_controll_1} together, we can conclude that
	\begin{equation*}
		\begin{aligned}
			&\limsup_{s \to +\infty} \sup_{s \leq j \leq \Lambda(\lambda(s) + T)}\norm{ \sum_{k = s}^{j} \eta_k \left(\xi_{k+1} - \beta \nabla p(\xk, \yk)\right) } \\
			\leq{}& \limsup_{s \to +\infty} \sup_{s\leq j\leq \Lambda(\lambda(s) + T)} \norm{\sum_{k = s}^j \eta_k \xi_{k+1}}  + \limsup_{s \to +\infty} \sup_{s\leq j\leq \Lambda(\lambda(s) + T)} \norm{\sum_{k = s}^j \eta_k \beta \nabla p(\xk, \yk)}=0. 
		\end{aligned}
	\end{equation*}
	This completes the proof. 
\end{proof}

Moreover, the following proposition illustrates that the term $\nabla \A(\xk, \yk)m_{k+1} + \beta \nabla p(\xk, \yk)$ can be viewed as an
approximated evaluation of $\D_{h_{\beta}}(\wk, \zk)$. 
\begin{prop}
	\label{Prop_local_convergence_Dh_esti}
	Suppose Assumption \ref{Assumption_joint_nondegeneracy}, Assumption \ref{Assumption_f}, and Assumption \ref{Assumption_alg_smooth} hold. For any sequence $\{(\xk, \yk)\}$ generated by \eqref{Eq_Subroutine_manifold} with $(x_0, y_0) \in \Omegajoint$, $\norm{m_0} \leq L_f$, and any sequences of stepsizes $\{\eta_k\}$ and $\{\theta_k\}$ satisfying \eqref{Eq_Cond_Stepsizes}, there exists $\{\hat{\delta}_k\}$ such that $\lim_{k\to +\infty} \hat{\delta}_k = 0$ and 
	\begin{equation*}
		\nabla \A(\xk, \yk) \mkp + \beta \nabla p(\xk, \yk) \in \D_{h_{\beta}}^{\hat{\delta}_k}(\wk, \zk). 
	\end{equation*}
\end{prop}
\begin{proof}
	From Lemma \ref{Le_esti_gy_joint_nondeg}, we have that $\lim_{k\to +\infty} \norm{\nabla_y g(\xk, \yk)} = 0$, hence 
	\begin{equation}
		\label{Eq_Prop_local_convergence_Dh_esti_0}
		\lim_{k\to +\infty} \norm{\nabla p(\xk, \yk)} = 0, \quad \lim_{k\to +\infty} \norm{(\xk, \yk) - (\wk, \zk)} = 0. 
	\end{equation}
	Moreover, Proposition \ref{Prop_dk_esti_UB} illustrates that there exists $\{\delta_k^{\star}\}$ such that $\lim_{k\to +\infty} \delta_k^{\star} = 0$ and $\mkp \in \D_{f}^{\delta_k^\star} (\A(\wk, \zk))$. Notice that $\D_{h_{\beta}}(\wk, \zk) = \nabla \A(\wk,\zk) \D_f(\A(\wk, \zk)) + \beta \nabla p(\wk, \zk)$, we have
	\begin{small}
		\begin{equation*}
			\begin{aligned}
				&\mathrm{dist}\left( \nabla \A(\xk, \yk) \mkp + \beta \nabla p(\xk, \yk), \D_{h_{\beta}}^{\delta_k^\star}(\wk, \zk) \right)\\
				\leq{}& \mathrm{dist}\left( \nabla \A(\xk, \yk) \mkp, \nabla \A(\wk,\zk)\D_f^{\delta_k^\star}(\A(\wk, \zk)) \right) + \beta \left( \norm{\nabla p(\xk, \yk)} +  \norm{\nabla p(\wk, \zk)} \right)\\
				\leq{}& 
				\left\| \nabla \A(\xk, \yk) \mkp- \nabla \A(\wk,\zk)\mkp \right\| + \norm{\nabla \A(\wk, \zk)}\delta_k^\star \\
				&+ \beta \left( \norm{\nabla p(\xk, \yk)} +  \norm{\nabla p(\wk, \zk)} \right)
				\\
				\leq{}& L_f\norm{\nabla \A(\wk, \zk) - \nabla \A(\xk, \yk)} + \norm{\nabla \A(\wk, \zk)}\delta_k^\star + \beta \left( \norm{\nabla p(\xk, \yk)} +  \norm{\nabla p(\wk, \zk)} \right).
			\end{aligned}
		\end{equation*}
	\end{small}
	As a result, by choosing 
	\begin{equation*}
		\hat{\delta}_k = L_f\norm{\nabla \A(\wk, \zk) - \nabla \A(\xk, \yk)} + \norm{\nabla \A(\wk, \zk)}\delta_k^\star + \beta \left( \norm{\nabla p(\xk, \yk)} +  \norm{\nabla p(\wk, \zk)} \right),
	\end{equation*}
	it holds that 
	\begin{equation*}
		\nabla \A(\xk, \yk) \mkp + \beta \nabla p(\xk, \yk) \in \D_{h_{\beta}}^{\hat{\delta}_k}(\wk, \zk).
	\end{equation*}
	In addition, from the uniform boundedness of $\{(\xk, \yk)\}$ in Assumption \ref{Assumption_alg_smooth}(1), the Lipschitz continuity of $\nabla \A$, and \eqref{Eq_Prop_local_convergence_Dh_esti_0}, we can conclude that $\lim_{k\to +\infty} \hat{\delta}_k = 0$. This completes the proof. 
\end{proof}

\begin{theo}
	\label{Theo_local_convergence_NS_joint_nondeg}
	Suppose Assumption \ref{Assumption_joint_nondegeneracy}, Assumption \ref{Assumption_f}, and Assumption \ref{Assumption_alg_smooth} hold. For any sequence $\{(\xk, \yk)\}$ generated by \eqref{Eq_Subroutine_manifold} with $(x_0, y_0) \in \Omegajoint$, and any sequences of stepsizes $\{\eta_k\}$ and $\{\theta_k\}$ satisfying \eqref{Eq_Cond_Stepsizes}, we have that any cluster point of $\{(\xk, \yk)\}$ is a first-order stationary point of \eqref{Prob_BLO}.
\end{theo}
\begin{proof}
	From Proposition \ref{Prop_local_convergence_Dh_esti}, there exists a diminishing sequence $\{\hat{\delta}_k\}$ such that the auxiliary sequence $\{(\wk, \zk)\}$ can be characterized by the following update scheme,  
	\begin{equation}
		\label{Eq_Theo_local_convergence_NS_joint_nondeg_0}
		(\wkp, \zkp) \in (\wk, \zk) - \eta_k \D_{h_{\beta}}^{\hat{\delta}_k}(\wk, \zk) -\eta_k \left(\xi_{k+1} -  \beta \nabla p(\xk, \yk)\right),
	\end{equation}
	which corresponds to the differential inclusion 
	\begin{equation}
		\label{Eq_Theo_local_convergence_NS_joint_nondeg_DI}
		\left( \frac{\mathrm{d}w}{\mathrm{d}t},  \frac{\mathrm{d}z}{\mathrm{d}t} \right) \in - \D_{h_{\beta}}(w, z). 
	\end{equation}
	
	From Proposition \ref{Prop_restrict_manifold_joint_nondeg_New} and Lemma \ref{Le_esti_gy_joint_nondeg}, the sequence $\{(\xk, \yk)\}$ is restricted to within $\Omega$ and eventually converges towards $\M$. Then it holds that the auxiliary sequence $\{(\wk, \zk)\}$ also converges towards $\M$.  
	As demonstrated in Lemma \ref{Le_DI_cdf}, the differential inclusion \eqref{Eq_Theo_local_convergence_NS_joint_nondeg_DI} admits $h_{\beta}$ as its Lyapunov function with stable set $\{(w, z) \in \Omega: 0 \in \D_{h_{\beta}}(w, z)\}$.

	Next we show that the interpolated sequence of $\{(\wk, \zk)\}$ is a perturbed solution of the differential inclusion \eqref{Eq_Theo_local_convergence_NS_joint_nondeg_DI} through Lemma \ref{Le_interpolated_process}. The uniform boundedness of $\{(\xk, \yk)\}$ and the local boundedness of $\A$ guarantees the uniform boundedness of $\{(\wk, \zk)\}$, which verifies the first condition in Lemma \ref{Le_interpolated_process}. Moreover, as $\lim_{k\to +\infty} \hat{\delta}_k = 0$, we can conclude the validity of the second condition of Lemma \ref{Le_interpolated_process}  based on \eqref{Eq_Theo_local_convergence_NS_joint_nondeg_0}. Additionally, Proposition \ref{Prop_local_convergence_noise_controll} illustrates that the sequence of error terms $\{\left(\xi_{k+1} -  \beta \nabla p(\xk, \yk)\right)\}$ in \eqref{Eq_Theo_local_convergence_NS_joint_nondeg_0} can be controlled, hence verifying the validity of the third condition of Lemma \ref{Le_interpolated_process}. As a result, we can conclude that the interpolated sequence $\{(\wk, \zk)\}$ is a perturbed solution of the differential inclusion \eqref{Eq_Theo_local_convergence_NS_joint_nondeg_DI}.

	Moreover, it is worth noting that $\A(w,z)=(w,z)$ whenever $\nabla_y g(w,z)=0$. Then from Proposition \ref{Prop_exactness_joint_nondeg} and the expression of $\D_{h_{\beta}}$, for any $\beta > \frac{32 L_f M_g M_{g,3}}{\sigma^4}$, it  holds that 
	\begin{small}
		\begin{equation*}
			\{h_{\beta}(w, z): (w, z) \in \Omega, 0 \in \D_{h_{\beta}}(w, z)\} = \{f(w, z): 0 \in P_{\A}(w, z) \D_f(w, z), \nabla_{y} g(w, z) = 0\}.
		\end{equation*}
	\end{small}
	Therefore, the critical values of $h_{\beta}$ have an empty interior in $\bb{R}$ by Assumption \ref{Assumption_f}. Together with the path-differentiability of $h_{\beta}$ in Lemma \ref{Le_conservative_field_joint_nondeg}, and \cite[Proposition 3.27]{benaim2005stochastic}, we can conclude that any cluster point of $\{(\wk, \zk)\}$ lies in $\{(w, z) \in \Rn \times \Rp: 0 \in \D_{h_{\beta}}(w,z), \nabla_{y} g(w, z) = 0\}$.

	Furthermore, with Proposition \ref{Prop_exactness_joint_nondeg}, it holds that 
	\begin{equation*}
		\{(w, z) \in \Omega:  0 \in \D_{h_{\beta}}(w, z)\} = \{(w, z) \in \Rn \times \Rp: 0 \in P_{\A}(w, z) \D_f(w, z), \nabla_{y} g(w, z) = 0\}.
	\end{equation*}
	Therefore, any cluster point of $\{(\xk, \yk)\}$ 
	% lies in $\{(w, z) \in \Rn \times \Rp: 0 \in P_{\A}(w, z) \D_f(w, z), \nabla_{y} g(w, z) = 0\}$, hence 
	is a first-order stationary point of \eqref{Prob_Con}, and hence is a first-order stationary point of  \eqref{Prob_BLO}. This completes the proof. 
\end{proof}

\begin{rmk}

	It is worth mentioning that, based on \revise{the global stability result in \cite[Proposition 3.6]{xiao2023globalstability}}, we can also prove that the sequence of iterates $\{(x_k, y_k)\}$ generated by \eqref{Eq_Subroutine_manifold} is uniformly bounded \revise{when $h_\beta$ satisfies the conditions in Assumption \ref{Assumption_global_stability}.} More precisely, \revise{under these conditions,} for any $M>0$ there exists $\alpha_{\mathrm{step}}>0$ such that, whenever $\norm{x_0}+\norm{y_0}\leq M$, $\mathrm{dist}((x_0,y_0),\mathcal{M})\leq \alpha_{\mathrm{step}}$, and $\sup_{k\geq 0}\max\left\{\eta_k,\theta_k,\frac{\eta_k}{\theta_k},\frac{\theta_k^2}{\eta_k}\right\}\leq \alpha_{\mathrm{step}}$, 
	the sequence $\{(x_k,y_k)\}$ generated by \eqref{Eq_Subroutine_manifold} remains uniformly bounded. That is, there exists $C_M>0$ such that $\norm{x_k}+\norm{y_k}\leq C_M$ for all $k\geq 0$. To clearly present the main results of this paper, we provide detailed proofs of the global stability analysis in Appendix \ref{Appendix_Global_Stability}.
\end{rmk}

\subsubsection{Globalization by a two-phase hybrid approach}
\label{Subsection_globalization}

In this part, we establish the global convergence of Algorithm \ref{Alg_TPHSD}. We begin our theoretical analysis with the following auxiliary lemma.

\begin{lem}
	\label{Le_diminish_ukek_joint_nondeg}
	Suppose Assumption \ref{Assumption_joint_nondegeneracy}, Assumption \ref{Assumption_f}, and Assumption \ref{Assumption_alg_smooth} hold. Then for any sequence of iterates $\{(\xk, \yk)\}$ generated by \eqref{Eq_Subroutine_gy} \revise{with $\norm{m_0} \leq L_f$}, it holds that $\liminf\limits_{k\to +\infty} \norm{\nabla_y g(\xk, \yk)} = 0$. 
\end{lem}
\begin{proof}
	From the update scheme of the subroutine \eqref{Eq_Subroutine_gy} and the differentiability of $g$, we have
	\begin{equation}
		\label{Eq_Le_diminish_ukek_joint_nondeg_0}
		\begin{aligned}
			&g(\xkp, \ykp) - g(\xk, \yk) \\
			\leq{}& -\inner{\nabla_x g(\xk, \yk), \eta_k m_{x, k+1}} - \inner{\nabla_y g(\xk, \yk), \eta_k m_{y, k+1} + \theta_k \nabla_y g(\xk, \yk)}  \\
			& + M_g \norm{\eta_k m_{x, k+1}}^2 + M_g \norm{\eta_k m_{y, k+1} + \theta_k \nabla_y g(\xk, \yk)}^2\\
			\leq{}& - \theta_k \norm{\nabla_y g(\xk, \yk)}^2  + \eta_k L_f \norm{\nabla_y g(\xk, \yk)} + \eta_k L_f L_g \\
			&+ 3\eta_k^2 M_g L_f^2 + 2\theta_k^2M_g \norm{\nabla_y g(\xk, \yk)}^2\\
			\leq{}& - \left( \theta_k - 2\theta_k^2 M_g \right) \norm{\nabla_y g(\xk, \yk)}^2+ \left(2\eta_k L_g + 3\eta_k^2 M_gL_f \right) L_f. 
			% \leq{}& - \frac{\theta_k}{4}\norm{\nabla_y g(\xk, \yk)}^2 + \eta_k(L_f + M_g L_f + L_g)L_f. 
		\end{aligned}
	\end{equation}
	Now we prove this lemma by contradiction, that is,  we assume that there exists $\varepsilon > 0$ such that  $\liminf_{k\to +\infty} \norm{\nabla_y g(\xk, \yk)} =\varepsilon$. Thus there exists $K_1 > 0$ such that $\inf_{k\geq K_1} \norm{\nabla_y g(\xk, \yk)} \geq \frac{\varepsilon}{2}$, $\sup_{k\geq K_1} \eta_k \leq \frac{1}{3M_g}$, $\sup_{k\geq K_1} \theta_k \leq \frac{1}{4M_g}$, and $\sup_{k\geq K_1} \frac{\eta_k}{\theta_k} \leq \frac{\varepsilon^2}{16(2L_g + L_f)L_f}$.
	Then from \eqref{Eq_Le_diminish_ukek_joint_nondeg_0}, it holds for any $k\geq K_1$ that 
	\begin{equation*}
		\begin{aligned}
			&g(\xkp, \ykp) - g(\xk, \yk) \\
			\leq{}& - \left(\theta_k - 2\theta_k^2 M_g \right) \norm{\nabla_y g(\xk, \yk)}^2+ \left(2\eta_k L_g + 3\eta_k^2 M_gL_f \right) L_f\\
			\leq{}& - \frac{\theta_k}{2}\norm{\nabla_y g(\xk, \yk)}^2 + \eta_k(2L_g + L_f)L_f.
		\end{aligned}
	\end{equation*}
	As a result, for any $k > K_1$, it holds that 
	\begin{equation*}
		\begin{aligned}
			&g(x_{k+1}, y_{k+1}) - g(x_{K_1}, y_{K_1}) \\
			\leq{}& \sum_{i = K_1}^k  - \frac{\theta_i}{2}\norm{\nabla_y g(x_i, y_i)}^2 + \eta_i(2L_g + L_f)L_f \leq  \sum_{i = K_1}^k  - \frac{\varepsilon^2}{16}\theta_i.
		\end{aligned}
	\end{equation*}
	This leads to the following inequality 
	\begin{equation*}
		\liminf_{k \to +\infty} g(x_k, y_k) \leq \liminf_{k\to +\infty }g(x_{K_1}, y_{K_1}) + \sum_{i = K_1}^k - \frac{\varepsilon^2}{16}\theta_i = -\infty,
	\end{equation*}
	which contradicts the fact that $\{(\xk, \yk)\}$ is 
	bounded. Therefore, we can conclude that $\liminf_{k\to +\infty} \norm{\nabla_y g(\xk, \yk)} = 0$. This completes the proof. 
\end{proof}

\begin{lem}
	\label{Le_error_esti_ukek_joint_nondeg}
	Suppose Assumption \ref{Assumption_joint_nondegeneracy}, Assumption \ref{Assumption_f}, and Assumption \ref{Assumption_alg_smooth} hold. Then for any $(\xk, \yk) \in \Omegajoint$, it holds that $\norm{u_k + e_k} \geq \frac{\sigma}{4} \norm{\nabla_y g(\xk, \yk)}$. 
\end{lem}
\begin{proof}
	From Assumption \ref{Assumption_alg_smooth}, it holds that 
	\begin{equation*}
		\norm{u_k + e_k} \geq \norm{u_k} - \norm{e_k} \geq \frac{\sigma}{2} \norm{\nabla_y g(\xk, \yk)} - M_e \norm{\nabla_y g(\xk, \yk)}^2 \geq \frac{\sigma}{4} \norm{\nabla_y g(\xk, \yk)}.
	\end{equation*}
	This completes the proof. 
\end{proof}

\begin{lem}
	\label{Le_Subroutine_manifold_nonincrease}
	Suppose Assumption \ref{Assumption_joint_nondegeneracy}, Assumption \ref{Assumption_f}, and Assumption \ref{Assumption_alg_smooth} hold. Then with any $\tilde{\varepsilon}> 0$, for any sequence $\{(\xk, \yk)\}$ generated by \eqref{Eq_Subroutine_manifold} \revise{with $\norm{m_0} \leq L_f$} such that 
	\begin{equation*}
		\begin{aligned}
			&\inf_{k\geq 0}\norm{u_k + e_k} \geq \tilde{\varepsilon} \norm{\nabla_y g(\xk, \yk)}, \quad \sup_{k\geq 0} \eta_k \leq \frac{\rho}{8M_g L_f}, \\
			&\sup_{k\geq 0}\theta_k \leq \min\left\{\frac{1}{16M_p}, \frac{\rho}{8 M_g^2 L_g + 8M_e M_g L_g^2} \right\},\quad \sup_{k\geq 0}\frac{\eta_k}{\theta_k} \leq \frac{\tilde{\varepsilon}^2 \rho}{16L_f},
		\end{aligned}
	\end{equation*}
	it holds that 
	$\liminf_{k\to +\infty} \mathrm{dist}\left((\xk, \yk) , \revise{\M} \right) = 0. 
	$
\end{lem}
\begin{proof}
	Let $d_k = (d_{x, k}, d_{y, k})$, $\mk = (m_{x, k}, m_{y, k})$, $u_k = (u_{x, k}, u_{y, k})$ and $e_k = (e_{x, k}, e_{y, k})$. From Assumption \ref{Assumption_alg_smooth}(3), there exists $K > 0$ such that for any $k\geq K$, $\norm{u_k} \geq \frac{\tilde{\varepsilon}}{2} \norm{\nabla_y g(\xk, \yk)}$. 
	Then from the update scheme of \eqref{Eq_Subroutine_manifold} and Proposition \ref{Prop_restrict_manifold_joint_nondeg_New}, it holds that 
	\begin{equation}
		\label{Eq_Le_Subroutine_manifold_nonincrease_0}
		% \footnotesize
		\begin{aligned}
			&p(\xkp, \ykp) - p(\xk, \yk) \\
			\leq{}& \inner{\nabla p(\xk, \yk), -\eta_k \mkp - \theta_k  (u_k + e_k)} + \frac{M_p}{2} \norm{\eta_k \mkp + \theta_k  (u_k + e_k)}^2\\
			\leq{}& \inner{\nabla p(\xk, \yk), -\eta_k \mkp - \theta_k  (u_k + e_k)} \\
			&+ M_p \eta_k^2 \norm{\mkp}^2 + 2M_p \theta_k^2 \norm{u_k}^2 + 2M_p \theta_k^2 \norm{e_k}^2\\
			\leq{}& \inner{\nabla p(\xk, \yk),  - \theta_k  u_k} + \left(M_p \eta_k^2 + \frac{8\eta_k^2}{\theta_k}\right) \norm{\mkp}^2 \\
			&+ \left(2M_p \theta_k^2 +8 \theta_k\right)\norm{e_k}^2 + \frac{\theta_k}{4} \norm{\nabla p(\xk, \yk)}^2 \\
			\leq{}&- \frac{3\theta_k}{4} \norm{\nabla p(\xk, \yk)}^2 + 16\theta_k \omega_k^2 M_e^2 \norm{\nabla_y g(\xk, \yk)}^4  + \frac{16\eta_k^2}{ \theta_k}  L_f^2\\
			\leq{}& - \frac{3\theta_k \tilde{\varepsilon}^2}{16} \norm{\nabla_y g(\xk, \yk)}^2 + 16\theta_k \omega_k^2 M_e^2 \norm{\nabla_y g(\xk, \yk)}^4 +  \frac{16\eta_k^2}{ \theta_k}  L_f^2. 
		\end{aligned}
	\end{equation}
	
	Now we prove this lemma by contradiction. That is, we assume that there exists $\kappa > 0$ such that $\liminf_{k\to +\infty} \norm{\nabla_{y} g(\xk, \yk)} = \kappa$. Then from Assumption \ref{Assumption_alg_smooth}, there exists $K > 0$ such that $\sup_{k\geq K} \omega_k \leq \frac{\tilde{\varepsilon}}{32M_e L_g}$, $\sup_{k\geq K} \frac{\eta_k}{\theta_k} \leq \frac{\tilde{\varepsilon} \kappa}{64 L_f}$ and $\inf_{k\geq K} \norm{\nabla_{y} g(\xk, \yk)} \geq \frac{\kappa}{2}$.
	% \begin{equation*}
		%     \sup_{k\geq K} \omega_k \leq \frac{\tilde{\varepsilon}}{32M_e L_g}, \quad \sup_{k\geq K} \frac{\eta_k}{\theta_k} \leq \frac{\tilde{\varepsilon} \kappa}{64 L_f}, \quad \inf_{k\geq K} \norm{\nabla_{y} g(\xk, \yk)} \geq \frac{\kappa}{2}.
		% \end{equation*}
	Then, for any $k\geq K$, it holds that $p(\xkp, \ykp) - p(\xk, \yk)\leq - \frac{\theta_k \tilde{\varepsilon}^2 \kappa^2}{32}$. 
	Therefore, we can conclude that 
	\begin{equation*}
		\liminf_{k \to +\infty} p(\xk, \yk) \leq \liminf_{k \to +\infty} p(x_K, y_K) -\frac{\tilde{\varepsilon}^2 \kappa^2}{32} \sum_{i = K}^k \theta_i = -\infty,
	\end{equation*}
	which contradicts the fact that $p(x, y) \geq 0$ holds for any $(x, y) \in \Rn \times \Rp$. As a result, we can conclude that $\liminf_{k\to +\infty} \norm{\nabla_{y} g(\xk, \yk)} = 0$, and thus the uniform boundedness of $\{(\xk, \yk)\}$ guarantees that $\liminf_{k\to +\infty} \mathrm{dist}\left((\xk, \yk) , \revise{\M} \right) = 0$. This completes the proof. 
\end{proof}

Now we define $\tau: \bb{N}_+ \to \bb{N}_+$ such that $\tau(i)$ equals the $i$-th smallest element in $\ca{S}_{\mathrm{idx}}$ in Algorithm \ref{Alg_TPHSD}. Moreover, we denote $\tau_{o}(i) := \tau(2i+1)$ and $\tau_{e}(i) := \tau(2i)$ for all $i \geq 0$. Then for any $i \geq 0$, the $i$-th transition from \eqref{Eq_Subroutine_gy} to \eqref{Eq_Subroutine_manifold} occurs at $\tau_e(i)$ iteration of  Algorithm \ref{Alg_TPHSD}. Similarly, for any $i \geq 0$, the $i$-th transition from \eqref{Eq_Subroutine_manifold} to \eqref{Eq_Subroutine_gy} occurs at $\tau_{o}(i)$-th iteration. It is worth mentioning that when $\tau(2i+2) = +\infty$ for a certain $i \geq  0$, then the update scheme of $\{(\xk, \yk)\}$ follows \eqref{Eq_Subroutine_gy} for any $k > \tau(2i+1)$.
\begin{prop}
	\label{Prop_finite_terminiate_gy_joint_nondeg}
	Suppose Assumption \ref{Assumption_f}, Assumption \ref{Assumption_joint_nondegeneracy}, and Assumption \ref{Assumption_alg_smooth} hold. Then for any $i \in \bb{N}_+$, it holds that $\tau(2i + 2) - \tau(2i+1) < +\infty$. 
\end{prop}
\begin{proof}
	We prove this proposition by contradiction, in the sense that we assume that $\tau(2i + 2) - \tau(2i+1) = +\infty$  for a certain $i \geq 0$. As a result, Lemma \ref{Le_diminish_ukek_joint_nondeg} illustrates that there exists $k > \tau(2i+1)$ such that $\norm{\nabla_y g(\xk, \yk)} \leq \varepsilon_k$. This implies that $\tau(2i+2) \leq k$, which leads to a contradiction to our assumption. This completes the proof. 
\end{proof}

\begin{prop}
	\label{Prop_stable_joint_nondeg}
	Suppose Assumption \ref{Assumption_joint_nondegeneracy}, Assumption \ref{Assumption_f}, and Assumption \ref{Assumption_alg_smooth} hold. Then for the sequence $\{ \varepsilon_k\}$ in Algorithm \ref{Alg_TPHSD}, it holds that $\lim_{k \to +\infty} \varepsilon_k > 0$.  
\end{prop}
\begin{proof}
	We prove this proposition by contradiction, that is, we assume $\varepsilon_k \to 0$. Then the update scheme in Algorithm \ref{Alg_TPHSD} illustrates that $|\ca{S}_{\mathrm{idx}}|$ is infinite. 
	From Assumption \ref{Assumption_alg_smooth}, there exists $K > 0$ such that $\sup_{k\geq K} \eta_k \leq \frac{\rho}{8M_g L_f}$, $\sup_{k\geq K}\theta_k \leq \min\left\{\frac{1}{M_p + M_g}, \frac{\rho}{8 M_g^2 L_g + 8M_e M_g L_g^2} \right\}$, and $\sup_{k\geq K}\frac{\eta_k}{\theta_k} \leq \frac{\sigma^2 \rho}{32L_fM_g}$.
	% \begin{equation}
		%     \sup_{k\geq K} \eta_k \leq \frac{\rho}{8M_g L_f}, \quad \sup_{k\geq K}\theta_k \leq \min\left\{\frac{1}{M_p + M_g}, \frac{\rho}{8 M_g^2 L_g + 8M_e M_g L_g^2} \right\},\quad \sup_{k\geq K}\frac{\eta_k}{\theta_k} \leq \frac{\sigma^2 \rho}{32L_fM_g}. 
		% \end{equation}
	
	Moreover, from the definition of $\tau_{e}$, it holds for any $i\geq 1$ that $\varepsilon_{\tau_e(i+1)} = \frac{\varepsilon_{\tau_e(i)}}{2}$ and $\norm{\nabla_y g(x_{\tau_e(i)}, y_{\tau_e(i)})} \leq \varepsilon_{\tau_e(i)}$. 
	% \begin{equation*}
		%     \varepsilon_{\tau_e(i+1)} = \frac{\varepsilon_{\tau_e(i)}}{2}, \quad \norm{\nabla_y g(x_{\tau_e(i)}, y_{\tau_e(i)})} \leq \varepsilon_{\tau_e(i)}.
		% \end{equation*}
	Therefore, there exists $i^* > 0$ such that $\tau_e(i^*) > K$ and $(x_{\tau_e(i^*)}, y_{\tau_e(i^*)}) \in \Omegajoint$. Then from Theorem \ref{Theo_local_convergence_NS_joint_nondeg} and Lemma \ref{Le_error_esti_ukek_joint_nondeg}, we can conclude that for any $k \geq \tau_e(i^*)$, $(\xk, \yk) \in \Omegajoint$ and $\norm{\nabla_y g(\xk, \yk)} \leq \frac{4}{\sigma} \norm{u_k + e_k}$, hence the switching from \eqref{Eq_Subroutine_manifold} to \eqref{Eq_Subroutine_gy} would never occur for any $k\geq \tau_e(i^*)$.  Therefore, we can conclude that $\tau_{o}(i^*) = \tau_e(i^*+1) = +\infty$, which leads to $\lim_{k \to +\infty} \varepsilon_k = \varepsilon_{\tau_e(i^*)} > 0$. This contradicts the assumption that $\lim_{k \to +\infty} \varepsilon_k = 0$, hence completing the proof.  
\end{proof}

Now let $\varepsilon_{\min} := \liminf_{k\to +\infty} \varepsilon_k$, then we are ready to present the proof for Theorem \ref{Theo_convergence_NS_joint_nondeg}.

\begin{proof}[\bf Proof for Theorem \ref{Theo_convergence_NS_joint_nondeg}]
	Let $K_0$ be the last finite element of $\ca{S}_{\rm idx}$, then Assumption \ref{Assumption_alg_smooth} illustrates that there exists $K > K_0$ such that $\sup_{k\geq K} \eta_k \leq \frac{\rho}{4M_g L_f}$, $\sup_{k\geq K}\theta_k \leq \min\{\frac{1}{M_p}, \frac{\rho}{4 M_g^2 L_g} \}$, and $\sup_{k\geq K}\frac{\eta_k}{\theta_k} \leq \frac{\min\{\sigma, \varepsilon_{\min}\} \rho}{16L_f}$. 
	Then Lemma \ref{Le_Subroutine_manifold_nonincrease} illustrates that there exists $K_1 > K$ such that $(x_{K_1}, y_{K_1}) \in \Omegajoint$. Therefore, for the iterates $\{(\xk, \yk): k\geq K_1\}$, it follows the update scheme \eqref{Eq_Subroutine_manifold}. Moreover, Proposition \ref{Prop_restrict_manifold_joint_nondeg_New} illustrates that $\{(\xk, \yk): k\geq K_1\} \subseteq \Omegajoint$.  Then, together with Theorem \ref{Theo_local_convergence_NS_joint_nondeg}, we can conclude that any cluster point of $\{(\xk, \yk)\}$ is a first-order stationary point of \eqref{Prob_Con}. This completes the proof. 
\end{proof}

\begin{rmk}
	It is worth mentioning that we can explicitly bound the total number of switches between \eqref{Eq_Subroutine_manifold} and \eqref{Eq_Subroutine_gy} in Algorithm \ref{Alg_TPHSD}. According to the adaptive mechanism in Algorithm \ref{Alg_TPHSD}, every transition from \eqref{Eq_Subroutine_manifold} to \eqref{Eq_Subroutine_gy} halves the tolerance $\varepsilon_k$. Moreover, as demonstrated in Theorem \ref{Theo_local_convergence_NS_joint_nondeg} and Proposition \ref{Prop_stable_joint_nondeg}, we can choose a sufficiently large $K$, such that the stepsizes satisfy \eqref{Eq_Cond_Stepsizes} and
    \revise{\begin{equation}
        \sup_{k\geq K} \eta_k \leq \frac{\rho}{4M_g L_f}, \;\; \sup_{k\geq K} \theta_k \leq \frac{\rho}{4M_g(L_p + M_e \rho^2)}, \;\;\sup_{k\geq K} \frac{\eta_k}{\theta_k} \leq \frac{\rho \sigma^2}{144M_gL_f}.
    \end{equation}}
    \revise{Then for any $ j\geq K$ such that $\norm{\nabla_y g(x_j, y_j)} \leq \hat{\gamma} := \min\{\frac{\rho \sigma}{4M_g}, \frac{\sigma^2 \rho}{32M_g^2}, \frac{\sigma}{4}\}$, it follows from Lemma \ref{Le_welldef_penalty_joint_nondeg} that $\mathrm{dist}((x_j, y_j), \M) \leq \min\{ \frac{\rho}{2M_g}, \frac{\sigma \rho}{16M_g^2} \}$, hence Proposition \ref{Prop_Appendix_restrict_manifold_joint_nondeg_New} illustrates that $\norm{u_k + e_k} \geq \frac{\sigma}{4} \norm{\nabla_y  g(\xk, \yk)}$ holds for any $k \geq j$.} Therefore, once the corresponding $\varepsilon_k \leq \revise{\hat{\gamma}}$, the condition for executing \eqref{Eq_Subroutine_manifold} is always satisfied, and no further switching occurs. Consequently, $\varepsilon_k$ is halved for at most $\max\{0, \lceil \log_2(\frac{\varepsilon_0}{\revise{\hat{\gamma}}}) \rceil\}$ times. Moreover, since Proposition \ref{Prop_finite_terminiate_gy_joint_nondeg} demonstrates that every switch to \eqref{Eq_Subroutine_gy} is followed by a return to \eqref{Eq_Subroutine_manifold}, the total number of switches in Algorithm \ref{Alg_TPHSD} after the $K$-th iteration is  bounded by $2 \max\left\{0, \left\lceil \log_2\left(\frac{\varepsilon_0}{\revise{\hat{\gamma}}}\right) \right\rceil \right\} +1$, 
	where $\varepsilon_0$ denotes the initial tolerance in the initialization of Algorithm \ref{Alg_TPHSD}, and $\lceil \cdot \rceil$ the ceiling of a real number.
	
	Notice that $\varepsilon_0$ is a pre-fixed hyper-parameter in Algorithm \ref{Alg_TPHSD}, $\revise{\hat{\gamma}}$  only depends on the optimization problem \eqref{Prob_BLO}, and $K$ only depends on the choices of stepsizes $\{\eta_k\}$ and $\{\theta_k\}$, hence $\varepsilon_0$, $\revise{\hat{\gamma}}$, and $K$ can be regarded as constants in the execution of Algorithm \ref{Alg_TPHSD}. Therefore, the total number of switches back and forth between \eqref{Eq_Subroutine_gy} and \eqref{Eq_Subroutine_manifold} throughout the execution of Algorithm \ref{Alg_TPHSD} is upper-bounded by $2 \max\left\{0,\left\lceil \log_2\left(\frac{\varepsilon_0}{\revise{\hat{\gamma}}}\right) \right\rceil \right\} +K+1$.
\end{rmk}

\section{Numerical Experiments}

This section presents preliminary numerical experiments to evaluate the computational efficiency of our proposed Algorithm \ref{Alg_TPHSD}. We compare its performance against established methods for nonconvex-nonconvex bilevel optimization problems, including BOME \cite{liu2022bome} and MEHA \cite{liu2024moreau}. All experiments were conducted using Python 3.10, with NumPy 1.26.3 and PyTorch 2.5.1, on a workstation equipped with an AMD Ryzen 7 5800H CPU and NVIDIA GeForce RTX 3060 GPU. In our numerical experiments, all the differentials of the UL objective function $f$ and the LL objective function $g$ are computed by the automatic differentiation package in PyTorch.

\subsection{Synthetic numerical experiments}

In this subsection, we consider the following synthetic toy example in \cite{liu2021value},
\begin{equation}   
	\label{Eq_NumExp_synthetic_1}
	\begin{aligned}
		\min_{x \in \bb{R}, y \in \Rn} \quad &  \min\{|x|, 1\} + \frac{1}{n}\sum_{i = 1}^n\norm{y_i - c}^2\\
		\text{s. t.}\quad  ~ & y_i \in \mathop{\arg\min}_{z_i \in \bb{R}} ~ \sin(x + z_i -c), \quad \forall i \in [n]. 
	\end{aligned}
\end{equation}
As demonstrated in \cite{liu2021value}, the global minimizers of \eqref{Eq_NumExp_synthetic_1} are $\{ (0, c + 2\pi v - \frac{\pi}{2}): v \in \bb{Z}^n \}$. For all test instances, the initial points $(x_0, y_0)$ are randomly generated by the built-in function \texttt{torch.randn()} in PyTorch, and all the compared algorithms start at the same initial point. For all the compared algorithms, the stepsizes are chosen from $\{\alpha_1 \times 10^{-\alpha_2}: \alpha_1 = 1,3,5,7,9, \text{ and } \alpha_2 = 2,3,4,5,6\}$ that best reduce the function value of $\{f(\xk, \yk)\}$ in the first $100$ iterations. All the other parameters are fixed as suggested in \cite{liu2022bome,liu2024moreau}. 

Figure \ref{Fig_Test_S1} presents the numerical results on solving \eqref{Eq_NumExp_synthetic_1}. Compared with MEHA and BOME, our proposed Algorithm \ref{Alg_TPHSD} converges faster than BOME and MEHA. Moreover, compared with MEHA, the solution generated by Algorithm \ref{Alg_TPHSD} corresponds to a lower function value of $g$, which illustrates that the solution of Algorithm \ref{Alg_TPHSD} achieves better quality than MEHA. Furthermore, compared to the case where $\alpha = 0$, choosing $\alpha = 0.9$ and $\alpha = 0.99$ accelerates the convergence of Algorithm \ref{Alg_TPHSD} while preserving the similar quality of the final solution. These numerical results demonstrate the efficiency of our proposed Algorithm \ref{Alg_TPHSD}.

\subsection{Hyperparameter selection}
In this subsection, we consider the following hyperparameter selection problem,
\begin{equation}   
	\label{Eq_NumExp_HS}
	\begin{aligned}
		\min_{x \in \Rn, y \in \Rp} \quad &  \frac{1}{|\ca{S}_{\rm val}|}\sum_{s \in \ca{S}_{\rm val}}L_{\rm val}(\tilde{y}, s)  + r(x,\tilde{y})\\
		\text{s. t.}\quad  ~ & \tilde{y} \in \mathop{\arg\min}_{y \in \Rp} ~ \frac{1}{|\ca{S}_{\rm train}|}\sum_{s \in \ca{S}_{\rm train}}L_{\rm train}(y, s) + r(x, y). 
	\end{aligned}
\end{equation}
Here $x \in \Rn$ and $y \in \Rp$ refer to hyperparameters corresponding to certain neural network weights, respectively. Moreover, $\ca{S}_{\rm val}$ and $\ca{S}_{\rm train}$ refer to the validation dataset and training dataset, respectively. In our numerical experiments, we choose $\ca{S}_{\rm val}$ and $\ca{S}_{\rm train}$ from the CIFAR-10 dataset \cite{krizhevsky2009learning}. Furthermore, the function $L_{\rm val}$ refers to the output of LeNet with ReLU activation function, which is nonconvex, nonsmooth, and definable in an $o$-minimal structure, hence satisfies the requirements in Assumption \ref{Assumption_alg_smooth}. Besides, $L_{\rm train}$ refers to the output of LeNet with GELU activations, hence is infinitely differentiable over $\Rp$.  Additionally, the regularization term $r(x, y)$ is chosen as $r(x, y) := \frac{1}{2} \sum_{i = 1}^n x_i^2 y_i^2$. Under these settings, \eqref{Eq_NumExp_HS} satisfies the requirements of Assumption \ref{Assumption_alg_smooth}. In particular, as the lower-level objective function in \eqref{Eq_NumExp_HS} is infinitely differentiable, Assumption \ref{Assumption_joint_nondegeneracy} is satisfied under generic tilt perturbations. 

Figure \ref{Fig_Test_HC1} illustrates the numerical results. From the figures in Figure \ref{Fig_Test_HC1}, we can observe that all the compared algorithms yield solutions with high training and test accuracy. Compared with BOME and MEHA, our proposed Algorithm \ref{Alg_TPHSD} demonstrates high efficiency while yielding solutions with the highest training and test accuracy. 

\begin{figure}[!tbp]
	\caption{Test results on the performance of all the compared methods for synthetic example \eqref{Eq_NumExp_synthetic_1}.}
	\label{Fig_Test_S1}
	\centering
	\subfigure[$n = 10$]{
		\begin{minipage}[t]{0.32\linewidth}
			\centering
			\includegraphics[width=\linewidth,height=2.8cm,keepaspectratio]{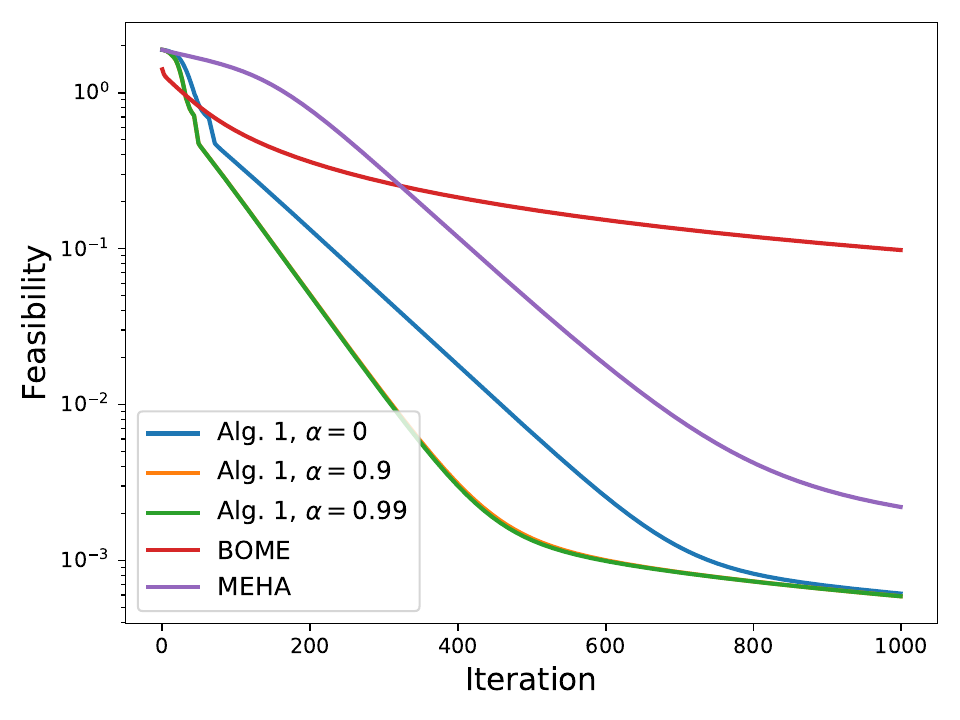}
			\label{Fig:Fig_curve_s1_10_feasLL}
		\end{minipage}%
	}%
	\subfigure[$n = 10$]{
		\begin{minipage}[t]{0.32\linewidth}
			\centering
			\includegraphics[width=\linewidth,height=2.8cm,keepaspectratio]{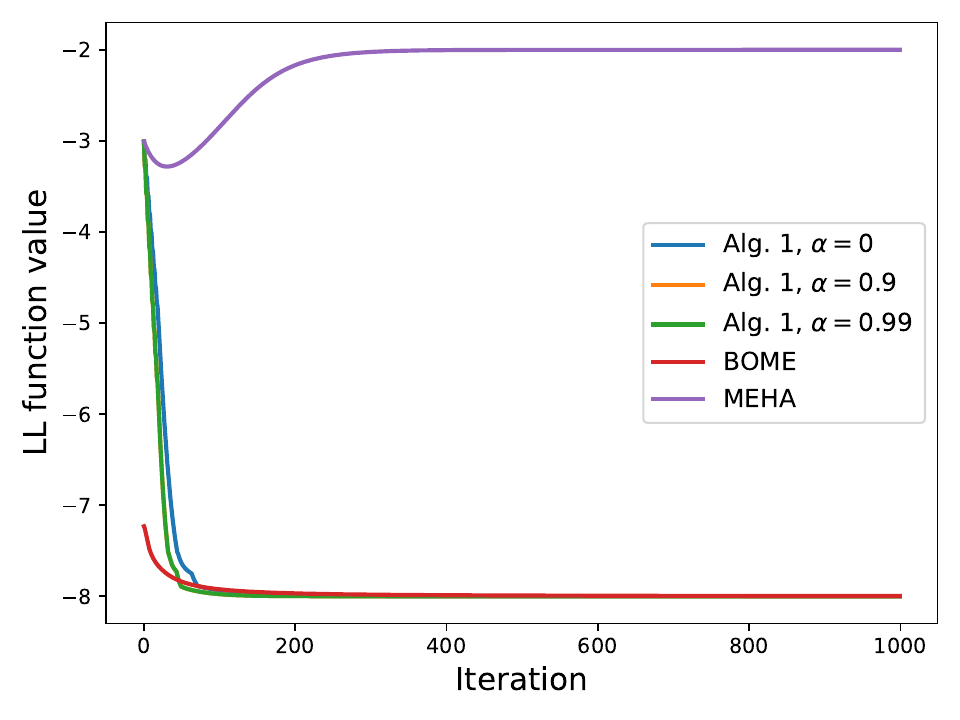}
			\label{Fig:Fig_curve_s1_10_fvalLL}
		\end{minipage}%
	}%
	\subfigure[$n = 10$]{
		\begin{minipage}[t]{0.32\linewidth}
			\centering
			\includegraphics[width=\linewidth,height=2.8cm,keepaspectratio]{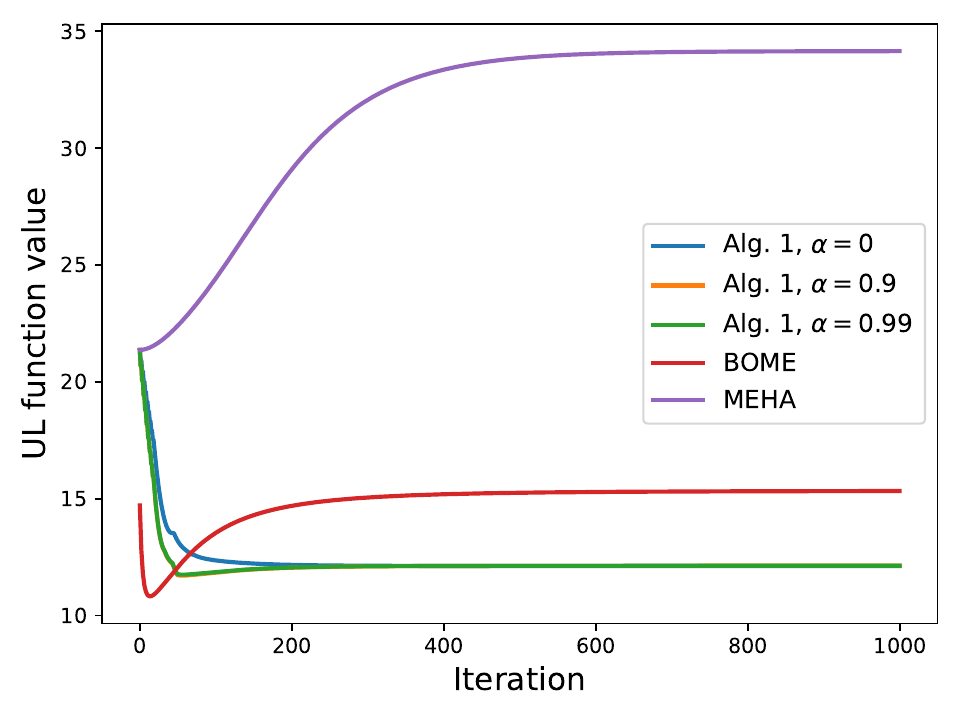}
			\label{Fig:Fig_curve_s1_10_fvalUL}
		\end{minipage}%
	}%
	
	\subfigure[$n = 50$]{
		\begin{minipage}[t]{0.32\linewidth}
			\centering
			\includegraphics[width=\linewidth,height=2.8cm,keepaspectratio]{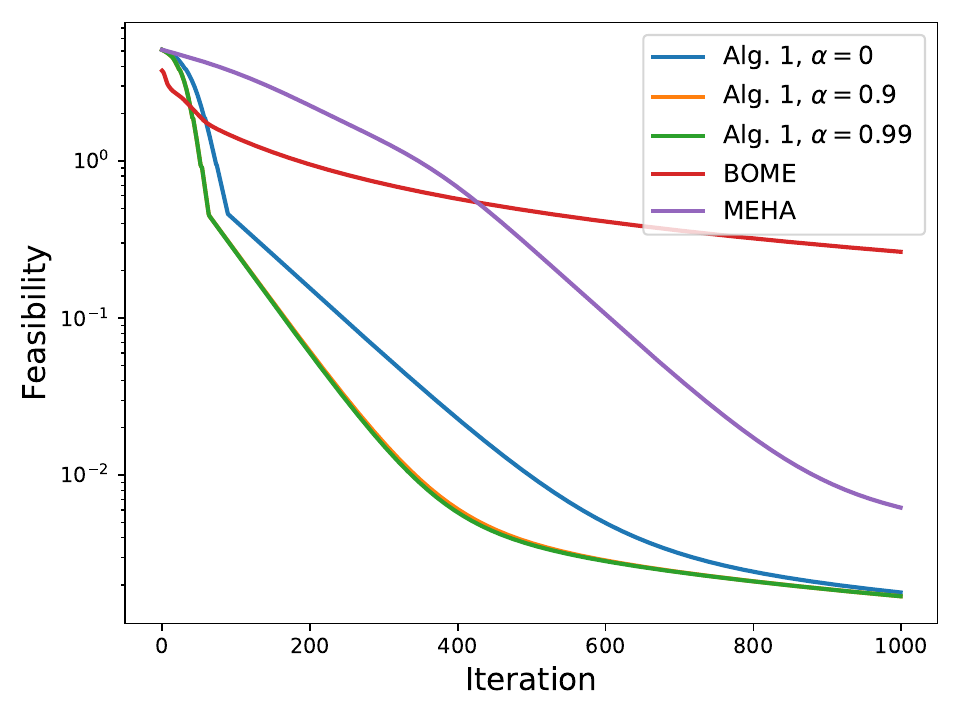}
			\label{Fig:Fig_curve_s1_50_feasLL}
		\end{minipage}%
	}%
	\subfigure[$n = 50$]{
		\begin{minipage}[t]{0.32\linewidth}
			\centering
			\includegraphics[width=\linewidth,height=2.8cm,keepaspectratio]{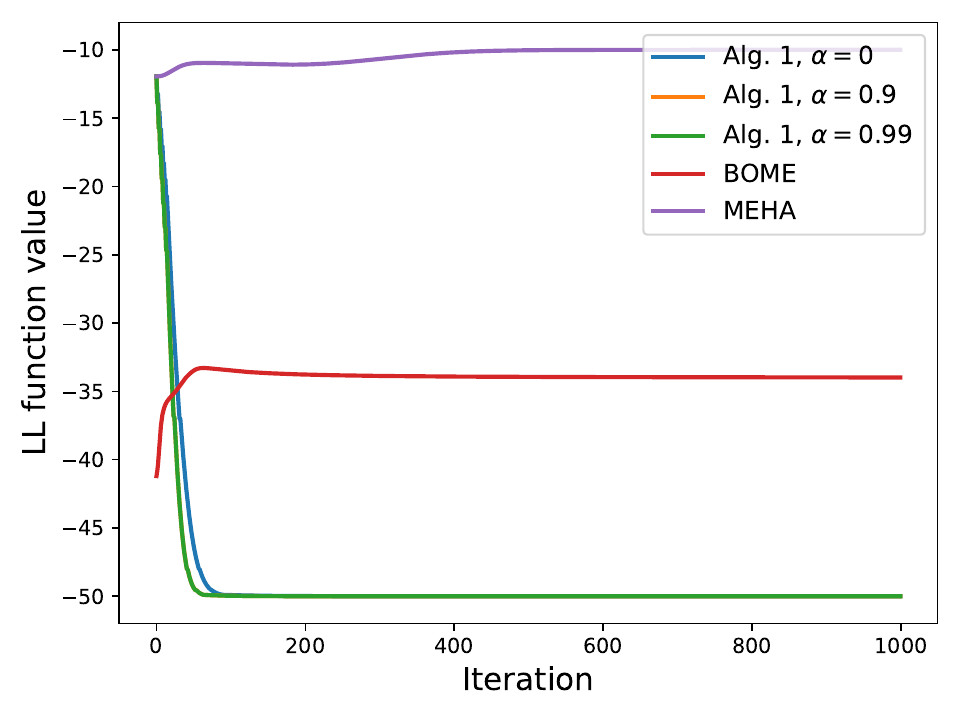}
			\label{Fig:Fig_curve_s1_50_fvalLL}
		\end{minipage}%
	}%
	\subfigure[$n = 50$]{
		\begin{minipage}[t]{0.32\linewidth}
			\centering
			\includegraphics[width=\linewidth,height=2.8cm,keepaspectratio]{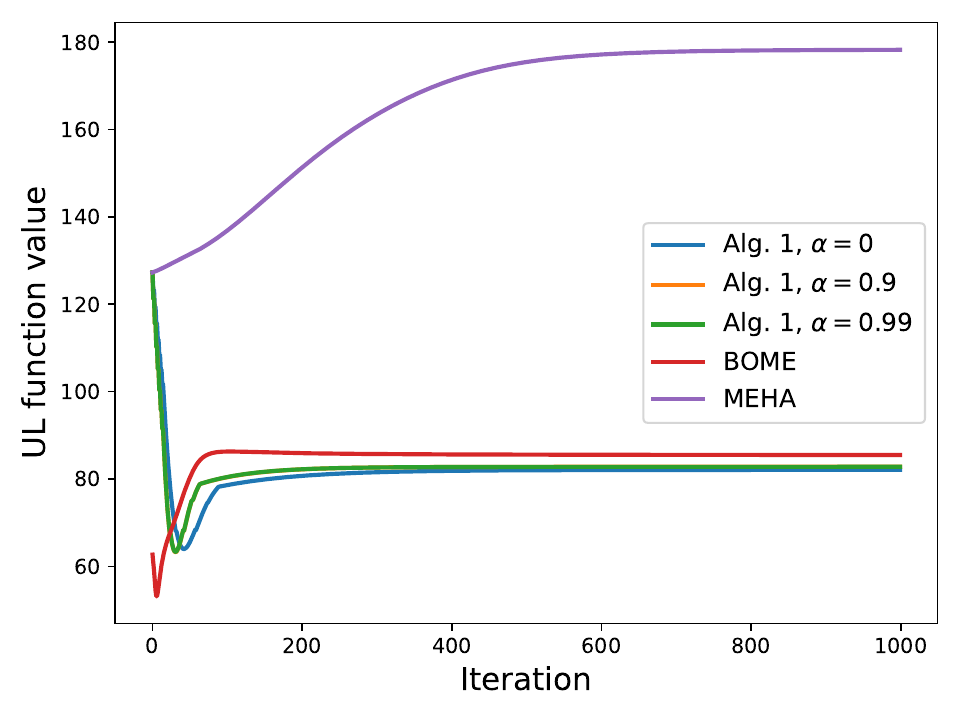}
			\label{Fig:Fig_curve_s1_50_fvalUL}
		\end{minipage}%
	}%
	
	\subfigure[$n = 100$]{
		\begin{minipage}[t]{0.32\linewidth}
			\centering
			\includegraphics[width=\linewidth,height=2.8cm,keepaspectratio]{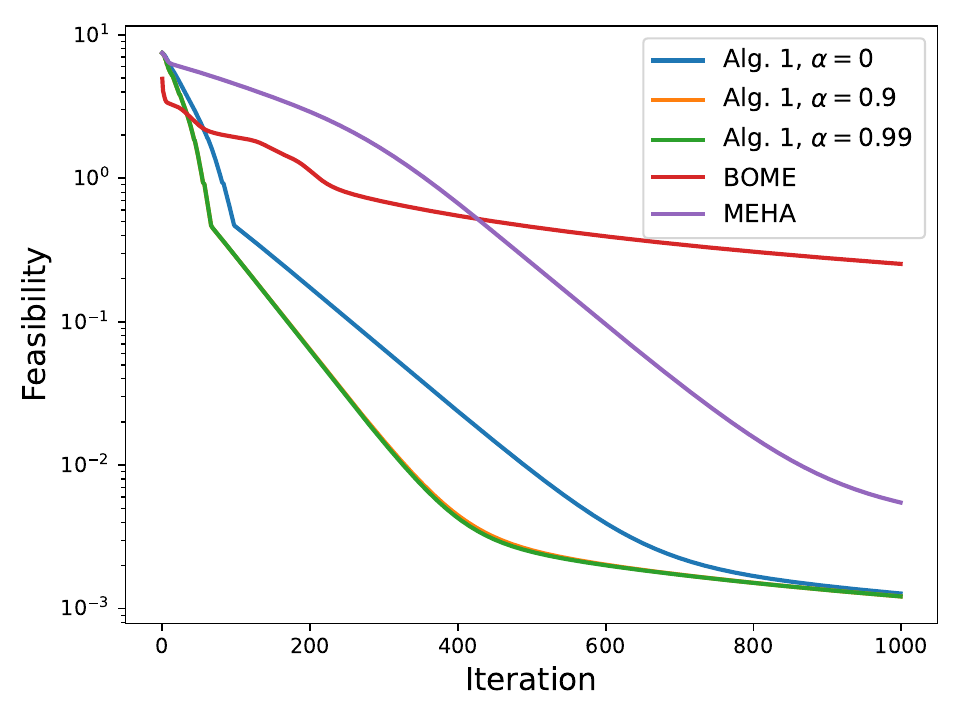}
			\label{Fig:Fig_curve_s1_100_feasLL}
		\end{minipage}%
	}%
	\subfigure[$n = 100$]{
		\begin{minipage}[t]{0.32\linewidth}
			\centering
			\includegraphics[width=\linewidth,height=2.8cm,keepaspectratio]{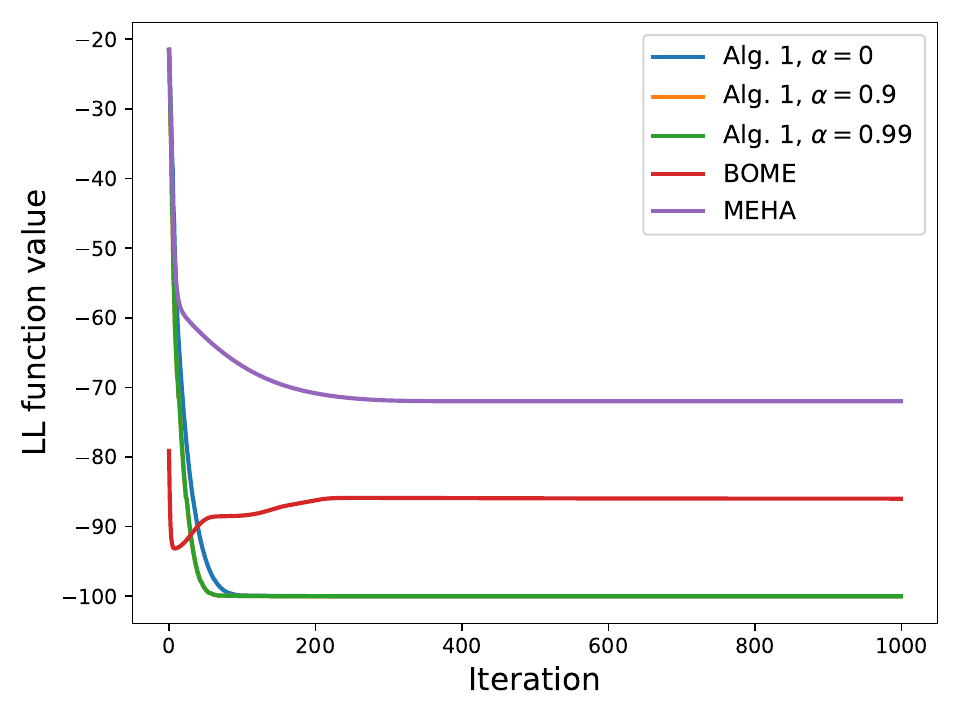}
			\label{Fig:Fig_curve_s1_100_fvalLL}
		\end{minipage}%
	}%
	\subfigure[$n = 100$]{
		\begin{minipage}[t]{0.32\linewidth}
			\centering
			\includegraphics[width=\linewidth,height=2.8cm,keepaspectratio]{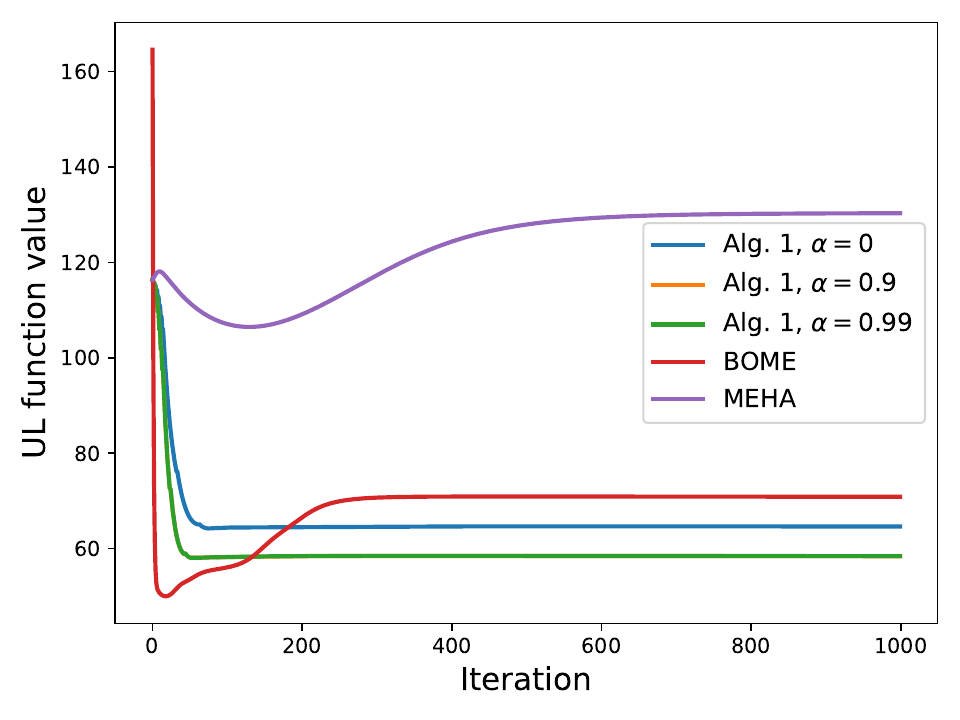}
			\label{Fig:Fig_curve_s1_100_fvalUL}
		\end{minipage}%
	}%
\end{figure}

\begin{figure}[!tbp]
	\caption{Test results on the performance of all the compared methods for the hyperparameter selection problem \eqref{Eq_NumExp_HS}. }
	\label{Fig_Test_HC1}
	\centering
	\subfigure[LL functional values]{
		\begin{minipage}[t]{0.33\linewidth}
			\centering
			\includegraphics[width=\linewidth,height=2.8cm,keepaspectratio]{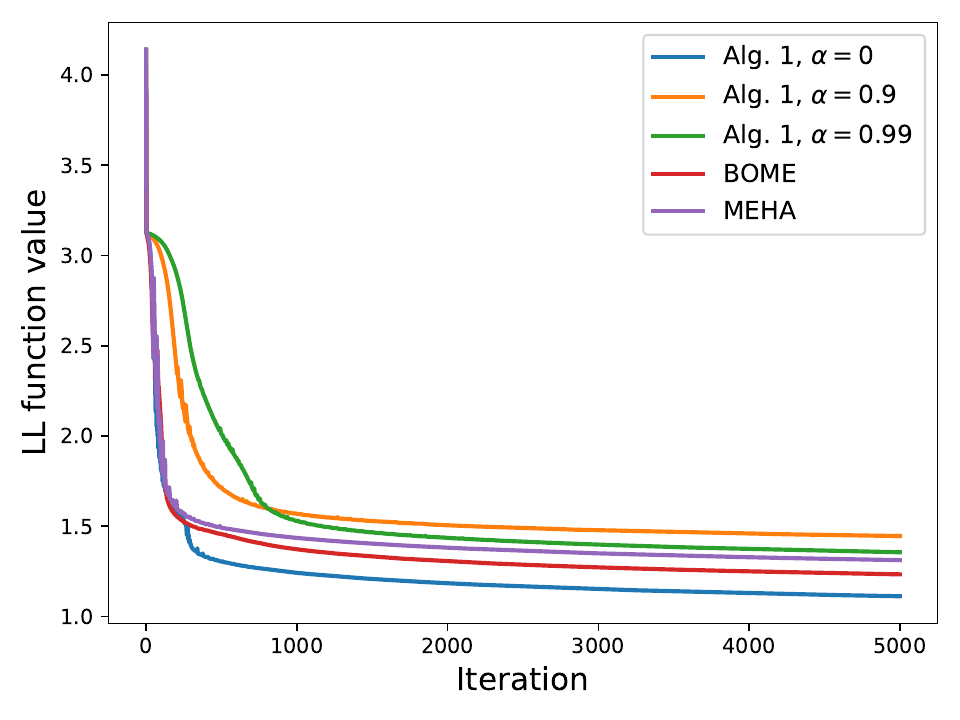}
			\label{Fig:Fig_curve_hc1_0_fvalLL}
		\end{minipage}%
	}%
	\subfigure[UL functional values]{
		\begin{minipage}[t]{0.33\linewidth}
			\centering
			\includegraphics[width=\linewidth,height=2.8cm,keepaspectratio]{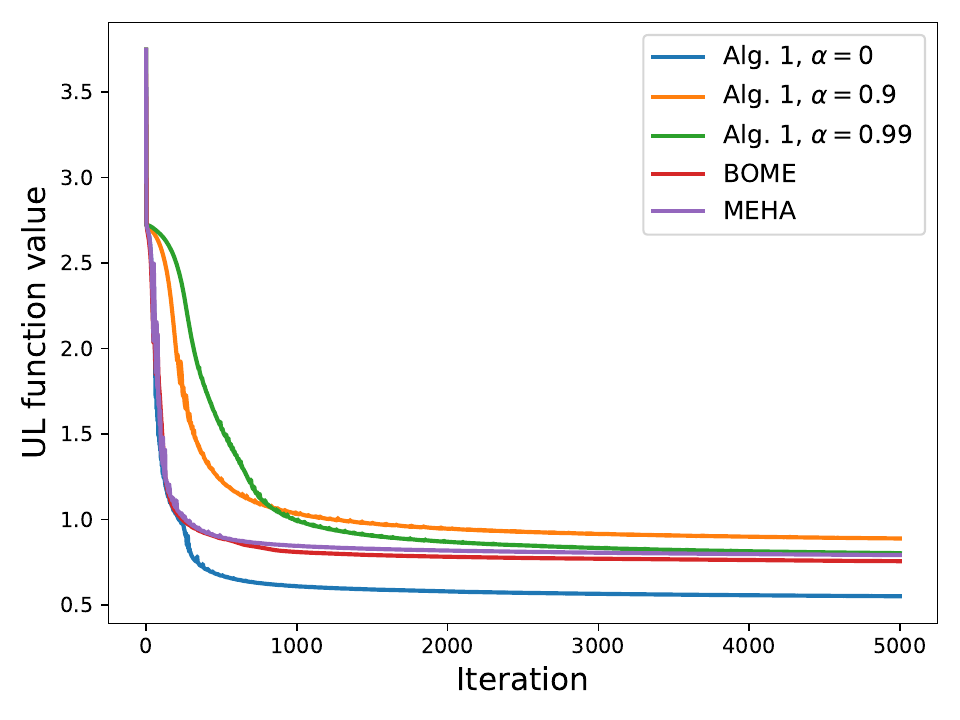}
			\label{Fig:Fig_curve_hc1_0_fvalUL}
		\end{minipage}%
	}%
	
	\subfigure[Train accuracy]{
		\begin{minipage}[t]{0.33\linewidth}
			\centering
			\includegraphics[width=\linewidth,height=2.8cm,keepaspectratio]{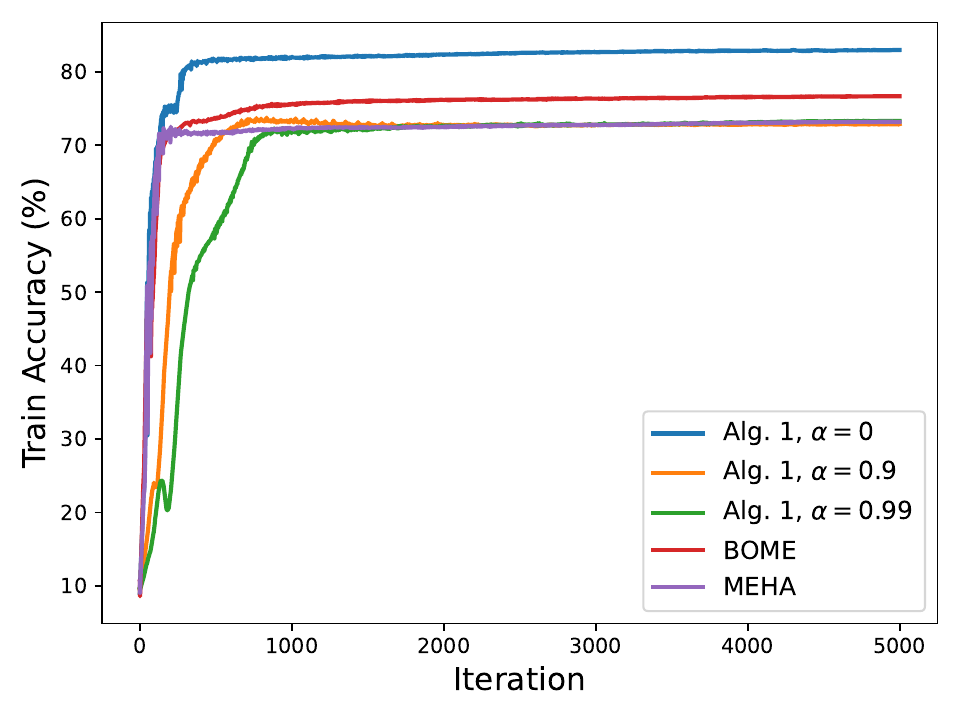}
			\label{Fig:Fig_curve_hc1_0_trainacc}
		\end{minipage}%
	}%
	\subfigure[Test accuracy]{
		\begin{minipage}[t]{0.33\linewidth}
			\centering
			\includegraphics[width=\linewidth,height=2.8cm,keepaspectratio]{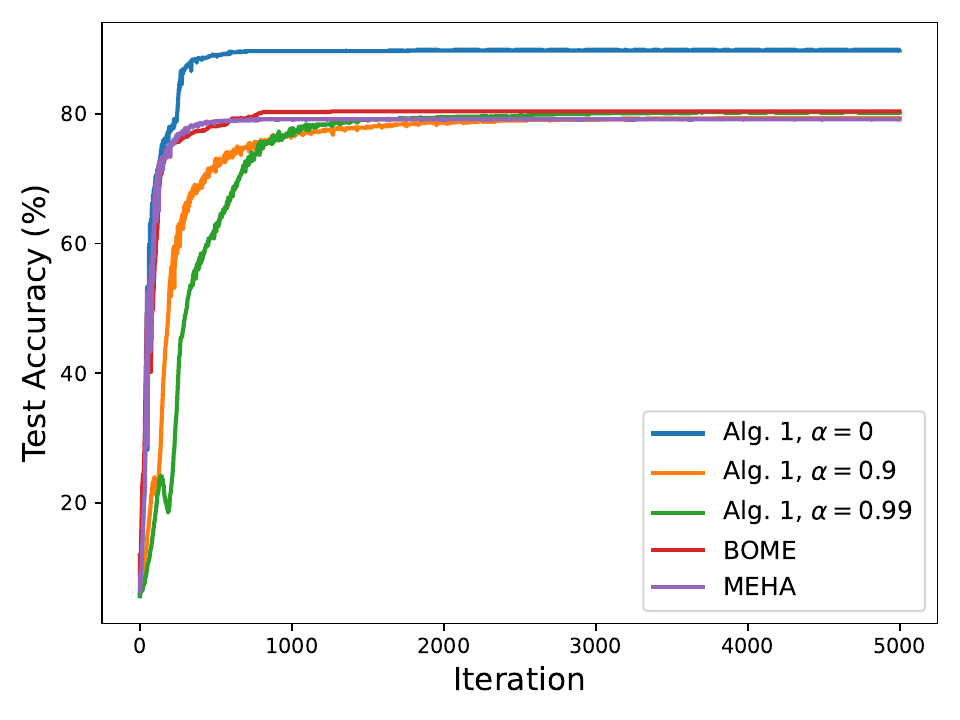}
			\label{Fig:Fig_curve_hc1_0_testacc}
		\end{minipage}%
	}%
\end{figure}

\section{Conclusion}
In this paper, we consider the nonsmooth bilevel optimization problem \eqref{Prob_BLO}, where the LL objective function $g(x, y)$ is a general three times continuously differentiable nonconvex function. By utilizing the optimality condition of the LL subproblem of \eqref{Prob_BLO}, we consider the constrained optimization problem \eqref{Prob_Con}, together with the regularity condition that LICQ holds for \eqref{Prob_Con}. This regularity condition holds under almost every tilt perturbation to the lower-level subproblem \eqref{Eq_BLO_LL}, and hence is mild in practice. 

Then we develop Algorithm \ref{Alg_TPHSD} for solving \eqref{Prob_Con}, which alternates between the two-timescale subgradient method \eqref{Eq_Subroutine_manifold} and the feasibility restoration scheme \eqref{Eq_Subroutine_gy}. Under mild conditions, we prove that the sequence of iterates $\{(\xk, \yk)\}$ generated by Algorithm \ref{Alg_TPHSD} converges towards the feasible region of \eqref{Prob_Con}, and any of its cluster points is a first-order stationary point of \eqref{Prob_Con}. Preliminary numerical experiments illustrate the high efficiency of Algorithm \ref{Alg_TPHSD}, showing that it achieves superior performance in terms of both computational efficiency and solution quality compared to existing methods for nonconvex bilevel optimization. These results demonstrate the promising potential of our proposed Algorithm \ref{Alg_TPHSD} on solving nonsmooth nonconvex bilevel optimization problems. 

The theoretical analysis of this paper focuses on the deterministic setting, where the accumulated error is required to vanish asymptotically. Extending the method to fully stochastic bilevel optimization is an interesting topic for future research. For finite-sum objectives, variance-reduction with periodic full-batch evaluations can provide the variance control needed to satisfy the required perturbation conditions while remaining efficient. How to develop efficient stochastic first-order methods based on Algorithm \ref{Alg_TPHSD} with convergence guarantees is also an interesting topic to investigate.

% \section*{Acknowledgments}
% We would like to acknowledge the assistance of volunteers in putting
% together this example manuscript and supplement.

\bibliographystyle{siamplain}
\bibliography{ref1}

\bigskip
\appendix

\section{Proof for Global Stability}
\label{Appendix_Global_Stability}
In this section, we aim to prove the global stability of (TMG). We first define several constants under Assumption~\ref{Assumption_f}, Assumption~\ref{Assumption_joint_nondegeneracy}, Assumption~\ref{Assumption_alg_smooth}(2)-(3), and any given $M_{iter} > 0$, as follows:
\begin{itemize}
	\item 
	$L_f := \sup \{\norm{d} : d\in \mathcal{D}_f(x,y), 
	\norm{x} + \norm{y} \leq M_{iter} \} 
	$,
	\item $L_g := \sup_{\norm{x} + \norm{y} \leq M_{iter}} \norm{\nabla g(x, y)}$, \quad
	$M_g := \sup_{\norm{x} + \norm{y} \leq M_{iter}} \norm{\revise{\nabla^2 g}(x,y)}$,
	
	\item  $L_p := \sup_{\norm{x} + \norm{y} \leq M_{iter}} \norm{\nabla p(x, y)}$, \quad $M_p := \sup_{\norm{x} + \norm{y} \leq M_{iter}} \norm{\nabla^2 p(x, y)}$, 
	
	% \item $L_A := \sup_{(x,y)\in\Omega} \norm{\nabla \A(x, y)}$, \; $M_A := \revise{\sup_{(x,y)\in\Omega}} \norm{\nabla^2 \A(x, y)}$,
	
	\item $M_{g,3} := \sup_{\norm{x} + \norm{y} \leq M_{iter}} \norm{\nabla^3 g(x, y) }$, 
	
	\item $\sigma := \min\Big\{ 1, ~ \inf \big\{\sigma_{\min}\left(J_g(x,y) \right) : \norm{x} + \norm{y} \leq M_{iter},(x,y)\in \mathcal{M}
	\big\} \Big\}$, 
	\item $\rho:= \min\left\{ \frac{\sigma}{4M_g L_g(4M_e +1)}, \frac{M_g}{8 M_e}, \frac{\sigma M_g}{2M_{g,3} +1}, \frac{\sigma}{64M_e}\right\}$. 
\end{itemize}

Then, for any given $M_{iter} > 0$, we define the attractive region $\Omega$ as 
\begin{equation}
	\Omega := \Big\{(x, y) \in \Rn \times \Rp: \mathrm{dist}\left((x, y), \M \right) \leq \frac{\rho}{M_g},  ~\norm{x} + \norm{y} \leq M_{iter} \Big\},
\end{equation}
which depends on both Assumption~\ref{Assumption_joint_nondegeneracy} and Assumption~\ref{Assumption_alg_smooth}\revise{(2)--(3)}. 
Then we define 
\begin{equation*}
	L_A := \sup_{(x,y)\in\Omega} \norm{\nabla \A(x, y)}, \quad M_A := \sup_{(x,y)\in\Omega} \norm{\nabla^2 \A(x, y)},
\end{equation*}
where the restriction to $\Omega$ guarantees the well-definedness of $\A$.

It is worth mentioning that the choices of the constants $L_f$, $L_g$, $L_p$, $L_A$, $M_g$, $M_p$, $M_A$, $M_{g, 3}$, $\sigma$ and $\rho$ depend on \revise{the arbitrarily fixed localization radius} $M_{\mathrm{iter}}$. When we wish to explicitly emphasize this dependence, we denote these constants by $L_{f, M_{\mathrm{iter}}}$, $L_{g, M_{\mathrm{iter}}}$, for example.
Similarly, we denote the subset $\Omega$ as $\Omega_{M_{\mathrm{iter}}}$ when highlighting its dependence on $M_{\mathrm{iter}}$.

\begin{prop}
	\label{Prop_Appendix_restrict_manifold_joint_nondeg_New}
	Suppose Assumption \ref{Assumption_joint_nondegeneracy}, Assumption \ref{Assumption_f}, and Assumption \ref{Assumption_alg_smooth}\revise{(2)--(3)} hold. \revise{Let $K$ be any positive integer and suppose that $\norm{\xk}+\norm{\yk}\leq M_{\mathrm{iter}}$ for all $0\leq k\leq K$.} For any sequence $\{(\xk, \yk)\}$ generated by \eqref{Eq_Subroutine_manifold} with \revise{$(x_0,y_0)\in\Omega$, $\norm{m_0} \leq L_f$}, and the sequences of stepsizes $\{\eta_k\}$ and $\{\theta_k\}$ \revise{satisfy \eqref{Eq_Cond_Stepsizes}, and}
	\begin{equation*}
		\revise{\frac{2}{\sigma}\max\left\{M_g\mathrm{dist}((x_0,y_0),\M),\frac{18L_f}{\sigma}\sup_{k\geq0}\frac{\eta_k}{\theta_k}\right\}
			+L_f\sup_{k\geq0}\eta_k+(L_p+M_e\rho^2)\sup_{k\geq0}\theta_k\leq\frac{\rho}{M_g},}
	\end{equation*}
	\revise{then, for all $0\leq k\leq K$, it holds that}
	\begin{equation*}
		\begin{aligned}
			&\revise{\norm{\nabla_y g(\xk,\yk)}\leq
				\max\left\{M_g\mathrm{dist}((x_0,y_0),\M),\frac{18L_f}{\sigma}\sup_{k\geq0}\frac{\eta_k}{\theta_k}\right\},}\\
			&\revise{\mathrm{dist}((\xk,\yk),\M)\leq\frac{2}{\sigma}\max\left\{M_g\mathrm{dist}((x_0,y_0),\M),\frac{18L_f}{\sigma}\sup_{k\geq0}\frac{\eta_k}{\theta_k}\right\}.}
		\end{aligned}
	\end{equation*}
\end{prop}
\begin{proof}
	Let $r= \sup_{k\geq 0}\frac{\eta_k}{\theta_k}$ \revise{and $r_0:=\max\left\{M_g\mathrm{dist}((x_0,y_0),\M),\frac{18L_f}{\sigma}r\right\}$, we prove below by induction that the iterates remain in $\Omega$ and that $\norm{\nabla_y g(\xk,\yk)}\leq r_0$.}
	
	Let $d_k = (d_{x, k}, d_{y, k})$, $\mk = (m_{x, k}, m_{y, k})$, $u_k = (u_{x, k}, u_{y, k})$ and $e_k = (e_{x, k}, e_{y, k})$. From the update scheme of $\{\mk\}$ \revise{and the initialization $m_0=d_0$}, it holds for any \revise{$0\leq k\leq K$} that 
	\begin{equation} 
		\norm{\mkp} = \norm{\alpha^{k+1} m_0 +  (1-\alpha)\sum_{i = 0}^k \alpha^{k-i} d_i} \leq \max\left\{\sup_{0\leq i\leq k} \norm{d_i}, \norm{m_0} \right\}\leq L_f. 
	\end{equation}
	Then from the update scheme of \eqref{Eq_Subroutine_manifold}, it holds \revise{for $0\leq k\leq K-1$} that 
	\begin{small}
		\begin{equation}
			\label{Eq_Appendix_Prop_restrict_manifold_joint_nondeg_0}
			\begin{aligned}
				&p(\xkp, \ykp) - p(\xk, \yk) \\
				\leq{}& \inner{\nabla p(\xk, \yk), -\eta_k \mkp - \theta_k  (u_k + e_k)} + \frac{M_p}{2} \norm{\eta_k \mkp + \theta_k  (u_k + e_k)}^2\\
				% \leq{}& \inner{\nabla p(\xk, \yk), -\eta_k \mkp - \theta_k  (u_k + e_k)} + M_p \eta_k^2 \norm{\mkp}^2 + M_p \theta_k^2 \norm{u_k + e_k}^2  \\
				\leq{}& \inner{\nabla p(\xk, \yk), -\eta_k \mkp - \theta_k  (u_k + e_k)} + M_p \eta_k^2 \norm{\mkp}^2 + 2M_p \theta_k^2 \norm{u_k}^2 + 2M_p \theta_k^2 \norm{e_k}^2
				\\
				\leq{}& \inner{\nabla p(\xk, \yk),  - \theta_k  u_k} + \left(M_p \eta_k^2 + \frac{16\eta_k^2}{\theta_k}\right) \norm{\mkp}^2 \\
				&+ \left(2M_p \theta_k^2 +16 \theta_k\right)\norm{e_k}^2 + \frac{\theta_k}{8} \norm{\nabla p(\xk, \yk)}^2
				\\
				\leq{}&- \frac{7\theta_k}{8} \norm{\nabla p(\xk, \yk)}^2 
				+\frac{18\eta_k^2}{ \theta_k}  L_f^2+ 18\theta_k M_e^2 \norm{\nabla_y g(\xk, \yk)}^4 
				\\
				\leq{}& - \frac{7\sigma^2 \theta_k}{16} p(\xk, \yk) + \frac{18\eta_k^2}{ \theta_k}  L_f^2 + 72 \theta_k M_e^2 p(\xk, \yk)^2 \leq -\frac{\sigma^2\theta_k}{4} p(\xk, \yk) + \frac{18\eta_k^2}{\theta_k}  L_f^2. 
			\end{aligned}
		\end{equation}
	\end{small}
	Here the third inequality uses the following inequalities, 
	\begin{equation*}
		\begin{aligned}
			&\inner{\nabla p(\xk, \yk),  - \eta_k \mkp} \leq \frac{16\eta_k^2}{\theta_k} \norm{\mkp}^2 + \frac{\theta_k}{16} \norm{\nabla p(\xk, \yk)}^2,\\
			&\inner{\nabla p(\xk, \yk),  -  \theta_k e_k} \leq \frac{\theta_k}{16} \norm{\nabla p(\xk, \yk)}^2 + 16 \theta_k \norm{e_k}^2. 
		\end{aligned}
	\end{equation*}
	In addition, the fourth inequality follows from Assumption \ref{Assumption_alg_smooth}\revise{(2)--(3)} and the fact that $\theta_k \leq \frac{1}{M_p}$, and the fifth inequality follows from the fact that 
	\begin{equation*}
		\norm{\nabla p(x, y)}^2 = \norm{J_g(x, y) \nabla_y g(x, y)}^2 \geq \frac{\sigma^2}{4} \norm{\nabla_y g(x, y)}^2 = \frac{\sigma^2}{2} p(x, y),
	\end{equation*}
	for any $(x, y) \in \Omegajoint$. The last inequality follows from the choice of $\Omega$, which leads to 
	\begin{equation}
		\label{Eq_Appendix_Prop_restrict_manifold_joint_nondeg_3}
		\norm{\nabla_y g(\xk, \yk)} \leq M_g \mathrm{dist}((\xk, \yk), \M) \leq \rho \leq \frac{\sigma}{16 M_e}, 
	\end{equation}
	and  it implies that $72  M_e^2 p(\xk, \yk) \leq \frac{3\sigma^2}{16} $.
	
	For any $(\xk, \yk)$ such that $\norm{\nabla_y g(\xk, \yk)} \geq  \frac{12 L_f}{\sigma} %\sup_{k\geq 0}\frac{\eta_k}{\theta_k}
	r$, it follows from \eqref{Eq_Appendix_Prop_restrict_manifold_joint_nondeg_0} that 
	\begin{equation}
		\label{Eq_Appendix_Prop_restrict_manifold_joint_nondeg_1}
		\norm{\nabla_y g(\xkp, \ykp)} \leq \norm{\nabla_y g(\xk, \yk)}.
	\end{equation}
	On the other hand, for any $(\xk, \yk)$ such that $\norm{\nabla_y g(\xk, \yk)} \leq \frac{12 L_f}{\sigma} %\sup_{k\geq 0}\frac{\eta_k}{\theta_k}
	r$, it holds  that 
	\begin{equation*}
		\mathrm{dist}((\xk, \yk), \M) \leq \frac{2}{\sigma} \norm{\nabla_y g(\xk, \yk)} \leq \frac{24 L_f}{\sigma^2}
		%\sup_{k\geq 0}\frac{\eta_k}{\theta_k}
		r.
	\end{equation*}
	Furthermore, from \eqref{Eq_Appendix_Prop_restrict_manifold_joint_nondeg_3} we have that $\norm{\nabla_y g(\xk, \yk)}^2 \leq \rho^2$. As a result, it follows from the update scheme \eqref{Eq_Subroutine_manifold} and Assumption \ref{Assumption_alg_smooth}(3) that  
	\begin{equation*}
		\label{Eq_Appendix_Prop_restrict_manifold_joint_nondeg_2}
		\begin{aligned}
			&\norm{(\xkp, \ykp) - (\xk, \yk)} \leq  \norm{\eta_k \mkp + \theta_k (u_k + e_k)} 
			\\
			\leq{}& \norm{\mkp}\eta_k + (\norm{u_k} + \norm{e_k})\theta_k 
			\;\leq{}\; L_f \eta_k  + \left(L_p + M_e\rho^2 \right) \theta_k. 
		\end{aligned}
	\end{equation*}
	Therefore, from the Lipschitz continuity of $\nabla g$, it holds that 
	\begin{equation*}
		\begin{aligned}
			&\norm{\nabla_y g(\xkp, \ykp)} \leq \norm{\nabla_y g(\xk, \yk)} + M_g \left( L_f \eta_k  + \left(L_p + M_e\rho^2 \right) \theta_k  \right) 
			\\
			\leq{}& \norm{\nabla_y g(\xk, \yk)} + \left(M_g L_f \sup_{k\geq 0} \theta_k + M_g (L_p + M_e\rho^2) \sup_{k\geq 0} \frac{\theta_k^2}{\eta_k} \right) r \\
			\leq{}& \norm{\nabla_y g(\xk, \yk)} + \frac{6 L_f}{\sigma}
			%\sup_{k\geq 0}\frac{\eta_k}{\theta_k}\ 
			r \leq  \frac{18 L_f}{\sigma}
			%\sup_{k\geq 0}\frac{\eta_k}{\theta_k}
			r. 
		\end{aligned}
	\end{equation*}
	
	\revise{We now prove the claimed estimate by induction. Since $(x_0,y_0)\in\Omega$, we have}
	\begin{equation*}
		\revise{\norm{\nabla_y g(x_0,y_0)}\leq M_g\mathrm{dist}((x_0,y_0),\M)\leq r_0.}
	\end{equation*}
	\revise{Suppose that $(\xk,\yk)\in\Omega$ and $\norm{\nabla_y g(\xk,\yk)}\leq r_0$ for some $0\leq k\leq K-1$, we have}
	\begin{equation*}
		\begin{aligned}
			&\mathrm{dist}((\xkp,\ykp),\M)\leq \mathrm{dist}((\xk,\yk),\M)+\norm{(\xkp,\ykp)-(\xk,\yk)}\\
			\leq{}& \frac{2r_0}{\sigma}+L_f\sup_{i\geq0}\eta_i+(L_p+M_e\rho^2)\sup_{i\geq0}\theta_i \leq\frac{\rho}{M_g}.
		\end{aligned}
	\end{equation*}
	Here the second inequality uses Lemma \ref{Le_welldef_penalty_joint_nondeg} and the fact that $\norm{(\xkp,\ykp)-(\xk,\yk)} \leq L_f\eta_k+(L_p+M_e\rho^2)\theta_k$.  The third inequality follows from the induction hypothesis $\norm{\nabla_y g(\xk,\yk)}\leq r_0$.
	\revise{Together with the fact that $\norm{x_{k+1}}+\norm{y_{k+1}}\leq M_{\mathrm{iter}}$, we can conclude that $(\xkp,\ykp)\in\Omega$. Furthermore, when $\norm{\nabla_y g(\xk,\yk)}\geq 12L_fr/\sigma$, then \eqref{Eq_Appendix_Prop_restrict_manifold_joint_nondeg_1} gives $\norm{\nabla_y g(\xkp,\ykp)}\leq r_0$. Otherwise, we have that $\norm{\nabla_y g(\xkp,\ykp)}\leq18L_fr/\sigma\leq r_0$. This completes the induction, and shows that for all $0\leq k\leq K$, we have}
	\begin{equation*}
		\revise{\mathrm{dist}((\xk,\yk),\M)\leq\frac{2}{\sigma}\norm{\nabla_y g(\xk,\yk)}\leq\frac{2r_0}{\sigma}.}
	\end{equation*}
	This completes the proof.  
\end{proof}

The following lemma illustrates that the accumulation of $\{\theta_k \norm{\nabla_y g(\xk, \yk)}^2\}$ can be controlled by the stepsizes $\{\eta_k\}$. 
\begin{lem}
	\label{Le_Appendix_accumulation_gy_joint_nondeg_New}
	\revise{Suppose Assumption \ref{Assumption_joint_nondegeneracy}, Assumption \ref{Assumption_f}, and Assumption \ref{Assumption_alg_smooth}(2)--(3) hold. Let $K$ be any positive integer. For any sequence $\{(\xk,\yk)\}$ generated by \eqref{Eq_Subroutine_manifold} with $(x_0,y_0)\in\Omega$, suppose that $\norm{\xk}+\norm{\yk}\leq M_{\mathrm{iter}}$ for all $0\leq k\leq K$, and that the sequences of stepsizes $\{\eta_k\}$ and $\{\theta_k\}$ satisfy \eqref{Eq_Cond_Stepsizes} and}
	\begin{equation*}
		\begin{aligned}
			\revise{\frac{2}{\sigma}\max\left\{M_g\mathrm{dist}((x_0,y_0),\M),\frac{18L_f}{\sigma}\sup_{k\geq0}\frac{\eta_k}{\theta_k}\right\} +L_f\sup_{k\geq0}\eta_k+(L_p+M_e\rho^2)\sup_{k\geq0}\theta_k\leq\frac{\rho}{M_g}.}
		\end{aligned}
	\end{equation*}
	\revise{Then, for any $T>0$, it holds that}
	\begin{equation*}
		\begin{aligned}
			&\revise{\sup_{\substack{0\leq s\leq j\leq K-1, ~j\leq\Lambda(\lambda(s)+T)}}
				\sum_{k=s}^{j}\theta_k\norm{\nabla_y g(\xk,\yk)}^2}\\
			&\revise{\leq \frac{4}{\sigma^2}\max\left\{M_g^2\mathrm{dist}((x_0,y_0),\M)^2,
				\frac{324L_f^2}{\sigma^2}\sup_{k\geq0}\frac{\eta_k^2}{\theta_k^2}\right\}
				+\frac{144L_f^2}{\sigma^2}\left(T+\sup_{k\geq0}\eta_k\right)
				\sup_{k\geq0}\frac{\eta_k}{\theta_k}.}
		\end{aligned}
	\end{equation*}
\end{lem}
\begin{proof}
	From the update scheme \eqref{Eq_Subroutine_manifold} and the inequality \revise{\eqref{Eq_Appendix_Prop_restrict_manifold_joint_nondeg_0}}, it holds for any \revise{$0\leq k\leq K-1$} that 
	\begin{equation*}
		p(\xkp, \ykp) - p(\xk, \yk) \leq  -\frac{\sigma^2\theta_k}{8} \norm{\nabla_y g(\xk, \yk)}^2 + \frac{18\eta_k^2}{\theta_k}  L_f^2.
	\end{equation*}
	Then we can estimate the upper bound for $\theta_k \norm{\nabla_y g(\xk, \yk)}^2$ by 
	\begin{equation}
		\label{Eq_Appendix_Le_accumulation_gy_joint_nondeg_0}
		\theta_k \norm{\nabla_y g(\xk, \yk)}^2 \leq \frac{8}{\sigma^2} \left( p(\xk, \yk) - p(\xkp, \ykp) \right) + \frac{\revise{144} \eta_k^2L_f^2}{\theta_k\sigma^2}. 
	\end{equation}
	\revise{Therefore, for any $0\leq s\leq j\leq K-1$ satisfying $j\leq\Lambda(\lambda(s)+T)$, it holds that}
	\begin{equation*}
		\begin{aligned}
			&\revise{\sum_{k=s}^{j}\theta_k\norm{\nabla_y g(\xk,\yk)}^2
				\leq \frac{8}{\sigma^2}\left(p(x_s,y_s)-p(x_{j+1},y_{j+1})\right)
				+\frac{144L_f^2}{\sigma^2}\sum_{k=s}^{j}\frac{\eta_k^2}{\theta_k}}\\
			&\revise{\leq \frac{4}{\sigma^2}\max\left\{M_g^2\mathrm{dist}((x_0,y_0),\M)^2,
				\frac{324L_f^2}{\sigma^2}\sup_{k\geq0}\frac{\eta_k^2}{\theta_k^2}\right\}
				+\frac{144L_f^2}{\sigma^2}\left(T+\sup_{k\geq0}\eta_k\right)
				\sup_{k\geq0}\frac{\eta_k}{\theta_k}.}
		\end{aligned}
	\end{equation*}
	\revise{Here the last inequality uses Proposition \ref{Prop_Appendix_restrict_manifold_joint_nondeg_New}, $p\geq0$, and $\sum_{k=s}^{j}\frac{\eta_k^2}{\theta_k}\leq \left(T+\sup_{k\geq0}\eta_k\right)\sup_{k\geq0}\frac{\eta_k}{\theta_k}$.}
	This completes the proof. 
\end{proof}

In the following proposition, we show that with appropriate initialization and stepsizes, the error terms in \eqref{Eq_update_auxiliary} can be controlled by the stepsizes $\{\eta_k\}$. 
\begin{prop}
	\label{Prop_Appendix_local_convergence_noise_controll}
	\revise{Suppose Assumption \ref{Assumption_joint_nondegeneracy}, Assumption \ref{Assumption_f}, and Assumption \ref{Assumption_alg_smooth}(2)--(3) hold. Let $K$ be any positive integer. For any sequence $\{(\xk,\yk)\}$ generated by \eqref{Eq_Subroutine_manifold} with $(x_0,y_0)\in\Omega$, suppose that $\norm{\xk}+\norm{\yk}\leq M_{\mathrm{iter}}$ for all $0\leq k\leq K$, and that the sequences of stepsizes $\{\eta_k\}$ and $\{\theta_k\}$ satisfy \eqref{Eq_Cond_Stepsizes} and}
	\begin{equation*}
		\begin{aligned}
			&\revise{\frac{2}{\sigma}\max\left\{M_g\mathrm{dist}((x_0,y_0),\M),\frac{18L_f}{\sigma}\sup_{k\geq0}\frac{\eta_k}{\theta_k}\right\} +L_f\sup_{k\geq0}\eta_k+(L_p+M_e\rho^2)\sup_{k\geq0}\theta_k\leq\frac{\rho}{M_g}.}
		\end{aligned}
	\end{equation*}
	\revise{Then for any $T>0$, there exists $M_{T,M_{\mathrm{iter}}}>0$, independent of $K$, such that}
	\begin{equation*}
		\begin{aligned}
			&\revise{\sup_{\substack{0\leq s\leq j\leq K-1, ~j\leq\Lambda(\lambda(s)+T)}}
				\norm{\sum_{k=s}^{j}\eta_k\left(\xi_{k+1}-\beta\nabla p(\xk,\yk)\right)}}\\
			&\revise{\leq M_{T,M_{\mathrm{iter}}}\left(\mathrm{dist}((x_0,y_0),\M)
				+\sup_{k\geq0}\max\left\{\eta_k,\frac{\eta_k}{\theta_k},\frac{\theta_k^2}{\eta_k}\right\}\right).}
		\end{aligned}
	\end{equation*}
\end{prop}
\begin{proof}
	\revise{For any fixed $T>0$, we choose $M_{T,M_{\mathrm{iter}}}>0$ sufficiently large so that the estimates below hold. This constant only depends on $T$, $M_{\mathrm{iter}}$, $\beta$, and the local constants defined above, and is independent of $K$.}
	From Lemma \ref{Le_welldef_penalty_joint_nondeg} and the differentiability of $\nabla_y g$, we can conclude that $\A$ is differentiable and $\nabla \A$ is locally Lipschitz continuous over $\Omega$. Together with  \eqref{Eq_Cond_Stepsizes}, it holds for any $0\leq k\leq K-1$ that 
	\begin{equation*}
		\begin{aligned}
			&\norm{r_k} \leq 3M_{A} \left(\norm{u_k}^2 + \norm{e_k}^2 + \frac{\eta_k^2}{\theta_k^2} \norm{\revise{\mkp}}^2 \right) \leq  \revise{3M_A\left(L_p^2+M_e^2\rho^4+
				\frac{\rho^2\sigma^4}{128^2(M_g+1)^2}\right).}
		\end{aligned}
	\end{equation*}
	Then with the definition of $P_{\A}$ and $R_{\A}$ in \eqref{Eq_defin_PA_RA} and Lemma \ref{Le_error_bound_cdf}, it holds that 
	\begin{equation*}
		\begin{aligned}
			&\norm{\nabla \A(\xk, \yk) u_k } = \norm{P_{\A}(\xk, \yk) u_k + R_{\A}(\xk, \yk)u_k} \\
			={}& \norm{ R_{\A}(\xk, \yk)u_k} \leq \frac{8M_g\revise{M_{g,3}}}{\sigma^2}  \norm{\nabla_y g(\xk, \yk)}^2. 
		\end{aligned}
	\end{equation*}
	Therefore, for any \revise{$0\leq k\leq K-1$}, it holds that 
	\begin{equation*}
		\eta_k \norm{\xi_{k+1}} \leq \theta_k \left(\frac{8M_g\revise{M_{g,3}}}{\sigma^2} + L_A M_e\right)\norm{\nabla_y g(\xk, \yk)}^2 + \theta_k^2  \norm{r_k}. 
	\end{equation*}
	\revise{For any $0\leq s\leq j\leq K-1$ satisfying $j\leq\Lambda(\lambda(s)+T)$, it holds that}
	\revise{\begin{equation}
			\label{Eq_Prop_Appendix_local_convergence_noise_controll_0_New}
			\begin{aligned}
				&\norm{\sum_{k=s}^{j}\eta_k\xi_{k+1}}\\
				\leq{}& \left(\frac{8M_gM_{g,3}}{\sigma^2}+L_AM_e\right)
				\sum_{k=s}^{j}\theta_k\norm{\nabla_y g(\xk,\yk)}^2 +3M_A\left(L_p^2+M_e^2\rho^4+
				\frac{\rho^2\sigma^4}{128^2(M_g+1)^2}\right)\sum_{k=s}^{j}\theta_k^2\\
				\leq{}& \left(\frac{8M_gM_{g,3}}{\sigma^2}+L_AM_e\right)\Bigg( \frac{4}{\sigma^2}\max\left\{M_g^2\mathrm{dist}((x_0,y_0),\M)^2,
				\frac{324L_f^2}{\sigma^2}\sup_{k\geq0}\frac{\eta_k^2}{\theta_k^2}\right\} +\frac{144L_f^2}{\sigma^2}\left(T+\sup_{k\geq0}\eta_k\right)
				\sup_{k\geq0}\frac{\eta_k}{\theta_k}\Bigg)\\
				&\quad+3M_A\left(L_p^2+M_e^2\rho^4+
				\frac{\rho^2\sigma^4}{128^2(M_g+1)^2}\right)
				\left(T+\sup_{k\geq0}\eta_k\right)\sup_{k\geq0}\frac{\theta_k^2}{\eta_k}.
			\end{aligned}
	\end{equation}}
	\revise{Here the last inequality follows from Lemma \ref{Le_Appendix_accumulation_gy_joint_nondeg_New} and the fact that $\sum_{k=s}^{j}\theta_k^2\leq \left(T+\sup_{k\geq0}\eta_k\right)\sup_{k\geq0}\frac{\theta_k^2}{\eta_k}$.}
	
	\revise{For any $0\leq s\leq j\leq K-1$ satisfying $j\leq\Lambda(\lambda(s)+T)$ and any $\beta>0$, \mbox{Proposition~\ref{Prop_Appendix_restrict_manifold_joint_nondeg_New}} illustrates that }
	\begin{equation}
		\label{Eq_Prop_Appendix_local_convergence_noise_controll_1_New}
		\revise{\norm{\sum_{k=s}^{j}\eta_k\beta\nabla p(\xk,\yk)} \leq \beta M_g\left(T+\sup_{k\geq0}\eta_k\right)
			\max\left\{M_g\mathrm{dist}((x_0,y_0),\M),
			\frac{18L_f}{\sigma}\sup_{k\geq0}\frac{\eta_k}{\theta_k}\right\}.}
	\end{equation}

	\revise{By combining \eqref{Eq_Prop_Appendix_local_convergence_noise_controll_0_New} and \eqref{Eq_Prop_Appendix_local_convergence_noise_controll_1_New}, it follows from the facts $(x_0,y_0)\in\Omega$ and \eqref{Eq_Cond_Stepsizes} that
		\begin{equation*}
			\begin{aligned}
				&\sup_{\substack{0\leq s\leq j\leq K-1\\j\leq\Lambda(\lambda(s)+T)}}
				\norm{ \sum_{k=s}^{j}\eta_k\left(\xi_{k+1}-\beta\nabla p(\xk,\yk)\right)}\\
				\leq{}& M_{T,M_{\mathrm{iter}}}\left(\mathrm{dist}((x_0,y_0),\M)
				+\sup_{k\geq0}\max\left\{\eta_k,\frac{\eta_k}{\theta_k},
				\frac{\theta_k^2}{\eta_k}\right\}\right).
			\end{aligned} 
		\end{equation*}
	}
	This completes the proof. 
\end{proof}

Then we aim to establish the global stability for the sequence of iterates $\{(\xk, \yk)\}$ generated by \eqref{Eq_Subroutine_manifold}. Specifically, we demonstrate that the sequence remains uniformly bounded \revise{under Assumption \ref{Assumption_global_stability}}, when the stepsizes $\{\eta_k\}$ and $\{\theta_k\}$ are sufficiently small.  

We begin our analysis with the following proposition, which illustrates that the term $\nabla \A(\xk, \yk)m_{k+1} + \beta \nabla p(\xk, \yk)$ can be viewed as an
approximated evaluation of $\D_{h_{\beta}}(\wk, \zk)$. 
\begin{prop}
	\label{Prop_local_convergence_Dh_esti_New}
	Suppose Assumption \ref{Assumption_joint_nondegeneracy}, Assumption \ref{Assumption_f}, and Assumption \ref{Assumption_alg_smooth}(2)-(3) hold. For any \revise{fixed $\beta>0$ and} $\delta > 0$, there exists $\alpha_{\text{step}} > 0$ such that \revise{for any positive integer $K$,} any sequences of stepsizes $\{\eta_k\}$ and $\{\theta_k\}$ satisfying
	\begin{equation*}
		\sup_{k\geq0}\max\left\{\eta_k,\theta_k,\frac{\eta_k}{\theta_k},
		\frac{\theta_k^2}{\eta_k}\right\}\leq\alpha_{\text{step}},
	\end{equation*}
	and any sequence $\{(\xk,\yk)\}$ generated by \eqref{Eq_Subroutine_manifold} \revise{satisfying}
	\begin{equation*}
		\revise{\norm{\xk}+\norm{\yk}\leq M_{\mathrm{iter}}\quad(0\leq k\leq K),
			\qquad \mathrm{dist}((x_0,y_0),\M)\leq\alpha_{\text{step}},}
	\end{equation*}
	it holds \revise{for all $0\leq k\leq K$ that $\mathrm{dist}((\xk,\yk),\M)\leq\delta$ and}
	\begin{equation*}
		\nabla \A(\xk, \yk) \mkp + \beta \nabla p(\xk, \yk) \in \D_{h_{\beta}}^{\delta}(\wk, \zk). 
	\end{equation*}
\end{prop}
\begin{proof}
	For any sequences of stepsizes $\{\eta_k\}$ and $\{\theta_k\}$ that satisfy \revise{the upper bound in the statement, if $\alpha_{\mathrm{step}}$ is sufficiently small, then \eqref{Eq_Cond_Stepsizes} and all the conditions of Proposition \ref{Prop_Appendix_restrict_manifold_joint_nondeg_New} hold. Therefore, Proposition \ref{Prop_Appendix_restrict_manifold_joint_nondeg_New} illustrates that, for all $0\leq k\leq K$,}
	\begin{equation*}
		\revise{\mathrm{dist}((\xk,\yk),\M)\leq\frac{2}{\sigma}
			\max\left\{M_g\alpha_{\mathrm{step}},
			\frac{18L_f}{\sigma}\alpha_{\mathrm{step}}\right\}\leq\delta.}
	\end{equation*}
	
	\revise{It remains to prove the claimed inclusion. Suppose for each sufficiently large $j$, with $\alpha_{\mathrm{step}}=1/j$ and an index $k_j$ such that}
	\begin{equation*}
		\revise{\nabla\A(x_{k_j}^j,y_{k_j}^j)m_{k_j+1}^j +\beta\nabla p(x_{k_j}^j,y_{k_j}^j) \notin\D_{h_\beta}^{\delta}(w_{k_j}^j,z_{k_j}^j).}
	\end{equation*}
	\revise{By the uniform boundedness of $\{(\xk, \yk)\}$, Proposition \ref{Prop_Appendix_restrict_manifold_joint_nondeg_New} and the continuity of $\A$ guarantee that $(x_{k_j}^j,y_{k_j}^j)\to(\bar x,\bar y)\in\M$ and $(w_{k_j}^j,z_{k_j}^j)\to(\bar x,\bar y)$.}
	\revise{For any fixed $C\in\bb{N}_+$, the update scheme of the $\{\mk\}$ implies that
    \begin{equation*}
        \lim_{j \to +\infty} \max_{\max\{0,k_j-C\}\leq i\leq k_j}
        \norm{(x_i^j,y_i^j)-(x_{k_j}^j,y_{k_j}^j)}\leq \lim_{j \to +\infty} C\left(L_f\sup_{k\geq0}\eta_k^j
        +(L_p+M_e\rho^2)\sup_{k\geq0}\theta_k^j\right) =0.
    \end{equation*}
		Moreover, we have $m_{k_j+1}^j=\alpha^{k_j+1}d_0^j +(1-\alpha)\sum_{i=0}^{k_j}\alpha^{k_j-i}d_i^j$, $\norm{m_{k_j+1}^j}\leq L_f$, and $\mathrm{dist}\left(m_{k_j+1}^j, \conv\left(\bigcup_{i=\max\{0,k_j-C\}}^{k_j} \D_f(x_i^j,y_i^j)\right)\right)\leq2L_f\alpha^C$.}
	\revise{For every fixed $C\in\bb{N}_+$, based on Carath\'eodory's theorem, together with the graph-closedness, local boundedness, and convex-valuedness of $\D_f$, we have that 
		\begin{equation*}
			\begin{aligned}
				&\lim_{j \to +\infty}\sup_{d\in\conv\left(\bigcup_{i=\max\{0,k_j-C\}}^{k_j}
					\D_f(x_i^j,y_i^j)\right)}
				\mathrm{dist}\left(d,\D_f(\bar x,\bar y)\right) = 0,\\
				&\limsup_{j\to+\infty}\mathrm{dist}\left(
				m_{k_j+1}^j,\D_f(\bar x,\bar y)\right)
				\leq2L_f\alpha^C.
			\end{aligned}
	\end{equation*}}
	\revise{Since $\alpha\in[0,1)$, let $C\to+\infty$, we have $\lim_{j \to +\infty}\mathrm{dist}\left(m_{k_j+1}^j,\D_f(\bar x,\bar y)\right) = 0$. }
	
	Notice that
	\begin{equation*}
		\D_{h_{\beta}}(\wk, \zk) = \nabla \A(\wk,\zk) \D_f(\A(\wk, \zk)) + \beta \nabla p(\wk, \zk)\revise{.}
	\end{equation*}
	\revise{Since $\A(\bar x,\bar y)=(\bar x,\bar y)$ and $\nabla p(\bar x,\bar y)=0$,} we have
	\begin{equation*}
		\begin{aligned}
			&\lim_{j \to +\infty}\mathrm{dist}\left(
			\nabla\A(x_{k_j}^j,y_{k_j}^j)m_{k_j+1}^j
			+\beta\nabla p(x_{k_j}^j,y_{k_j}^j),
			\D_{h_\beta}(\bar x,\bar y)\right)\\
			\leq{}&
			\lim_{j \to +\infty}L_f\norm{\nabla\A(x_{k_j}^j,y_{k_j}^j)-\nabla\A(\bar x,\bar y)} +\norm{\nabla\A(\bar x,\bar y)}
			\mathrm{dist}\left(m_{k_j+1}^j,\D_f(\bar x,\bar y)\right)
			+\beta\norm{\nabla p(x_{k_j}^j,y_{k_j}^j)}\\
			={}& 0.
		\end{aligned}
	\end{equation*}
	\revise{Together with the fact that $(w_{k_j}^j,z_{k_j}^j)\to(\bar x,\bar y)$, it follows from the definition of $\D_{h_\beta}^{\delta}$ that, $\nabla\A(x_{k_j}^j,y_{k_j}^j)m_{k_j+1}^j +\beta\nabla p(x_{k_j}^j,y_{k_j}^j) \in\D_{h_\beta}^{\delta}(w_{k_j}^j,z_{k_j}^j)$ holds for all sufficiently large $j$, which is a contradiction.}
	
	Therefore, for any $\delta > 0$, we can choose $\alpha_{\mathrm{step}}$ \revise{sufficiently small} to ensure that $\nabla \A(\xk, \yk) \mkp + \beta \nabla p(\xk, \yk) \in \D_{h_{\beta}}^{\delta}(\wk, \zk)$ holds for any \revise{$0\leq k\leq K$.} This completes the proof. 
\end{proof}

\revise{Then we present the basic assumptions that are necessary for the global stability properties of \eqref{Eq_Subroutine_manifold} as follows.}
\begin{assumpt}
	\label{Assumption_global_stability}
	\revise{For any $r\in\bb{R}$ and any $M_{iter} > 0$, there exists an open neighborhood $\ca{U}$ of $\M$ on which $\A$ is well-defined and $\hat{\beta} > 0$, such that for any $\beta>\hat{\beta}$, the following conditions hold,}
	\begin{enumerate}
		\item $f(x, y)$ is coercive over $\Rn \times \Rp$. That is, for any $r \geq 0$, the subset $\{(x, y) \in \Rn \times \Rp: f(x, y) \leq r\}$ is a compact subset of $\Rn \times \Rp$. 
		\item \revise{The set $\{f(x,y):(x,y)\in\ca{U},\text{ $(x, y)$ is a first-order stationary point of \eqref{Prob_Con}}\}$ is finite.}
	\end{enumerate}
\end{assumpt}
\revise{It is worth mentioning that Assumption \ref{Assumption_global_stability}(1) assumes the coercivity of the objective function $f$, which is a common assumption in existing works on the global stability of subgradient methods \cite{josz2023lyapunov,le2023nonsmooth,xiao2023globalstability}. Additionally, it is easy to verify that Assumption \ref{Assumption_global_stability}(1) guarantees the coercivity of $f$ over $\M$. Moreover, Assumption \ref{Assumption_global_stability}(2) refers to the nonsmooth Morse-Sard properties of \eqref{Prob_Con}, which can be guaranteed by the definability of $f$ and $g$ \cite{davis2020stochastic,bolte2021conservative}.} Then the following theorem demonstrates the global stability for the iterates $\{(\xk, \yk)\}$ generated by \eqref{Eq_Subroutine_manifold}.

\begin{prop}
	\label{Prop_levelwise_compactness_cdf}
	\revise{Suppose Assumption \ref{Assumption_joint_nondegeneracy}, Assumption \ref{Assumption_f}(1) and Assumption \ref{Assumption_global_stability}(1) hold. Then there exists an open neighborhood $\ca{U}$ of $\M$  such that $\A$ is well-defined on $\ca{U}$. Moreover,  for any $r\in\bb{R}$, there exist  constants $\hat{\beta}_r>0$ and $R_r>0$  such that for any $\beta>\hat{\beta}_r$, the set $\{(x,y)\in\ca{U}:h_\beta(x,y)\leq r\}$ is a compact subset of $\Rn \times \Rp$.}
\end{prop}
\begin{proof}
	\revise{Fix any $r\in\bb{R}$. Since $f$ is coercive on $\Rn \times \Rp$, there exists $R_r>0$ such that $f(x,y)>r$ for any $(x,y)\in\Rn \times \Rp$ satisfying $\norm{x}+\norm{y}\geq R_r$. Define}
	\begin{equation*}
		\revise{\ca{U}:=\left\{(x,y):
		\sigma_{\min}(J_g(x,y))>0,~ \norm{\A(x,y)-(x,y)}<\frac{1}{2}, ~\mathrm{dist}(\A(x, y), \M) < \frac{1}{2}\right\}.}
	\end{equation*}
	\revise{Then it is easy to verify that $\ca{U}$ is an open neighborhood of $\M$ on which $\A$ is well-defined.}

\revise{For any $\beta>0$ and any $(x,y)\in\{(x,y)\in\ca{U}:h_\beta(x,y)\leq r\}$, the nonnegativity of the penalty term implies that $f(\A(x,y))\leq r$. By the definition of $\ca{U}$, we further have $\|x\|+\|y\|<R_r+1$ and $\|\A(x,y)\|<R_r$.}

\revise{Let $\underline{f}_r:=\min_{\|(u,v)\|\leq R_r+1}f(u,v)$, which is finite due to the continuity of $f$ and the compactness of the closed ball. If $\{(x,y):\norm{x}+ \norm{y}\leq R_r+1\}\subseteq\ca{U}$, we set $\hat\beta_r:=1$. Otherwise, the nonempty set $\{(x,y): \norm{x}+ \norm{y}\leq R_r+1, (x,y)\notin \ca{U}\}$ is compact and disjoint from $\M$. In this case, we define}
\begin{equation*}
\revise{
\hat{\beta}_r := 1+\frac{2\max\{r-\underline{f}_r,0\}}
{\displaystyle
\min_{\norm{x}+ \norm{y} \leq R_r+1,(x,y)\notin\ca{U}} \|\nabla_y g(x,y)\|^2}.
}
\end{equation*}

 \revise{Fix any $\beta>\hat\beta_r$. For any $(x,y)\in\mathbb{R}^n\times\mathbb{R}^p$ satisfying $h_{\beta}(x,y)\leq r$, the nonnegativity of the penalty term implies that $f(\A(x,y))\leq r$, and hence $\norm{x}+ \norm{y}<R_r+1$. Moreover, by the definition of $\hat\beta_r$, we have
\begin{equation*}
\frac{1}{2}\|\nabla_y g(x,y)\|^2
\leq \frac{r-f(\A(x,y))}{\beta}
<
\min_{\substack{\norm{x}+ \norm{y}\leq R_r+1,~(x,y)\notin\ca{U}}}
\frac{1}{2}\|\nabla_y g(x,y)\|^2 .
\end{equation*}
Therefore, $(x,y)\notin\ca{U}$ cannot hold, and hence $(x,y)\in\ca{U}$ and $\|x\|+\|y\|<R_r+1$. Consequently, $\{(x,y)\in\ca{U}:h_{\beta}(x,y)\leq r\}$ is a closed subset of the compact set $\{(x,y):\|x\|+\|y\|\leq R_r+1\}$, and therefore is compact. This completes the proof.}
\end{proof}

\begin{theo}
	\label{Theo_local_convergence_NS_joint_nondeg_New}
	Suppose Assumption \ref{Assumption_joint_nondegeneracy}, Assumption \ref{Assumption_f}, and Assumption \ref{Assumption_alg_smooth}(2)-(3) hold\revise{, and Assumption \ref{Assumption_global_stability} holds}. Then for any $M > 0$, there exists  $\alpha_{\mathrm{step}} > 0$, such that for any $(x_0, y_0) \in \{(x, y): \norm{x} +\norm{y} \leq M, \revise{\mathrm{dist}((x,y), \M)} \leq \alpha_{\mathrm{step}}\}$, any sequences of stepsizes $\{\eta_k\}$ and $\{\theta_k\}$ satisfying $\sup_{k\geq 0}\max\{\eta_k, \theta_k, \frac{\eta_k}{\theta_k}, \frac{\theta_k^2}{\eta_k}\} \leq \alpha_{\text{step}}$, and any sequence $\{(\xk, \yk)\}$ generated by \eqref{Eq_Subroutine_manifold}, it holds that the iterates $\{(\xk, \yk)\}$ are uniformly bounded. 
\end{theo}
\begin{proof}
	For any $M>0$, \revise{since $\A$ is well-defined and continuously differentiable on $\ca{U}$, the results in Lemmas \ref{Le_conservative_field_joint_nondeg} and \ref{Le_DI_cdf} can be directly applied on $\ca{U}$. Hence, $h_\beta$ is a locally Lipschitz Lyapunov function for the differential inclusion $-\D_{h_\beta}$. Moreover, Assumption \ref{Assumption_f}(1), together with the continuous differentiability of $\A$ and the definition of $\D_{h_\beta}$, imply that $\D_{h_\beta}$ is convex-valued, graph-closed, and locally bounded on $\ca{U}$.}
	
	\revise{Since $\{(x,y)\in\M:\norm{x}+\norm{y}\leq M+1\}$ is compact, $\ca{U}$ is open, and $\A$ is identity over $\M$, we can choose $\alpha_{\mathrm{step}}>0$ sufficiently small such that the set}
	\begin{equation*}
		\revise{\left\{\A(x,y):\norm{x}+\norm{y}\leq M,\ 
			\mathrm{dist}((x,y),\M)\leq\alpha_{\mathrm{step}}\right\}}
	\end{equation*}
	\revise{is a compact subset of $\ca{U}$.}
	\revise{By Assumption \ref{Assumption_global_stability}, Proposition \ref{Prop_levelwise_compactness_cdf}, and the graph-closedness of $\D_{h_\beta}$, the set $\{(x,y)\in\ca{U}:0\in\D_{h_\beta}(x,y)\}$ is also compact. Thus, we can choose $\tilde{r}>0$ such that}
	\begin{equation*}
		\begin{aligned}
			\revise{\tilde{r}>4\max\Bigg\{1,}
			&\revise{\sup_{\substack{\norm{x}+\norm{y}\leq M,
						\mathrm{dist}((x,y),\M)\leq\alpha_{\mathrm{step}}}}
				h_\beta(\A(x,y)), \sup_{\substack{(x,y)\in\ca{U},0\in\D_{h_\beta}(x,y)}}
				h_\beta(x,y)\Bigg\}.}
		\end{aligned}
	\end{equation*}

	\revise{Then we aim to  apply \cite[Proposition 3.6]{xiao2023globalstability} with $h_\beta$ as the objective function and $\D_{h_\beta}$ as the conservative field on the compact set $\left\{(x,y)\in\ca{U}:h_\beta(x,y)\leq\frac{5\tilde r}{4}\right\}$. More precisely, there exist constants $\delta, T > 0$ such that, for any $K$, if $(w_0,z_0)=\A(x_0,y_0)$, $\norm{x_0}+\norm{y_0}\leq M$, and $\mathrm{dist}((x_0,y_0),\M)\leq \alpha_{\mathrm{step}} $, and if $(\wk,\zk)\in\ca{U}$ and $\eta_k\leq\delta$ for all $0\leq k\leq K$, and $$(\wkp,\zkp)\in(\wk,\zk)-\eta_k\left(\D_{h_\beta}^{\delta}(\wk,\zk)+\xi_{k+1}-\beta\nabla p(\xk,\yk)\right)$$ with  $\sup_{0\leq s\leq j\leq\Lambda(\lambda(s)+T), ~j\leq K }\norm{\sum_{k=s}^{j}\eta_k\left(\xi_{k+1}-\beta\nabla p(\xk,\yk)\right)} \leq \delta$. Under these conditions,  \cite[Proposition 3.6]{xiao2023globalstability} guarantees that
	\begin{equation*}
		\revise{(\wk,\zk)\in\ca{U},\quad \text{and} \quad h_\beta(\wk,\zk)\leq\tilde r}
	\end{equation*}
    for any $0\leq k\leq K+1$.
    }
	
	Thus, we can choose a constant $\hat{M}_{\mathrm{iter}}\revise{\geq M+1}$ such that $\norm{x}+\norm{y}\leq\hat{M}_{\mathrm{iter}}$ for any $(x,y)\revise{\in\ca{U}}$ satisfying $h_\beta(x,y)\leq\tilde{r}$. 
	Now, we define $M_{\mathrm{iter}}=\hat{M}_{\mathrm{iter}}+1$.
	For this specific $M_{\mathrm{iter}}$, \revise{we can choose a sufficiently small $\alpha_{\mathrm{step}}\leq\frac{\delta}{2M_{T,M_{\mathrm{iter}}}+1}$ such that all the conditions in Proposition \ref{Prop_Appendix_restrict_manifold_joint_nondeg_New}, Lemma \ref{Le_Appendix_accumulation_gy_joint_nondeg_New}, and Propositions \ref{Prop_Appendix_local_convergence_noise_controll}--\ref{Prop_local_convergence_Dh_esti_New} are satisfied.} 
	\revise{For $k=0$, the update scheme, the definition of $\A$, and $(x_0,y_0)\in\Omega$ illustrates that, for any positive integer $k$, whenever $\norm{x_i}+\norm{y_i}\leq M_{\mathrm{iter}}$ holds for any $0\leq i\leq k$, we have $\norm{x_i}+\norm{y_i}\leq M_{\mathrm{iter}}$. Moreover, Proposition \ref{Prop_Appendix_restrict_manifold_joint_nondeg_New} guarantees that $\norm{x_{k+1}-x_k}+\norm{y_{k+1}-y_k}\leq\frac{1}{4}$ and $\norm{x_k-w_k}+\norm{y_k-z_k}\leq\frac{1}{4}$. 
    }

    Now we aim to use the induction to prove that $(\wk,\zk)\in\ca{U}$, $h_\beta(\wk,\zk)\leq\tilde r$, and $\norm{\xk}+\norm{\yk}\leq\hat{M}_{\mathrm{iter}}+\frac{1}{4}$ for any $k\geq 0$. Firstly, the choice of $\alpha_{\mathrm{step}}$ guarantees that $(w_0,z_0)\in\ca{U}$, $h_\beta(w_0,z_0)<\frac{\tilde r}{4}$, and $\norm{x_0}+\norm{y_0}\leq M\leq\hat{M}_{\mathrm{iter}}+\frac{1}{4}$. 

    Then suppose the assertion holds for any $0\leq k\leq K$, we have 
    \begin{equation*}
		\begin{aligned}
			&\revise{\norm{x_{K+1}}+\norm{y_{K+1}} \leq\norm{x_K}+\norm{y_K}
				+\norm{x_{K+1}-x_K}+\norm{y_{K+1}-y_K}}\\
			\leq{}&\revise{\hat{M}_{\mathrm{iter}}+\frac12<M_{\mathrm{iter}}.}
		\end{aligned}
	\end{equation*}
	\revise{Therefore, by Proposition \ref{Prop_local_convergence_Dh_esti_New} and \eqref{Eq_update_auxiliary}, for $0\leq k\leq K$, we have}
	\begin{equation*}
		(\wkp, \zkp) \in (\wk, \zk) - \eta_k \D_{h_{\beta}}^{\delta}(\wk, \zk) -\eta_k \left(\xi_{k+1} -  \beta \nabla p(\xk, \yk)\right),
	\end{equation*}
	and \revise{Proposition \ref{Prop_Appendix_local_convergence_noise_controll} illustrates that}
	\begin{equation*}
		\revise{\sup_{\substack{0\leq s\leq j\leq K, ~j\leq\Lambda(\lambda(s)+T)}}
			\norm{\sum_{k=s}^{j}\eta_k\left(\xi_{k+1}-\beta\nabla p(\xk,\yk)\right)}
			\leq2M_{T,M_{\mathrm{iter}}}\alpha_{\mathrm{step}}\leq\delta.}
	\end{equation*}
	
	Then it follows from \revise{\cite[Proposition 3.6]{xiao2023globalstability}} that $(w_{K+1},z_{K+1})\in\ca{U}$ and $h_\beta(w_{K+1},z_{K+1})\leq\tilde r$. 
	\revise{By the choice of $\hat{M}_{\mathrm{iter}}$, it holds that $\norm{w_{K+1}}+\norm{z_{K+1}}\leq\hat{M}_{\mathrm{iter}}$. }
	As a result, we have 
	\begin{equation*}
		\revise{\norm{x_{K+1}}+\norm{y_{K+1}}
			\leq\norm{w_{K+1}}+\norm{z_{K+1}}+\frac14
			\leq\hat{M}_{\mathrm{iter}}+\frac{1}{4}.}
	\end{equation*}
	\revise{This completes the induction, and thus completes the proof.}
\end{proof}

\end{document}